\documentclass[12pt]{article}
\usepackage{amsxtra,amssymb,amsmath,enumerate,graphicx,bbm,amsthm,color,mathrsfs,tabu}
\usepackage[active]{srcltx}
\usepackage{t1enc}
\usepackage{aurical}
\usepackage[T1]{fontenc}
\usepackage{comment}
\usepackage[left=2.5cm, right=2.5cm, top=2cm, bottom=4cm]{geometry}
\usepackage[urlcolor=black,linkcolor=black,citecolor=black,colorlinks=true]{hyperref}

\newtheorem{lemma}{Lemma}
\newtheorem{proposition}{Proposition}
\newtheorem{theorem}{Theorem}

\newtheorem{remark}{Remark}
\newtheorem{definition}{Definition}
\newtheorem{example}{Example}

\newtheorem{Assumption}{Assumption}

\numberwithin{lemma}{section}
\numberwithin{proposition}{section}
\numberwithin{theorem}{section}
\numberwithin{corollary}{section}
\numberwithin{remark}{section}
\numberwithin{definition}{section}
\numberwithin{example}{section}
\numberwithin{problem}{section}
\numberwithin{application}{section}
\numberwithin{Assumption}{section}
\numberwithin{equation}{section}
\numberwithin{figure}{section}

\DeclareMathOperator*{\argmax}{arg\,max}

\newcommand{\e}{\mathbb{E}}

\def\be{\begin{align}}
\def\ee{\end{align}}
\def\b*{\begin{eqnarray*}}
\def\e*{\end{eqnarray*}}

\def\qr{{\rm q}}

\def\vr{{\rm v}}

\def\blue#1{\textcolor[RGB]{41, 128, 185}{ #1}}
\def\bru#1{\blue{#1}}
\begin{document}

\title{Optimal control under uncertainty: Application to the issue of CAT bonds}

\author{Nicolas Baradel\footnote{Ecole Polytechnique Paris, Centre de Mathématiques Appliquées, {nicolas.baradel@polytechnique.edu}.  This work
benefits from the financial support of the Chairs {\it Financial Risk} and {\it Finance and Sustainable Development}. The author takes the opportunity to express his gratitude to Bruno Bouchard for fruitful discussions.}
\ }


\maketitle

\begin{abstract}

We propose a general framework for studying optimal issue of CAT bonds in the presence of uncertainty on the parameters. In particular, the intensity of arrival of  natural disasters is inhomogeneous and may depend on unknown parameters. Given a prior on the distribution of the unknown parameters, we explain how it should evolve according to the classical Bayes rule. Taking these progressive prior-adjustments into account, we characterize the optimal policy through a quasi-variational parabolic equation, which can be solved numerically. We provide examples of application in the context of hurricanes in Florida.

\end{abstract}
\section{Introduction}

We consider an insurer or a reinsurer who holds a portfolio in non-life insurance exposed to one or several natural disasters. He can issue one or several CAT bonds\footnote{Catastrophe bonds, or CAT bonds, are tradable floating rate notes. The risk associated with a CAT bond is not linked to the default of one entity (state or corporate) but is related to the occurrence of a catastrophe.} in order to reduce the risk taken, see e.g. \cite{cummins2008cat} or \cite{cummins2012cat} for a general introduction to CAT bonds.

The first CAT bonds where issued at the end of the 1990s and the market is globally increasing, with a total risk capital outstanding greater that USD 30 billion at the end of 2017, see \cite{artemis2017} and \cite{guy2015}. CAT bonds give a strong alternative to the classical reinsurance market.

However, issuing a CAT bond leads to the choice of several parameters, as the layer e.g. and the date of issuance. The coupon is not a priori perfectly known as well as the claim distribution. Moreover, the global warming will lead to an increase of several natural disasters which is a source of uncertainty on the distribution of future claims. For example, in \cite{oouchi2006tropical}, the authors estimate that if the temperature rises of 2.5 degrees in the next decades, the frequency of Hurricanes in North Atlantic will rise by 30\%.

The aim of this paper is to provide a rigorous continuous-time framework in which we can establish the optimal behavior policy in issuing CAT bonds, taking into account the uncertainty described above as the risk evolution.

The coupon of the CAT bond is generally not known in advance, even its distribution is not always clearly fixed. We therefore need to model it as a random variable whose  distribution  depends on unknown parameters. It is the same for the distribution of the natural disasters.

The particular case of acting on a system with partially unknown response distributions has been studied in \cite{baradel2016optimal} in a Brownian framework, and widely for the case of discrete settings; see, e.g., \cite{easley1988controlling, hernandez2012adaptive} for references. They fix a prior distribution on the unknown parameter and introduce a stochastic process on the space of measures which leads to a dynamic programming principle and a PDE characterization of the value function (in the viscosity solution sense). We will adapt \cite{baradel2016optimal} in a new context, dealing with CAT bonds.

In this paper, the natural disasters will be represented by a random Poisson measure\footnote{The activity of the random Poisson measure will be finite, by construction} and  two parameters are unknown: the distribution of the severity of the natural disasters and the intensity of their arrivals.
As in  \cite{baradel2016optimal}, we allow the agent to issue new CAT bonds at any time, the actions are discrete but chosen in a continuous time framework.

To the best of our knowledge, the study of such a general problem with an application to the CAT bonds seems to be new in the literature, even in the case where all parameters are known. From a mathematical point of view, the main difficulty comes from the fact that the conditional distribution on the unknown parameters evolves continuously and jumps at the occurrence times of a catastrophic event. In \cite{baradel2016optimal}, it was only evolving when an action was taken on the system. For tractability,  we assume that the associated process remains in a finite-dimensional space which can be linked smoothly to a subset of $\mathbb{R}^{d}$ for some $d \geq 1$. Moreover, in \cite{baradel2016optimal}, only one action could be running at the same time. In this paper, we will deal with CAT bonds that are active simultaneously. This leads to specific boundary conditions on the characterization of the value function.

Although the  model presented below has been designed for the particular case of CAT bonds, it is quite general from a mathematical view-point and can be applied to all cases where the agent faces a random Poisson measure and can issue contracts from which he pays a premium and receives a specific payoff depending on some event.

\smallskip

The paper is organized as follows. Section 2 presents the framework. It introduces all concepts in order to describe the controlled problem. Section 3 gives the characterization of the associated value function as a PDE in the viscosity sense. Section 4 shows the viscosity properties, following and adapting the arguments of \cite{baradel2016optimal} in our context. Section 5 provides a sufficient condition for a comparison principle for the PDE satisfied by the value function. Section 6 gives a numerical scheme in order to solve numerically the controlled problem in practice. Finally, section 7 provides a case study of issuing CAT bonds in a optimal way, in the context of Hurricanes in Florida.

\section{The framework}

\subsection{General framework}

\bigskip

All over this paper, $D([0, T], \mathbb{R}^{d})$ is the Skorohod space of c\`{a}dl\`{a}g\footnote{Continue \`{a} droite, limite \`{a} gauche (right continuous with left limits)} functions from $[0, T]$ into $\mathbb{R}^{d}$, $\mathbb{P}$ is a probability measure on this space, and $T > 0$ is a fixed time horizon.

We consider three Polish spaces: $(U^{\lambda}, \mathcal{B}(U^{\lambda}))$ , $(U^{\gamma}, \mathcal{B}(U^{\gamma}))$ and $(U^{\upsilon}, \mathcal{B}(U^{\upsilon}))$ that will support three unknown parameters, respectively $\lambda_{0}$, $\gamma_{0}$ and $\upsilon_{0}$. Here $\mathcal{B}(.)$ denotes the Borel $\sigma$-algebra. We set $U := (U^{\lambda}, U^{\gamma}, U^{\upsilon}).$

Let $N(dt, du)$ be a random Poisson measure with compensator $\nu(dt,du)$ such that $\nu$ is finite on $(\mathbb{R}^{d*}, \mathcal{B}(\mathbb{R}^{d*}))$ where $\mathbb{R}^{d*} :=  \mathbb{R}^{d} \setminus \{0_{\mathbb{R}^{d}}\}$. The intensity of the random Poisson measure is supposed to be inhomogeneous of the form $s \mapsto \Lambda(s, \lambda_{0})$ where $\lambda_{0}$ is a random variable valued in $U^{\lambda}$. The jump distribution is assumed to be $\Upsilon(\gamma_{0}, \cdot)$ where $\gamma_{0}$ is a random variable valued in $U^{\gamma}$. The parameter $\upsilon$ is a random variable valued in $U^{\upsilon}$. We denote by $\mathbf{M}^{\lambda}$ a subset of the set of Borel probability measures on $U^{\lambda}$ and by $\mathbf{M}^{\gamma} \otimes \mathbf{M}^{\upsilon} =: \mathbf{M}$ the product of two locally compact subsets of the set of Borel probability measures, respectively on $U^{\gamma}$ and $U^{\upsilon}$, endowed with the weak topology.
 
We also allow an additional randomness when acting on the system and consider another Polish space $(E, \mathcal{B}(E))$ on which is defined a family $(\epsilon_{i})_{i \geq 1}$ of i.i.d. random variables with common probability measure $\mathbb{P}_{\epsilon}$ on $\mathcal{B}(E)$.

On the product space $\Omega := D([0, T], \mathbb{R}^{d}) \times U \times E^{\mathbb{N}^{*}}$, we consider the family of measures $\{\mathbb{P} \times \overline{m} \times \mathbb{P}_{\epsilon}^{\otimes \mathbb{N}^{*}}, \overline{m} \in \overline{\mathbf{M}} \}$ where $\overline{\mathbf{M}} := \mathbf{M}^{\lambda}  \otimes \mathbf{M}$. An element $\overline{m} \in \overline{\mathbf{M}}$ is a probability distribution on $(\lambda_0, \gamma_0, \upsilon_0)$. We denote by $\mathbb{P}_{\overline{m}}$ an element of this family whenever $\overline{m} \in \overline{\mathbf{M}}$ is fixed. The operator $\mathbb{E}_{\overline{m}}$ is the expectation associated with $\mathbb{P}_{\overline{m}}$. Note that $N(dt, du)$ and  $(\epsilon_{i})_{i \geq 1}$ are independent under each $\mathbb{P}_{\overline{m}}$. For $\overline{m} \in \overline{\mathbf{M}}$ given, we let $\mathbb{F}^{\overline{m}} := (\mathcal{F}_{t}^{\overline{m}})_{t \geq 0}$ denote the $\mathbb{P}_{\overline{m}}$-augmentation of the filtration $\mathbb{F} := (\mathcal{F}_{t})_{t \geq 0}$ defined by $\mathcal{F}_{t} := \sigma(N([0, s] \times \cdot)_{s \leq t}, \lambda_{0}, \gamma_{0}, \upsilon_{0}, (\epsilon_{i})_{i \geq 1})$. Hereafter, all random variables are considered with respect to the probability space $(\Omega, \mathcal{F}_{T}^{\overline{m}}, \mathbb{P}_{\overline{m}})$ with $\overline{m} \in \overline{\mathbf{M}}$ given by the context.

\subsection{CAT bond framework}

In this framework, $d \in \mathbb{N}^{*}$ is the number of perils. The insurer has some exposure related to these perils and may issue CAT bonds to reduce the risk taken. The random Poisson measure represents the arrival of claims. The intensity of arrival is $s \mapsto \Lambda(s, \lambda_{0})$ in which $\lambda_{0}$, valued in $U^{\lambda}$, may be unknown to the insurer. The dependence in time may represent the seasonality or a structural change, for example caused by the global warming.

The measure $m^{\lambda} \in \mathbf{M}^{\lambda}$ is the initial knowledge of the insurer on $\lambda_{0}$ and will evolve through the observations of $N$, whose jumps model the arrival of natural disasters. The severity distribution of the claims may also be unknown, it depends on the unknown parameter $\gamma_{0}$, valued in $U^{\gamma}$. An initial prior is given as an element $m^{\gamma} \in \mathbf{M}^{\gamma}$.  Acting on the system consists in issuing a CAT bond, which means transferring a part of the risk to the market. The equilibrium premium that the insurer will pay is random (since it comes from the law of supply and demand and is not known when the decision to issue is taken), and the distribution may not be perfectly known. We assume that it depends on the unknown parameter $\upsilon_{0}$, valued in $U^{\upsilon}$. Its prior distribution is represented by some $m^{\upsilon} \in \mathbf{M}^{\upsilon}$.

\subsection{The control}

Let $\mathbf{A} \subset \mathbb{R}^{d + 1}$ be a non-empty compact set. Let $\ell  \in \mathbb{R}^{*}_{+}$ be the time-length of each action on the controlled system. Given $\overline{m} \in \overline{\mathbf{M}}$, we denote by $\Phi^{\circ, \overline{m}}$ the collection of random variables $\phi = (\tau_{i}^{\phi}, \alpha_{i}^{\phi})_{i \geq 1}$ on $(\Omega, \mathcal{F}_{T}^{\overline{m}})$ with values in $\mathbb{R}_{+} \times \mathbf{A}$ such that $(\tau_{i}^{\phi})_{i \geq 1}$ is a non-decreasing sequence of $\mathbb{F}^{\overline{m}}$-stopping times and each $\alpha_{i}^{\phi}$ is $\mathcal{F}_{\tau_{i}}^{\overline{m}}$-measurable for $i \geq 1$. We shall write $\alpha_{i}^{\phi} := (k_{i}^{\phi}, n_{i}^{\phi}) \in \mathbf{A}$ where $k_{i}^{\phi}$ and $n_{i}^{\phi}$ are $ \mathbb{R}^{d}$ and $\mathbb{R}$-valued. To each $k_{i}^{\phi}$, we associate a non-empty closed set $A_{k_{i}^{\phi}} \subset \mathbb{R}^{d*}$ through a one-to-one map.

The $\tau_{i}^{\phi}$'s will be the times at which the $i$-th CAT bond is issued. The fixed value $\ell$ is the time-length (or maturity) of all CAT bonds. In $\alpha_{i}^{\phi} := (k_{i}^{\phi}, n_{i}^{\phi}) \in \mathbf{A}$, $n^{\phi}_{i}$ is related to the notional and $A_{k_{i}^{\phi}}$ is the layer chosen for one peril and one region: it is the characteristics of the CAT bonds associated with the risk covered. If a natural disaster occurs and its severity is in the layer $A_{k_{i}^{\phi}}$, i.e. the random Poisson measure has a jump in $A_{k_{i}^{\phi}}$, then the associated CAT bond ends and the reinsurer gains a payoff proportional to the notional $n_{i}^{\phi}$.

\begin{remark}
    Since there are $d \geq 1$ perils, the layer of the CAT bonds is defined in $\mathbb{R}^{d}$. In practice, one CAT bond covers one peril. For example, in order to cover the first peril, the sets $A_{k_{i}}$ will be of the form $A \times \mathbb{R}^{d-1}$ for some $A \subset \mathbb{R}_{+}^{*}$, where $A$ is the covered layer in relation to the natural disaster of the first peril. The case $d > 1$ gives the opportunity to manage the risk associated to several perils at the same time, with different CAT bonds on different perils, whereas $d = 1$ is the simple case in which we focus on a single risk.
\end{remark}

We denote by $\vartheta_{i}^{\phi}$ the end of the $i$-th CAT bond defined by:
		\begin{equation}\label{vartheta}
			\vartheta_{i}^{\phi} := \inf\{t > \tau_{i}^{\phi}, N(\{t\} \times A_{k_{i}}^{\phi}) = 1\} \wedge (\tau_{i}^{\phi} + \ell).
		\end{equation}

\begin{remark}
	According to the definition of $(\vartheta_{i}^{\phi})_{i \geq 1}$, it can happen that $\vartheta_{i_{1}}^{\phi} = \vartheta_{i_{2}}^{\phi}$ for $i_{1} \not= i_{2}$. Moreover,
		\[
			\tau_{i}^{\phi} < \vartheta_{i}^{\phi} \leq \tau_{i}^{\phi} + \ell.
		\]
\end{remark}

\noindent For $\kappa \in \mathbb{N}^{*}$, we say that $\phi \in \Phi^{\circ, \overline{m}}$ belongs to $\Phi^{\circ, \overline{m}}_{\kappa}$ if the condition
		\begin{equation}\label{N_cond}
			  \ \sum_{i \geq 1} \mathbf{1}_{\{\tau_{i}^{\phi} < t \leq \vartheta_{i}^{\phi}\}} \leq \kappa,\;\;\;\;\forall t \leq T
		\end{equation}
holds. This means that the insurer can hold a maximum of $\kappa$ running CAT bonds simultaneously. 

\subsection{The CAT bonds process}

We need to keep track of how many CAT bonds are running, and which parameters are associated with, in order to get a Markovian framework. A CAT bond will has its characteristics determined at $\tau_{i}$, for $\phi \in \Phi^{\circ, \overline{m}}_{\kappa}$. Moreover, a CAT bond will end from a jump or after the time-length $\ell$. We need to define a process which will keep track the characteristics and the time-length elapsed. We introduce the sets $\mathbf{C} := \left((\mathbb{R}^{d} \times \mathbb{R} \times \mathbf{A}) \cup \partial\right)^{\kappa}$, $\mathbf{L} := ([0, \ell[\cup\partial)^{\kappa}$, in which
	\begin{itemize}
		\item An element $(x, r, a)$ of the set $\mathbb{R}^{d} \times \mathbb{R} \times \mathbf{A}$ represents the initial parameters (characteristics) of the CAT bond, the first component will be the state process defined in the next subsection, the second one the coupon rate, and the last one the notional/layer chosen;
		\item An element of the set $[0, \ell [$ represents the time-length elapsed of a running CAT bond;
		\item The point $\partial$ represents the absence of CAT bond, it is a cemetery point. 
	\end{itemize}
  The set of CAT bonds is
    \[
    	\mathbf{CL} := \left\{(c, l) \in \mathbf{C} \times \mathbf{L} \ \ | \ \   c_{j} = \partial \Longleftrightarrow l_{j} = \partial,\; \forall \ 1 \leq j \leq \kappa\right\}
    \]
and we denote by $\overline{\mathbf{CL}}$ its closure. A component of $(c, l) \in \mathbf{CL}$ corresponds to a CAT bond, where $c$ are its characteristics and $l$ is the time-length elapsed since its issue. The special case $(c, l) = (\partial, \partial)$ describes the case where there is no CAT bond for this component, and there are $\kappa$ component: the maximum possible of running CAT bonds simultaneously.

\smallskip

We set $\mathbf{K} := \{0, \ldots, \kappa\}$ and we define by $\mathcal{P}(\mathbf{K})$ the set of subsets of $\mathbf{K}$. We can now define the sets $\mathbf{C}\mathbf{L}_{\mathbf{J}}$ with $\mathbf{J} \in \mathcal{P}(\mathbf{K})$:
	\[
        	\mathbf{C}\mathbf{L}_{\mathbf{J}} := \{(c, l) \in \mathbf{CL}  \ \ | \ \   j \in \mathbf{J} \Longleftrightarrow c_{j} \not= \partial,\; \forall \ 1 \leq j \leq \kappa \}  \\
    \]
which represent the sets of CAT bonds in which there are CAT bonds running exactly in the indexes of $\mathbf{J}$.
    
Moreover, for $(c, l) \in \mathbf{C}\mathbf{L} \backslash \mathbf{C}\mathbf{L}_{\mathbf{K}}$, we introduce:
	\[
		\Pi^{0}(c, l) := \min\{1 \leq j \leq \kappa : c^{j} = \partial\},
	\]
which is the first index with no CAT bond.

For $z := (t, x, c, l) \in \mathbf{Z} := [0, T] \times \mathbb{R}^{d} \times \mathbf{CL}$ and a control $\phi \in \Phi^{\circ, \overline{m}}_{\kappa}$, we now define the process $((C, L)_{s}^{z, \phi, j})_{t \leq s \leq T}^{1 \leq j \leq \kappa}$ valued in $\mathbf{C}\mathbf{L}$ and denoted hereafter $(C, L)$ for ease of notation. In $z$, the variable $t$ is the time and $x$ is the state of the output process, defined in next section. The process $(C, L)$ will jump at the $\tau_{i}'s$ (new CAT bond) and at the $\vartheta_{i}$'s (end of one or several CAT bonds). $C$ will be a pure jump process whereas the indexes of $L$ will evolve continuously over time, recall that it represents the elapsed time-length of the CAT bonds.

We now define the functions associated with the jumps of $(C, L)$. The first one, denoted by $\mathfrak{C}_{+}$, represents the arrival of one new CAT bond with parameters $(x, r, a) \in \mathbb{R}^{d} \times \mathbb{R} \times \mathbf{A}$ and is defined by
		\[
			\begin{aligned}
			\mathfrak{C}_{+} : (\mathbf{C}\mathbf{L} \backslash \mathbf{C}\mathbf{L}_{\mathbf{K}}) \times \mathbb{R}^{d} \times  \mathbb{R} \times \mathbf{A} &\rightarrow \mathbf{C} \mathbf{L} \\
						(c, l ; x, r, a) &\mapsto \mathfrak{C}_{+}(c, l ; x, r, a)
			\end{aligned}
		\]
such that,
	\begin{equation}\label{newCB}
		\begin{aligned}
			&\mathfrak{C}_{+}(c, l ; x, r, a)_{\Pi^{0}(c, l)} := ((x, r, a), 0), & \\
			&\mathfrak{C}_{+}(c, l ; x, r, a)_{j} =  (c, l)_{j} &j \not= \Pi^{0}(c, l).
		\end{aligned}
	\end{equation}

Above, when we issue a new CAT bond, we add it, with its characteristics, to the state process $(C, L)$, which is done by $\mathfrak{C}_{+}$.

The second function, denoted by $\mathfrak{C}_{-}$, represents the end of the CAT bonds by an event associated with the random Poisson measure, of severity $u \in \mathbb{R}^{d*}$, and is defined by
		\[
			\begin{aligned}
			\mathfrak{C}_{-} : \mathbf{C} \mathbf{L} \times \mathbb{R}^{d*} &\rightarrow \mathbf{CL} \\
						(c, l ; u) &\mapsto \mathfrak{C}_{-}(c, l ;u).
			\end{aligned}
		\]

Nonetheless, several CAT bonds may end with a single event. We define the set of indexes in $c \in \mathbf{C}$ which end after the natural disaster $u \in \mathbb{R}^{d*}$, by
		\begin{equation}\label{J}
			\mathcal{J}(c ; u) := \left\{j \in \{1, \ldots, \kappa\} \mid c_{j} \not= \partial, u \in A_{k_{j}} \right\}.
		\end{equation}
Using this set, $\mathfrak{C}_{-}(c, l ;u)$ is defined simply through its $j$-component
		\begin{equation}\label{Cj}
				\mathfrak{C}_{-}(c, l ;u)_{j} := \left\{
    \begin{array}{ll}
        \partial \times \partial & \text{ if }  j \in \mathcal{J}(c ; u)\\
        (c, l)_{j} & \text{ if } j \not\in \mathcal{J}(c ; u)
    \end{array},
    ~~~~1 \leq j \leq \kappa
\right..
		\end{equation}

Above, when a CAT bonds ends following a natural disaster, we clear it, with its characteristics, to the state process $(C, L)$, which is done by $\mathfrak{C}_{-}$.

It remains to consider the case where a CAT bond ends because $l_{j} = \ell$ for some $1 \leq j \leq \kappa$. We define:
		\[
			\begin{aligned}
			\mathfrak{C}_{-}^{\ell} : (\overline{\mathbf{C} \mathbf{L}} \backslash \overline{\mathbf{C}\mathbf{L}}_{\emptyset}) &\rightarrow \mathbf{C} \mathbf{L} \\
						(c, l) &\mapsto \mathfrak{C}_{-}^{\ell}(c, l),
			\end{aligned}
		\]
where, for all $1 \leq j \leq \kappa$,
	\[
    	\mathfrak{C}_{-}^{\ell}(c, l)_{j} = (\partial \times \partial)\mathbf{1}_{\{l_{j} = \ell\}} + (c, l)_{j}\mathbf{1}_{\{l_{j} \not= \ell\}}.
    \]

We are now in position to define the processes $((C, L)_{s}^{z, \phi})_{t \leq s \leq T}$ for $\phi \in \Phi^{\circ, \overline{m}}_{\kappa}$. The process evolves at $\tau_{i}^{\phi}$ and $\vartheta_{i}^{\phi}$, for $i \geq 1$, according to:
	\begin{equation}\label{sautsC}
		\begin{aligned}
			&(C, L)_{\tau_{i}^{\phi}+}^{z, \phi}:= \mathfrak{C}_{+}((C, L)_{\tau_{i}^{\phi}}^{z, \phi}) ; X_{\tau_{i}^{\phi}}^{z, \phi}, r_{i}^{\phi}, \alpha_{i}^{\phi}) ; \\
			&(C, L)_{\vartheta_{i}^{\phi}}^{z, \phi} \ := \mathbf{1}_{\{\vartheta_{i}^{\phi} < \tau_{i}^{\phi} + \ell\}}\mathfrak{C}_{-}((C, L)_{\vartheta_{i}^{\phi}-}^{z, \phi}, u_{i}) + \mathbf{1}_{\{\vartheta_{i}^{\phi} = \tau_{i}^{\phi} + \ell\}}\mathfrak{C}_{-}^{\ell}((C, L)_{\vartheta_{i}^{\phi}-}^{z, \phi}),
		\end{aligned}
	\end{equation}
in which
\begin{itemize}
    \item $X_{\tau_{i}^{\phi}}^{z, \phi}$ is the output process, defined in the next section,
    \item $r_{i}^{\phi} := \mathfrak{C}_{0}(\tau_{i}^{\phi}, X_{\tau_{i}^{\phi}-}, \alpha_{i}^{\phi}, \upsilon, \epsilon_{i})$ with $\mathfrak{C}_{0} : [0,T] \times \mathbb{R}^{d} \times \mathbf{A} \times U^{\upsilon} \times E \rightarrow \mathbb{R}$ a measurable function is the coupon size. It has a common noise with $\upsilon$ and a specific noise $\epsilon_i$.
\end{itemize}

\noindent Elsewhere, $C^{z, \phi}$ is constant. For $ 1 \leq j \leq \kappa$, $L^{z, \phi, j}$ evolves according to:
			\[
				dL^{z, \phi, j}_{t} = \mathbf{1}_{\{L^{z, \phi, j}_{t} \not= \partial\}}dt.
			\]
This closes the definition of the process $(C, L)$. Note that we separated both the initial parameters $C$ with the elapsed time-length $L$ since the second one will play a different role in the PDE characterization in consequence of its continuous part.

\begin{remark}
If $c \mapsto \Pi(c) := \#\{j \in \mathbf{K} : c^{j} \not= \partial\}$, the process $C^{z, \phi}$ (and also, by construction, $L^{z, \phi})$ satisfies :
	\[
		\begin{aligned}
		& \Pi(C_{s}^{z, \phi}) \leq \kappa,  \ \forall s \in [t, T],~\mathbb{P}_{\overline{m}}- a.s.
        \end{aligned}
	\]
\end{remark}

We also give a metric on $\overline{\mathbf{CL}}$.

\begin{definition}We associate to $\overline{\mathbf{CL}}$ the metric $\mathfrak{d}$ defined by
	\[
    	\begin{aligned}
    	\mathfrak{d}\left[(c, l), (c', l')\right] := &\sum_{j \in \mathbf{J}\cap\mathbf{J}'}\left[\|c_{j} - c_{j}'\|^{2} + (l_{j} - l'_{j})^{2}\right] + \sum_{j \in \mathbf{J} \backslash\mathbf{J}'}(\|c_{j}\|^{2} + l_{j}^{2}) \\
        &+ \sum_{j \in \mathbf{J}' \backslash\mathbf{J}}(\|c_{j}'\|^{2} + (l_{j}')^{2}) + Card(\mathbf{J} \Delta \mathbf{J}'),
        \end{aligned}
    \]
where $\mathbf{J}$ and $\mathbf{J}'$ are respectively the set of running CAT bonds of parameters $(c, l)$ and $(c', l')$.
\end{definition}

\subsection{The output process}

We are now in position to describe the controlled state process. Given some initial data $(t, x) \in [0, T] \times \mathbb{R}^{d}$, and $\phi \in \Phi^{\circ, \overline{m}}_{\kappa}$, we let $X^{t, x, \phi}$ be a strong solution on $[t, T]$ of
		\begin{equation}\label{X}
			\begin{aligned}
			X := x &+ \int_{t}^{\cdot}\left[\mu(s, X_{s}) + \overline{C}(s, C_{s})\right]ds\\
   &+  \int_{t}^{\cdot}\int_{\mathbb{R}^{d}}\left[\beta(s, X_{s-}, u) + \mathfrak{F}(s, X_{s-}, C_{s-}, L_{s-} ; u)\right]N(ds, du) \\
				&+ \sum_{i \geq 1} \mathbf{1}_{\{t \leq \tau_{i}^{\phi} < \cdot\}} H(\tau_{i}^{\phi}, X_{\tau_{i}^{\phi}}, \alpha_{i}^{\phi}),
			\end{aligned}
		\end{equation}
in which \begin{itemize}
    \item $\overline{C}$ is a function which gives the total coupon associated to the CAT bonds,
    \item $\mathfrak{F}$ is the \textit{payoff} of the running CAT bonds after a natural disaster, defined by
		\[
			\mathfrak{F}(t, x, c, l ; u) := \sum_{j \in \mathcal{J}(c ; u)}F(t, x, c^{j}, l^{j}; u),
		\]
  where $\mathcal{J}(c ; u)$ was defined in (\ref{J}) and $F(t, x, c^{j}, l^{j}; u)$ is the \textit{payoff} for the end of the CAT bond $c^{j}, l^{j}$ according to the jump $u$,
  \item $H$ is a function that gives the initial cost of issuing a CAT bond.
\end{itemize}

\noindent To guarantee existence and uniqueness of the above, we make the following standard assumptions.
\begin{Assumption} $\mu : [0, T] \times \mathbb{R}^{d} \mapsto \mathbb{R}^{d}$, $\beta : [0, T] \times \mathbb{R}^{d} \times \mathbb{R}^{d} \mapsto \mathbb{M}^{d}$ and $\overline{C} : [0, T] \times \mathbf{C} \mapsto \mathbb{R}^d$, are assumed to be measurable, continuous, and Lipschitz with linear growth in their second argument, uniformly in the other ones.
The maps $H : [0, T] \times \mathbb{R}^{d} \times \mathbf{A}$, and $F : [0, T] \times \mathbb{R}^{d} \times \mathbf{CL} \times \mathbb{R}^{d} $ are assumed to be measurable. Moreover, $H$ and $F$ have linear growth in their second component.
\end{Assumption}

In practice, the CAT bonds are usually issued by the special purpose vehicle (SPV). To simplify, we act as if the reinsurer and the SPV are a single entity, since it will not play an essential role for our purpose.

\medskip

This dynamics means the following. Without any CAT bond, the process $X$ follows a pure jump process with a drift described by $\mu$ and $\beta$ in (\ref{X}). In pratice, a component will be the cash. The function $\overline{C}$ is the instantaneous cash flow generated by the running CAT bonds whearas the function $\mathfrak{F}$ represents the \textit{payoff} of the CAT bonds that ends at the jump $u$. The third line refers to a jump of the whole process when a CAT bond is issued, for example, for a fixed cost.
    
\bigskip

	For $z := (t, x, c, l) \in \mathbf{Z}$, we shall write $X^{z, \phi}$ for the process $X$ starting with the CAT bonds $(c, l)$.  We denote by $\mathbb{F}^{z,\overline{m},\phi} := \left(\mathcal{F}_{s}^{z, \overline{m}, \phi}\right)_{s\ge 0}$ the $\mathbb{P}_{\overline{m}}$-augmentation of the filtration generated by $\left(X^{t, x, \phi}, \sum_{i \geq 1}r_{i}^{\phi}\mathbf{1}_{[\tau_{i}^{\phi}, +\infty[}, N([t, s] \times \cdot )_{s \geq t}\right)$. It corresponds to the observation of the reinsurer on which the admissible controls will be built.

\bigskip

For $\kappa \in \mathbb{N}^{*}$, we say that $\phi \in \Phi^{\circ, \overline{m}}_{\kappa}$ belongs to $\Phi^{z, \overline{m}}_{\kappa}$ if it is $\mathbb{F}^{z,\overline{m},\phi}$ adapted. The set $\Phi^{z, \overline{m}}_{\kappa}$ is the set of admissible controls. Recall that it satisfies the constraint (\ref{N_cond}) which refers to the fact that the controller cannot have more than $\kappa$ simultaneous running CAT bonds at each time.

\begin{remark}
    In our impulse framework with jumps, there remains a case to make explicit: how does the control work after the observation of a natural disaster? After such an event, the control may immediately issue a new CAT bond. At a jump date $\zeta$, the controller observes a size jump of $U$. On the same date, the control may issue a CAT bond, namely $\tau_{i} = \zeta$ for some $i\geq 1$. In this case, $X_{\tau_{i}}$ is $X_{\tau_{i}-}$ to which we add, firstly, the jump from the random Poisson measure (which may end some CAT bonds) and, secondly, the control action involves deciding whether or not to issue a new CAT bond based on the observation of the natural disaster and its consequences. The process $X$ is càdlàg.
\end{remark}

\subsection{Bayesian updates}

Obviously, the prior $\overline{m} \in \overline{\mathbf{M}}$ will evolve over time. Recall that $\overline{\mathbf{M}} := \mathbf{M}^{\lambda} \otimes \mathbf{M}$ and denote by $\overline{m} := (m^{\lambda}, m^{\gamma}, m^{\upsilon})$ the corresponding element. The observation of $X$ over time will lead to a continuous update of $m^{\lambda}$, whereas $m^{\gamma}$ will be updated by observing the size of a jump from $N$ and the measure $m^{\upsilon}$ will be updated by acting on the system at times $\tau_{i}^{\phi}$. This leads to the definition of the process $M := (M^{\lambda}, M^{\gamma}, M^{\upsilon})$ valued in $\overline{\mathbf{M}}$. We first focus on $m^{\lambda}$.

\subsubsection{Evolution of the intensity}\label{intensity}

We start with the assumption associated with the unknown and inhomogeneous intensity of the random Poisson measure.

\begin{Assumption}\label{h_lambda}
    For all $m^{\lambda} \in \mathbf{M}^{\lambda}$,
    
    \[
    t \mapsto \Lambda(t, \lambda_{0})
    \]
    is a c\`{a}dl\`{a}g process \ $m^{\lambda}-a.s.$
\end{Assumption}

Given $B \in \mathcal{B}(U^{\lambda})$, we set $M_{s}^{t,m^{\lambda}}(B) := \mathbb{E}_{\overline{m}}\left(\mathbf{1}_{\left\{\lambda_{0} \in B\right\}} | \mathcal{F}_{s}^{z,\overline{m}, \phi}\right)$ for $z=(t,x, c, l)$ and $\phi \in \Phi^{z,\overline{m}}_{\kappa}$.

From now on, we denote by $(\zeta_{i})_{i\geq1}$ the jump times associated with the random Poisson measure. We shall prove the following proposition.

\begin{proposition}\label{bayes_poisson}
    Under Assumption \ref{h_lambda}, the process $M_{s}^{t,m^{\lambda}}$ is
    \[
    M_{s}^{t,m^{\lambda}}(d\lambda) = \frac{\left[\prod_{t < \zeta_i \leq s}\Lambda(\zeta_{i}, \lambda)\right]e^{-\int_{t}^{s}\Lambda(u, \lambda)du}m^{\lambda}(d\lambda)}{\int_{U^{\lambda}}\left[\prod_{t < \zeta_i \leq s}\Lambda(\zeta_{i}, \lambda)\right]e^{-\int_{t}^{s}\Lambda(u, \lambda)du}m^{\lambda}(d\lambda)}.
    \]
\end{proposition}

\noindent The rest of this subsection is dedicated to the proof of the above proposition. To this aim, we first describe how $M$ evolves between two jumps and after at a jump. We will need the following technical remarks.

\begin{remark}\label{localement_Lambda}
Under Assumption \ref{h_lambda}, $m^{\lambda}-a.s.$ $u \mapsto \Lambda(u, \lambda_{0})$ is bounded on $[0, T]$ (\cite[Lemma 1 p122]{billingsley2013convergence}), and 
\[
\int_{s}^{t}\Lambda(u, \lambda_{0})du < +\infty
\]
$m^{\lambda}-a.s.$, for all $0 \leq s \leq t$.
\end{remark}

\begin{remark}\label{continuite_Lambda}
Under Assumption \ref{h_lambda}, for all $\epsilon > 0$ and $s\geq 0$, there exists $h_0$ such that, for all $0 \leq h \leq h_0$,
    \[
        \int_{s}^{s+h}\Lambda(u, \lambda_{0})du \leq h\left[\Lambda(s, \lambda_0) + \epsilon\right].
    \]
\end{remark}
\begin{remark}\label{integrale_Lambda}
	Since a c\`{a}dl\`{a}g function has at most a countable set of points of discontinuity, under Assumption \ref{h_lambda} we have
    	\begin{equation}\label{eq_Lambda_primitive}
			\int_{s}^{t}\Lambda(u, \lambda_{0})e^{-\int_{\alpha}^{u}\Lambda(v, \lambda_{0})dv}du =  e^{-\int_{\alpha}^{s}\Lambda(v, \lambda_{0})dv}-e^{-\int_{\alpha}^{t}\Lambda(v, \lambda_{0})dv} \;\;\mbox{$m^{\lambda}$-a.e.}
		\end{equation}
		for  almost all $0 \leq \alpha \leq s \leq t$.
\end{remark}

We now describe the evolution between two jumps.

\begin{lemma}\label{bayes_poisson_pas_saut}
For all  $z=(t,x, c, l)\in \mathbf{Z}$ and $s > t$, 
	\begin{equation*}
		M_{s}^{t,m^{\lambda}}(B)\mathbf{1}_{\left\{\zeta_{i} \leq s < \zeta_{i+1}\right\}} = \mathfrak{M}_{\lambda}(B ; \zeta_{i}, s)\mathbf{1}_{\left\{\zeta_{i} \leq s < \zeta_{i+1}\right\}} 
	\end{equation*}	
	where 
	$$
	\mathfrak{M}_{\lambda}(B ; \zeta_{i}, s)  := \frac{\int_{B}e^{-\int_{\zeta_{i}}^{s}\Lambda(u, \lambda)du}M_{\zeta_{i}}^{t,m^{\lambda}}(d\lambda)}{\int_{U^{\lambda}}e^{-\int_{\zeta_{i}}^{s}\Lambda(u, \lambda)du}M_{\zeta_{i}}^{t,m^{\lambda}}(d\lambda)}\mathbf{1}_{\left\{\zeta_{i} \leq s \right\}}  .
	$$
	
\end{lemma}

\begin{proof} Hereafter, $\cdot \wedge \cdot$ denotes the minimum and $\cdot \vee \cdot$ the maximum. Let $\varphi$ be a Borel bounded function on $D([0, T], \mathbb{R}^{d+1})$. Set $\xi^{\phi} := \sum_{i \geq 1}r_{i}^{\phi}\mathbf{1}_{[\tau_{i}^{\phi}, +\infty[}$,    $\delta X^{i} := X_{\cdot \vee \zeta_{i}}^{z, \phi} - X_{\zeta_{i}}^{z, \phi}$, and   $\delta \xi^{i} := \xi_{\cdot \vee \zeta_{i}} - \xi_{\zeta_{i}}$. We can find a Borel measurable map $\overline{\varphi}$ such that
	\[
		\varphi(X_{\cdot \wedge s}^{z, \phi}, \xi_{\cdot \wedge s}^{\phi})\mathbf{1}_{\left\{ \zeta_{i} \leq s < \zeta_{i+1}\right\}} = \overline{\varphi}(X_{\cdot \wedge \zeta_{i}}^{z, \phi}, \xi_{\cdot \wedge \zeta_{i}}^{\phi}, \delta X^{i}_{\cdot \wedge s}, \delta \xi^{i}_{\cdot \wedge s})\mathbf{1}_{\left\{ \zeta_{i} \leq s < \zeta_{i+1}\right\}}.
	\]
In view of Remark \ref{integrale_Lambda}, it then follows:
	\[
	\begin{aligned}
		&\mathbb{E}_{\overline{m}}\left(\mathbf{1}_{\left\{\lambda_{0} \in B \right\}}\mathbf{1}_{\left\{ \zeta_{i} \leq s < \zeta_{i+1}\right\}}\varphi(X_{\cdot \wedge s}^{z, \phi}, \xi_{\cdot \wedge s}^{\phi})\right)  \\
&= \mathbb{E}_{\overline{m}}\left(\mathbf{1}_{\left\{\lambda_{0} \in B \right\}}\mathbf{1}_{\left\{ \zeta_{i} \leq s < \zeta_{i+1}\right\}}\overline{\varphi}(X_{\cdot \wedge \zeta_{i}}^{z, \phi}, \xi_{\cdot \wedge \zeta_{i}}^{\phi}, \delta X^{i}_{\cdot \wedge s}, \delta \xi^{i}_{\cdot \wedge s})\right) \\
		&= \mathbb{E}_{\overline{m}}\left(\int_{\mathbb{R}_{+}} \mathbf{1}_{\left\{\lambda_{0} \in B \right\}}\mathbf{1}_{\left\{ \zeta_{i} \leq s < u\right\}}\overline{\varphi}(X_{\cdot \wedge \zeta_{i}}^{z, \phi}, \xi_{\cdot \wedge \zeta_{i}}^{\phi}, \delta X^{i}_{\cdot \wedge s}, \delta \xi^{i}_{\cdot \wedge s})\Lambda(u, \lambda_{0}) e^{-\int_{\zeta_{i}}^{u}\Lambda(v, \lambda_{0})dv}du\right)\\
		&= \mathbb{E}_{\overline{m}}\left(\mathbf{1}_{\left\{\lambda_{0} \in B \right\}}\overline{\varphi}(X_{\cdot \wedge \zeta_{i}}^{z, \phi}, \xi_{\cdot \wedge \zeta_{i}}^{\phi}, \delta X^{i}_{\cdot \wedge s}, \delta \xi^{i}_{\cdot \wedge s})\int_{\mathbb{R}_{+}}\mathbf{1}_{\left\{ \zeta_{i} \leq s < u\right\}}\Lambda(u, \lambda_{0}) e^{-\int_{\zeta_{i}}^{u}\Lambda(v, \lambda_{0})dv}du\right) \\
		&= \mathbb{E}_{\overline{m}}\left(\mathbf{1}_{\left\{\lambda_{0} \in B \right\}}\overline{\varphi}(X_{\cdot \wedge \zeta_{i}}^{z, \phi}, \xi_{\cdot \wedge \zeta_{i}}^{\phi}, \delta X^{i}_{\cdot \wedge s}, \delta \xi^{i}_{\cdot \wedge s})\mathbf{1}_{\left\{ \zeta_{i} \leq s \right\}} e^{-\int_{\zeta_{i}}^{s}\Lambda(v, \lambda_{0})dv}\right) \\
		&= \mathbb{E}_{\overline{m}}\left(\overline{\varphi}(X_{\cdot \wedge \zeta_{i}}^{z, \phi}, \xi_{\cdot \wedge \zeta_{i}}^{\phi}, \delta X^{i}_{\cdot \wedge s}, \delta \xi^{i}_{\cdot \wedge s})\mathbf{1}_{\left\{ \zeta_{i} \leq s \right\}} \int_{B} e^{-\int_{\zeta_{i}}^{s}\Lambda(v, \lambda)dv}M_{\zeta_{i}}^{t,m^{\lambda}}(d\lambda)\right) \\
		&= \mathbb{E}_{\overline{m}}\left(\overline{\varphi}(X_{\cdot \wedge \zeta_{i}}^{z, \phi}, \xi_{\cdot \wedge \zeta_{i}}^{\phi}, \delta X^{i}_{\cdot \wedge s}, \delta \xi^{i}_{\cdot \wedge s})\mathbf{1}_{\left\{ \zeta_{i} \leq s \right\}}\mathfrak{M}_{\lambda}(B ; \zeta_{i}, s)\int_{U^{\lambda}}e^{-\int_{\zeta_{i}}^{s}\Lambda(v, \lambda)dv}M_{\zeta_{i}}^{t,m^{\lambda}}(d\lambda)\right) \\
	&= \mathbb{E}_{\overline{m}}\left(\overline{\varphi}(X_{\cdot \wedge \zeta_{i}}^{z, \phi}, \xi_{\cdot \wedge \zeta_{i}}^{\phi}, \delta X^{i}_{\cdot \wedge s}, \delta \xi^{i}_{\cdot \wedge s})\mathbf{1}_{\left\{\zeta_{i} \leq s < \zeta_{i+1}\right\}}\mathfrak{M}_{\lambda}(B ; \zeta_{i}, s)\right) 	\\
	&= \mathbb{E}_{\overline{m}}\left(\varphi(X_{\cdot \wedge s}^{z, \phi}, \xi_{\cdot \wedge s}^{\phi})\mathbf{1}_{\left\{\zeta_{i} \leq s < \zeta_{i+1}\right\}}\mathfrak{M}_{\lambda}(B ; \zeta_{i}, s)\right) 
		\end{aligned}
	\]
This shows that $M_{s}^{t,m^{\lambda}}(B)\mathbf{1}_{\left\{\zeta_{i} \leq s < \zeta_{i+1}\right\}} = \mathfrak{M}_{\lambda}(B ; \zeta_{i}, s)\mathbf{1}_{\left\{\zeta_{i} \leq s < \zeta_{i+1}\right\}} \ \ \mathbb{P}_{\overline{m}}$-a.s.
\end{proof}

\begin{lemma}\label{bayes_L1}
	For all $m^{\lambda} \in \mathbf{M}^{\lambda}$ and almost all $s \geq t$, we have
    \begin{itemize}
    \item[i)]
    	\[
        	\int_{U^{\lambda}}\Lambda(s, \lambda)M_{s}^{t,m^{\lambda}}(d\lambda) < +\infty \ \ \mathbb{P}_{\overline{m}}-a.s.
        \]
       \item[ii)] 
       				\[
        					\int_{U^{\lambda}}\Lambda(\zeta_{i}, \lambda)M_{\zeta_{i}-}^{t,m^{\lambda}}(d\lambda) < +\infty \ \ \mathbb{P}_{\overline{m}}-a.s, \ \ i \geq 1.
				    \]
       \item[iii)]
       				\[
        					\int_{U^{\lambda}}\Lambda(s, \lambda)M_{s-}^{t,m^{\lambda}}(d\lambda) < +\infty \ \ \mathbb{P}_{\overline{m}}-a.s.
				    \]

       \end{itemize}
\end{lemma}
\begin{proof}
	{\bf Step 1}. For almost all $\lambda \in U^{\lambda}$, we fix $\mathcal{N}_{\lambda} \subset [0, T]$ the set of discontinuity of $t \mapsto \Lambda(t, \lambda)$ which is, at most, countable. We introduce:
	\[    
	    \mathcal{N}^{c} := \{\forall i \geq 1, \zeta_{i} \not\in \mathcal{N}_{\lambda_{0}}\}.
    \]
    We shall show that $\mathbb{P}_{\overline{m}}(\mathcal{N}^{c}) = 1$ by showing that  $\mathbb{P}_{\overline{m}}(\zeta_{i} \in \mathcal{N}_{\lambda_{0}}) = 0$ for all $i \geq 1$. Fix $i \geq 1$ and  remark that,  given $\lambda \in U^{\lambda}$, the distribution of $\zeta_{i} \mid \{\lambda_{0} = \lambda\}$ is absolutely continuous with respect to the Lebesgue measure. Denote by $f_{i \mid \lambda}$ a corresponding density function. Then,
    \[
    	\mathbb{P}_{\overline{m}}(\zeta_{i} \in \mathcal{N}_{\lambda_{0}}) = \int_{U^{\lambda}}\left[\int_{\mathbb{R}_{+}}\mathbf{1}_{\mathcal{N}_{\lambda}}(z)f_{i \mid \lambda}(z)dz\right]dm^{\lambda}(\lambda) = \int_{U^{\lambda}}0\,dm^{\lambda}(\lambda) = 0.
    \]
    
    {\bf Step 2.} We show {\it i)}. We set:
    	\[
        	K_{i}(s) := \left(\int_{U^{\lambda}}e^{-\int_{\zeta_{i}}^{s}\Lambda(u, \lambda)du}M_{\zeta_{i}}^{t,m^{\lambda}}(d\lambda) \right)^{-1} \leq  K_{i}(\zeta_{i+1}) \ \ \text{ on } \{\zeta_{i} \leq s < \zeta_{i+1}\}.
        \]
        We have, from Remark \ref{localement_Lambda},
        \[
        	K_{i}(\zeta_{i+1}) < +\infty.
        \]
Moreover, by Fubini's Lemma and Remark \ref{integrale_Lambda}, 
        
		\[
       		\begin{aligned}
            	\int_{\zeta_{i}}^{\zeta_{i+1}}\int_{U^{\lambda}}\Lambda(s, \lambda)e^{-\int_{\zeta_{i}}^{s}\Lambda(u, \lambda)du}M_{\zeta_{i}}^{t,m^{\lambda}}(d\lambda)ds &= \int_{U^{\lambda}}\int_{\zeta_{i}}^{\zeta_{i+1}}\Lambda(s, \lambda)e^{-\int_{\zeta_{i}}^{s}\Lambda(u, \lambda)du}dsM_{\zeta_{i}}^{t,m^{\lambda}}(d\lambda) \\
                &= \int_{U^{\lambda}}[1-e^{-\int_{\zeta_{i}}^{\zeta_{i+1}}\Lambda(u, \lambda)du}]M_{\zeta_{i}}^{t,m^{\lambda}}(d\lambda) < +\infty,
       		\end{aligned}
        \]
        on $\mathcal{N}^{c}$. On the other hand, using Lemma \ref{bayes_poisson_pas_saut}, 
        \[
        	\int_{\zeta_{i}}^{\zeta_{i+1}} \int_{U^{\lambda}}\Lambda(s, \lambda)M_{s}^{t,m^{\lambda}}(d\lambda)ds \leq   K_{i}(\zeta_{i+1})\int_{\zeta_{i}}^{\zeta_{i+1}}\int_{U^{\lambda}}\Lambda(s, \lambda)e^{-\int_{\zeta_{i}}^{s}\Lambda(u, \lambda)du}M_{\zeta_{i}}^{t,m^{\lambda}}(d\lambda)ds <+\infty
        \]
        on $\mathcal{N}^{c}$. 
        This shows that, for almost all $s \geq t$,  
        \[
        	\mathbf{1}_{\{\zeta_{i} < s < \zeta_{i+1}\}}\int_{U^{\lambda}}\Lambda(s, \lambda)M_{s}^{t,m^{\lambda}}(d\lambda) < +\infty \;\;\mbox{on $\mathcal{N}^{c}$.}
        \]
	        This leads to the result since $\zeta_{i} \rightarrow +\infty$ when $i \rightarrow +\infty$ for almost all $\omega$.

	{\bf Step 3.} We show {\it ii)}. Since $M^{t, m^{\lambda}}$ evolves continuously on all $]\zeta_{i}, \zeta_{i+1}[$, we also have,
    	\[
        	\int_{U^{\lambda}}\Lambda(\zeta_{i}-, \lambda)M_{\zeta_{i}-}^{t,m^{\lambda}}(d\lambda) < +\infty \ \ \mathbb{P}_{\overline{m}}-a.s.
        \]
      
    Moreover, on $\mathcal{N}^{c}$, $\zeta_{i}$ cannot be on a discontinuity of $\Lambda$ by construction, $ i \geq 1$. Then, we have, on $\mathcal{N}^{c}$,
    
    	\[
			\int_{U^{\lambda}}\Lambda(\zeta_{i}, \lambda)M_{\zeta_{i}-}^{t,m^{\lambda}}(d\lambda) < +\infty.
		\]
        
        {\bf Step 4.} We show {\it iii)}. We introduce:
        \[
        	A := \{s \in [t, T] : m^{\lambda}\left[\Lambda(s, \lambda_{0}) = 0\right] < 1 \}.
        \]
        
        Recall that, by construction, $M_{s}^{t, m^{\lambda}} << m^{\lambda}$ for all $s \geq t$. If $s \not\in A$, $\int_{U^{\lambda}}\Lambda(s, \lambda)M_{s}^{t,m^{\lambda}}(d\lambda) = 0 < +\infty$. If $s \in A$, the distribution of $\zeta_{i}$ is equivalent to the Lebesgue measure and then, by $ii)$, we get the result.
\end{proof}

We now look at the intensity at the observation of a jump $\zeta_{i}$.

\begin{lemma}\label{bayes_poisson_saut}
	For all  $z=(t,x, c, l)\in \mathbf{Z}$ and $B \in \mathcal{B}(U^{\lambda})$, 
		\[
			M_{\zeta_{i}}^{t,m^{\lambda}}(B) = \frac{\int_{B}\Lambda(\zeta_{i}, \lambda) M_{\zeta_{i}-}^{t,m^{\lambda}}(d\lambda)}{\int_{U^{\lambda}}\Lambda(\zeta_{i}, \lambda) M_{\zeta_{i}-}^{t,m^{\lambda}}(d\lambda)}, \ i \geq 1.
		\]
\end{lemma}

\begin{proof}
	We use the same notations as in the proof of Lemma \ref{bayes_poisson_pas_saut}. 
    
    \medskip

	\textbf{1.} For ease of notation, we set $B_{i}(\zeta) := \left\{\zeta_{i-1} < s, \zeta_{i} \in [s,s+h], s+h <\zeta_{i+1}\right\}$. For $s > 0$, we show that, with $\zeta_{0} := 0$,
	\begin{equation}\label{bayes_poisson_saut_h}
		\begin{aligned}
		M_{s+h}^{t,m^{\lambda}}(B)&\mathbf{1}_{B_{i}(\zeta)} = \mathfrak{M}_{\lambda}'(B ; M_{s-}^{t,m^{\lambda}}, s, h)\mathbf{1}_{B_{i}(\zeta)},
		\end{aligned}
	\end{equation}
    where
    \[
    \mathfrak{M}_{\lambda}'(B ; M_{s-}^{t,m^{\lambda}}, s, h) := \frac{\int_{B}\left[\int_{s}^{s+h}\Lambda(u, \lambda)du\right] e^{-\int_{s}^{s+h}\Lambda(u, \lambda)du}M_{s-}^{t,m^{\lambda}}(d\lambda)}{\int_{U^{\lambda}}\left[\int_{s}^{s+h}\Lambda(u, \lambda)du\right] e^{-\int_{s}^{s+h}\Lambda(u, \lambda)du}M_{s-}^{t,m^{\lambda}}(d\lambda)}.
    \]
    
  Let $\varphi$ be a Borel bounded function of $D([0, T], \mathbb{R}^{d+1})$, we can find a Borel measurable map $\overline{\varphi}$ such that
	\[
		\varphi(X_{\cdot \wedge s+h}^{z, \phi}, \xi_{\cdot \wedge s+h}^{\phi})\mathbf{1}_{B_{i}(\zeta)} = \overline{\varphi}(X_{\cdot \wedge s}^{z, \phi}, \xi_{\cdot \wedge s}^{\phi}, \delta X^{i}_{\cdot \wedge s+h}, \delta \xi^{i}_{\cdot \wedge s+h})\mathbf{1}_{B_{i}(\zeta)}.
	\]
We shall write $\overline{\varphi}(X, \xi)$ for $\overline{\varphi}(X_{\cdot \wedge s}^{z, \phi}, \xi_{\cdot \wedge s}^{\phi}, \delta X^{i}_{\cdot \wedge s+h}, \delta \xi^{i}_{\cdot \wedge s+h})$. It then follows: 
	\[
	\begin{aligned}
		&\mathbb{E}_{\overline{m}}\left(\mathbf{1}_{\left\{\lambda_{0} \in B \right\}}\mathbf{1}_{B_{i}(\zeta)}\varphi(X_{\cdot \wedge s+h}^{z, \phi}, \xi_{\cdot \wedge s+h}^{\phi})\right)  \\
		&= \mathbb{E}_{\overline{m}}\left(\mathbf{1}_{\left\{\lambda_{0} \in B \right\}}\mathbf{1}_{B_{i}(\zeta)}\overline{\varphi}(X, \xi)\right) \\
		&= \mathbb{E}_{\overline{m}}\left(\int_{U^{\lambda}} \mathbf{1}_{\left\{\lambda \in B \right\}}\mathbf{1}_{\left\{ \zeta_{i-1} < s\right\}}\overline{\varphi}(X, \xi)\left[\int_{s}^{s+h}\Lambda(u, \lambda)du\right] e^{-\int_{s}^{s+h}\Lambda(u, \lambda)du}M_{s-}^{t,m^{\lambda}}(d\lambda)\right)\\
		&= \mathbb{E}_{\overline{m}}\left(\overline{\varphi}(X, \xi)\mathbf{1}_{\left\{ \zeta_{i-1} < s \right\}} \int_{B} \left[\int_{s}^{s+h}\Lambda(u, \lambda)du\right] e^{-\int_{s}^{s+h}\Lambda(u, \lambda)du}M_{s-}^{t,m^{\lambda}}(d\lambda)\right) \\
		&= \mathbb{E}_{\overline{m}}\left(\overline{\varphi}(X, \xi)\mathbf{1}_{\left\{ \zeta_{i-1} < s \right\}}\mathfrak{M}_{\lambda}'(B ; M_{s-}^{t,m^{\lambda}}, s, h)\int_{U^{\lambda}}\left[\int_{s}^{s+h}\Lambda(u, \lambda)du\right] e^{-\int_{s}^{s+h}\Lambda(u, \lambda)du}M_{s-}^{t,m^{\lambda}}(d\lambda)\right) \\
	&= \mathbb{E}_{\overline{m}}\left(\overline{\varphi}(X, \xi)\mathbf{1}_{B_{i}(\zeta)}\mathfrak{M}_{\lambda}'(B ; M_{s-}^{t,m^{\lambda}}, s, h)\right) \\
	&= \mathbb{E}_{\overline{m}}\left(\varphi(X_{\cdot \wedge s+h}^{z, \phi}, \xi_{\cdot \wedge s+h}^{\phi})\mathbf{1}_{B_{i}(\zeta)}\mathfrak{M}_{\lambda}'(B ; M_{s-}^{t,m^{\lambda}}, s, h)\right) 
		\end{aligned}
	\]
This shows that (\ref{bayes_poisson_saut_h}) hold $\mathbb{P}_{\overline{m}}$-a.s.

\textbf{2.} For $i=1$, on $\{\zeta_{1} \geq s\}$, by Lemma \ref{bayes_L1} \textit{iii)}, $\Lambda(s, \lambda_{0}) \in L^{1}(M_{s-}^{t, m^{\lambda}})$ for almost all $s$. Using remark \ref{continuite_Lambda}, by the dominated convergence theorem, we deduce that
	\[
		M_{s}^{t,m^{\lambda}}(B)\mathbf{1}_{\left\{\zeta_{0} < s,  \zeta_{1} = s\right\}} = \frac{\int_{B}\Lambda(s, \lambda) M_{s-}^{t,m^{\lambda}}(d\lambda)}{\int_{U^{\lambda}}\Lambda(s, \lambda) M_{s-}^{t,m^{\lambda}}(d\lambda)}\mathbf{1}_{\left\{\zeta_{0} < s,  \zeta_{1} = s\right\}},
	\]
i.e., since the law of $\zeta_{1}$ is absolutely continuous with respect to the Lebesgue measure,
	\[
		M_{\zeta_{1}}^{t,m^{\lambda}}(B) = \frac{\int_{B}\Lambda(\zeta_{1}, \lambda) M_{\zeta_{1}-}^{t,m^{\lambda}}(d\lambda)}{\int_{U^{\lambda}}\Lambda(\zeta_{1}, \lambda) M_{\zeta_{1}-}^{t,m^{\lambda}}(d\lambda)} \ \ \mathbb{P}_{\overline{m}} \text{ - a.s.}
	\]
Since almost surely, $\zeta_{i+1} > \zeta_{i}, i \geq 1$, and since the law of each $\zeta_{i}$ is absolutely continuous with respect to the Lebesgue measure, we deduce the result by a straightforward induction.
\end{proof}

With lemma \ref{bayes_poisson_pas_saut} and lemma \ref{bayes_poisson_saut}, a direct induction shows proposition \ref{bayes_poisson}.

	\subsubsection{Evolution of the parameters $\gamma_{0}$ and $\upsilon_{0}$}

We use the notations of Section \ref{intensity}. We define $M_{s}^{t,m^{\gamma}}(B) := \mathbb{E}_{\overline{m}}\left(\mathbf{1}_{\left\{\gamma \in B\right\}} | \mathcal{F}_{s}^{z,\overline{m}, \phi}\right)$ and \\ $M_{s}^{z,m^{\upsilon}, \phi}(B) := \mathbb{E}_{\overline{m}}\left(\mathbf{1}_{\left\{\upsilon \in B\right\}} | \mathcal{F}_{s}^{z,\overline{m}, \phi}\right)$.

Between two jumps of the random Poisson measure, no information about the size distribution of the jumps is revealed, and therefore, about $\gamma_{0}$. Whereas no information is revealed about $\upsilon$ between two jumps from our control. In this case, both processes should remain constant. At the $i$-th Poisson jump of size $u_{i}$, the process $M^{t, m^{\gamma}}$ should evolve according to the classical Bayes rule. The process $M^{z, m^{\upsilon}, \phi}$ should evolve at the time of issuance of the $j$-th CAT bond with the coupon $c_{j}$ according to, again, the Bayes rule.
\begin{lemma} Fix $s\ge 0$. Assume that, for almost all $\gamma \in U^{\gamma}$, the claim size distribution is dominated by some common measure $\mu_{\circ}$.
	We have
		\[
			\begin{aligned}
				&M_{s}^{t,m^{\gamma}}(B)\mathbf{1}_{\left\{\zeta_{i} \leq s < \zeta_{i+1}\right\}} = M_{\zeta_{i}}^{t,m^{\gamma}}(B)\mathbf{1}_{\left\{\zeta_{i} \leq s < \zeta_{i+1}\right\}} \\
				&M_{\zeta_{i}}^{t,m^{\gamma}}(B) = \mathfrak{M}_{\gamma}(M_{\zeta_{i}-}^{t,m^{\gamma}}(B); U_{i})
			\end{aligned}
		\]
in which
		\[
			\mathfrak{M}_{\gamma}(m_{\circ}^{\gamma}; u_{\circ}) = \frac{\int_{B}\qr_{\gamma}(u_{\circ} \mid \gamma)dm_{\circ}^{\gamma}(\gamma)}{\int_{U}\qr_{\gamma}(u_{\circ} \mid \gamma)dm_{\circ}^{\gamma}(\gamma)}.
		\]
for almost all $(m_{\circ}^{\gamma}, u_{\circ}) \in \mathbf{M}^{\gamma} \times \mathbb{R}^{d*}$,  in which $\qr_{\gamma}(u_{\circ} \mid \gamma)$ is the conditional density, with respect to $m_{\circ}^{\gamma}$, of observing a jump of size $u_{\circ}$ knowing $\{\gamma_{0} = \gamma\}$.

Moreover,
		\[
			\begin{aligned}
				&M_{s}^{t,m^{\upsilon}, \phi}(B)\mathbf{1}_{\left\{\tau_{j} \leq s < \tau_{j+1}\right\}} = M_{\tau_{j}}^{t,m^{\upsilon}, \phi}(B)\mathbf{1}_{\left\{\tau_{j} \leq s < \tau_{j+1}\right\}} \\
				&M_{\tau_{j}}^{t,m^{\upsilon}, \phi}(B) = \mathfrak{M}_{\upsilon}(M_{\tau_{j}-}^{t,m^{\upsilon}, \phi}(B); r_{j}, \tau_{j}, X_{\tau_{j}-}^{z, \phi}, \alpha_{j})
			\end{aligned}
		\]
in which
		\[
			\mathfrak{M}_{\upsilon}(m_{\circ}^{\upsilon}; r_{\circ}, t_{\circ}, x_{\circ}, a_{\circ}) = \frac{\int_{C}\qr_{\upsilon}(r_{\circ} \mid t_{\circ}, x_{\circ}, a_{\circ}, \upsilon)dm_{\circ}^{\upsilon}(\upsilon)}{\int_{U}\qr_{\upsilon}(r_{\circ} \mid t_{\circ}, x_{\circ}, a_{\circ}, \upsilon)dm_{\circ}^{\upsilon}(\upsilon)}.
		\]
for almost all $(m_{\circ}^{\upsilon}, r_{\circ}, t_{\circ}, x_{\circ}, a_{\circ}) \in \mathbf{M}^{\upsilon} \times \mathbb{R} \times [0, T] \times \mathbb{R}^{d} \times \mathbf{A}$,  in which $\qr_{\upsilon}(r_{\circ} \mid t_{\circ}, x_{\circ}, a_{\circ}, \upsilon)$ is the conditional density, with respect to $m_{\circ}^{\upsilon}$, of observing a jump of size $r_{\circ}$ knowing $\{\tau_{j} = t_{\circ}, X_{\tau_{j}-}^{z, \phi} = x_{\circ}, \alpha_{i} = a_{\circ}, \upsilon_{0} = \upsilon\}$.
\end{lemma}

\begin{proof}
	Use the same arguments as in the proof of Proposition 2.1 in \cite{baradel2016optimal}.
\end{proof}

\subsection{Parametrization of the set $\mathbf{M}^{\lambda}$}

Here, we have three measures on which will depend the value function. The one associated with the distribution of the jumps of the Poisson measure (parameter $\gamma_0$) and the one from the coupon distribution (parameter $\upsilon$) evolve by a finite number of jumps on each bounded interval. Those will not lead to deal with derivatives on the space of measures and a specific It\^{o} formula nor generator of the diffusion. However, the measure associated with the intensity (parameter $\lambda_0$) evolves continuously. To deal with this, we will assume that the associated space of measures can be linked smoothly to a subset of $\mathbb{R}^{k}$ for some $k \geq 1$.

	\begin{Assumption}\label{m_lambda}
		We assume that there exists an open or compact set $\mathbf{P} \subset \mathbb{R}^{k}$, for some $k \in \mathbb{N}^{*}$, and a function
		\[
			\begin{aligned}
			f : \mathbf{P} &\rightarrow \mathbf{M}_{\lambda} \\
					\theta& \mapsto f(\theta),
			\end{aligned}
		\]
which is a homeomorphism between $\mathbf{P}$ and $\mathbf{M}_{\lambda}$.
	\end{Assumption}

\begin{remark}\label{remark_P}
	The process $P^{t, p}$ defined by:
    \[
			p = f^{-1}(m^{\lambda}), \ \ P_{s}^{t, p} := f^{-1}(M_{s}^{t, m^{\lambda}}), \ \ s \geq t,
	\]
    
    remains, by construction, in $\mathbf{P}$. Moreover, Lemma \ref{bayes_poisson_pas_saut} and \ref{bayes_poisson_saut} provide that $M^{t, m^{\lambda}}$ only depends on the stopping times of the jumps of the random Poisson measure on $[0, t]$, thus, $M^{t, m^{\lambda}}$ is $\mathcal{F}^{N} := s \mapsto \sigma\left(N(u, \cdot), t \leq u \leq s\right)$-adapted. Then, from Assumption \ref{m_lambda}, $P^{t, p}$ is also $\mathcal{F}^{N}$-adapted. Moreover, $P^{t, p}$ does not depend on the size of the jumps.
\end{remark}

\noindent According to Remark \ref{remark_P}, we formulate the following assumption.

\begin{Assumption}\label{meds}
	Let $P^{t, p}$ be the process defined in Remark \ref{remark_P}.

	There exist Lipschitz maps $h_{1} : [0, T] \times \mathbf{P} \rightarrow \mathbb{R}^{k}$ and $h_{2} : [0, T] \times \mathbf{P} \rightarrow \mathbb{R}^{k}$ with linear growth (uniformly  in time) such that
		\[
        	\begin{aligned}
			P^{t, p} &= p + \int_{t}^{\cdot}h_{1}(s, P_{s}^{t, p})ds + \int_{t}^{\cdot}\int_{\mathbb{R}^{d*}}h_{2}(s, P_{s-}^{t, p})N(ds,du) \\
            	&= p + \int_{t}^{\cdot}h_{1}(s, P_{s}^{t, p})ds + \int_{t}^{\cdot}h_{2}(s, P_{s-}^{t, p})dN_{s},
            \end{aligned}
		\]
        
        where we use the notation: $dN_{s} := N(ds, \mathbb{R}^{d*})$.
\end{Assumption}

\noindent We provide two examples in which the Assumptions \ref{m_lambda} and \ref{meds} are fulfilled. The first one uses the conjugate prior on $\lambda_0$ with the Gamma distribution when $\Lambda$ is linear in $\lambda_0$ in our framework. The second one requires a finite discrete space, but with no other assumption, however, the dimension is quickly high.

\begin{example}\label{gamma}
Assume that there exists a c\`{a}dl\`{a}g function $h : [0, T] \mapsto \mathbb{R}^{+}$ such that $\Lambda(t, \lambda) = \lambda h(t)$ for all $t \geq 0, \lambda \in U^{\lambda}$.  Set $m^{\lambda} = M_{t}^{t, m^{\lambda}} := \mathcal{G}(\alpha_{t}, \beta_{t})$, where $\mathcal{G}$ denotes the Gamma distribution. Then, if we define

\[
\left(\alpha, \beta\right) := \left(\alpha_{t} + N - N_{t}, \beta_{t} + \int_{t}^{\cdot}h(u)du\right),
\]

it follows that

	\[
    	M^{t, m^{\lambda}} = \mathcal{G}\left(\alpha, \beta\right),
    \]

and $P^{t, p} = (\alpha, \beta)$ satisfies Assumption \ref{meds}.

\end{example}

\begin{example}\label{bernoulli}
	Assume that 
    	$
        		U^{\lambda} := \{\lambda_{1}, \ldots, \lambda_{k}\} \in (\mathbb{R}_{+}^{*})^{k}.
        $
	Define, for $p = (p_{i})_{1 \leq i \leq k}$ with  $p_{i} > 0, 1 \leq i \leq k$, the distribution $\mathcal{D}(p)$ by:
    \[
    		\mathcal{D}(p) := \frac{\sum_{i=1}^{k}p_{i}\delta_{\lambda_{i}}}{\sum_{i=1}^{k}p_{i}}.
    \]
    
   Set, for  $s \geq t$,
    
    \[
    	P_{s}^{t, p, i} := p^{i}\left[\prod_{j = N_{t}+1}^{N_{s}}\Lambda(\zeta_{j}, \lambda_{i})\right]e^{-\int_{s}^{t}\Lambda(u, \lambda_{i})du}, \ \ 1 \leq i \leq k.
    \]
    
    Then $M^{t, m^{\lambda}} = \mathcal{D}(P^{t, p})$ and the process above satisfies the stochastic differential equation:
    
    \[
    	P^{t, p, i} = p^{i} - \int_{t}^{\cdot}P_{s}^{t, p, i}\Lambda(s, \lambda_{i})ds + \int_{t}^{\cdot}P_{s-}^{t, p, i}[\Lambda(s, \lambda_{i})-1]dN_{s}, \ \ 1 \leq i \leq k,
    \]
    and $P^{t, p} := (P^{t, p, i})_{1 \leq i \leq k}$ satisfies Assumption \ref{meds} under Assumption \ref{h_lambda} since, for each $1 \leq i \leq k$, $t \mapsto \Lambda(t, \lambda_i)$ is bounded on $[0, T]$ (see Remark \ref{localement_Lambda}).
    \end{example}

\subsection{Gain function}

Given $z = (t,x, c, l)  \in \mathbf{Z}$ and $(p, m) \in \mathbf{P} \times \mathbf{M}$, the aim of the controller is to maximize the expected value of the gain functional
	\[
		\phi \in \Phi^{z, \overline{m}} \mapsto G^{z, p, m}(\phi) := g(X_{T}^{z, \phi}, C_{T}^{z, \phi}, L_{T}^{z, \phi}, P_{T}^{t, p}, M_{T}^{z, m, \phi}),
	\]
in which $g$ is a continuous and bounded function on $\mathbb{R}^{d} \times \mathbf{C}\mathbf{L} \times \mathbf{P} \times \mathbf{M}$. 

\medskip

\noindent Given $\phi \in \Phi^{z, \overline{m}}_{\kappa}$, the expected gain is
	\[
		J(z, p, m; \phi) := \mathbb{E}_{\overline{m}}[G^{z, p, m}(\phi)],
	\]
and
	\[
		\vr(z, p, m) := \sup_{\phi \in \Phi^{z,\overline{m}}_{\kappa}} J(z, p, m ; \phi)
	\]
is the corresponding value function. Note that $\vr$ is bounded.

\section{Value function characterization}

For ease of notation, we define $\mathbf{D} := [0, T] \times \mathbb{R}^{d} \times \mathbf{CL} \times \mathbf{P} \times \mathbf{M}$, and for $\mathbf{J} \in \mathcal{P}(\mathbf{K})$, $\mathbf{D}_{\mathbf{J}} := [0, T] \times \mathbb{R}^{d} \times \mathbf{CL}_{\mathbf{J}} \times \mathbf{P} \times \mathbf{M}$. To $\mathbf{J} \in \mathcal{P}(\mathbf{K})$, we denote by $\mathbf{1}_{\mathbf{J}} = (\mathbf{1}_{\mathbf{J}}(j))_{1 \leq j \leq \kappa}$ the vector in $\mathbb{R}^{\kappa}$ in which $\mathbf{1}_{\mathbf{J}}(j) = 1$ if $j \in \mathbf{J}$, 0 else.

\medskip

\noindent For $(z, p, m) \in \mathbf{D}$ and $u \in \mathbb{R}^{d*}$, we introduce the operator $\mathcal{I}$ defined, for all $(z, p, m) \in \mathbf{D}$, by:
\[
	\mathcal{I}[\varphi, u](z, p, m) := \varphi(t, x + \beta(t, x, u) + \mathfrak{F}(z ; u), \mathfrak{C}_{-}(c, l; u), p+h_{2}(t, p), \mathfrak{M}_{\gamma}(m^{\gamma}; u), m^{\upsilon}).
\]
Thus, the Dynkin operator associated with our problem with policies running in indexes $\mathbf{J}$ is:

	\begin{equation*}
		\begin{aligned}
		& \mathcal{L}^{\mathbf{J}}\varphi :=  \partial_{t}\varphi +  \langle\mu + \overline{C}, D\varphi\rangle + \langle \mathbf{1}_{\mathbf{J}}, D_{l}\varphi \rangle + \langle h_{1}, D_{p}\varphi \rangle + \\
		&\int_{\mathbb{R}^{d}}\left[\mathcal{I}[\varphi, u] - \varphi\right]\Lambda(t, \lambda_{0})\Upsilon(\gamma_{0}, du),
		\end{aligned}
	\end{equation*}
in which $D$ is the differentiation operator with respect to $x$, $D_l$ with respect to $l$, and $D_p$ with respect to $p$. Recall that $\Upsilon$ denotes the size distribution of the jumps of the random Poisson measure $N$. Moreover, we introduce:
	\[
		\mathcal{L}_{\star}^{\mathbf{J}}\varphi := \mathbb{E}_{\overline{m}}\left[\mathcal{L}^{\mathbf{J}}\varphi\right],
	\]
    and
    \[
    	\begin{aligned}
        	\mathbf{D}_{\circ} &:= [0, T) \times \mathbb{R}^{d} \times \mathbf{CL}_{\mathbf{J}} \times \mathbf{P} \times \mathbf{M}, \\
            \mathbf{D}_{T} &:= \{T\} \times \mathbb{R}^{d} \times \mathbf{CL} \times \mathbf{P} \times \mathbf{M}.
        \end{aligned}
    \]
Then, we expect that $\vr$ is a viscosity solution of, for each $\mathbf{J} \in \mathcal{P}(\mathbf{K})$ and non-empty $\mathbf{J}' \subset \mathbf{J}$,
		\begin{align}
			\mathbf{1}_{\{\mathbf{J} = \mathbf{K}\}}\left[-\mathcal{L}_{\star}^{\mathbf{K}}\varphi \right]+ \mathbf{1}_{\{\mathbf{J} \not= \mathbf{K}\}}\min\{-\mathcal{L}_{\star}^{\mathbf{J}}\varphi, \varphi - \mathcal{K}\varphi\} = 0 & \ \text{ on } \mathbf{D}_{\circ} \label{interieur} \\
			\varphi = \mathbf{1}_{\{\mathbf{J} = \mathbf{K}\}}g + \mathbf{1}_{\{\mathbf{J} \not= \mathbf{K}\}}\max\left\{\mathcal{K}g, g\right\} & \ \text{ on } \mathbf{D}_{T}\label{bordT} \\
			\lim_{l' \rightarrow \mathfrak{L}_{\mathbf{J}}^{\mathbf{J}'}(l)} \varphi (., c, l', .) = \max\{\varphi (., \mathfrak{C}_{-}^{\ell}[c, \mathfrak{L}_{\mathbf{J}}^{\mathbf{J}'}(l)], .), \mathcal{K}\varphi (., \mathfrak{C}_{-}^{\ell}[c, \mathfrak{L}_{\mathbf{J}}^{\mathbf{J}'}(l)], .)\} & \ \text{ on } \mathbf{D}\backslash\mathbf{D}_{\mathbf{\emptyset}} \label{bordl}
		\end{align}
in which, for $(z, p, m) \in \mathbf{D}_{\mathbf{J}}$ and $\phi^{a} \in \Phi^{z, \overline{m}}$ a control such that $\{\tau_{1}^{\phi^{a}} = t, \alpha_{1}^{\phi^{a}} = a\}$ holds with probability one,
	 \[
		\mathcal{K}\varphi := \sup_{a \in \mathbf{A}}\mathcal{K}^{a}\varphi, \ \ \ \mathcal{K}^{a}\varphi(z, p, m) := \mathbb{E}_{\overline{m}}[\varphi(Z_{t+}^{z,\phi^{a}}, p, M_{t+}^{z, m, \phi^{a}})] ;
	\]
and, for $\mathbf{J}' \subset \mathbf{J}$, 
	\begin{align}
		\mathfrak{L}_{\mathbf{J}}^{\mathbf{J}'} : [0, \ell]^{\mathbf{J}} &\rightarrow [0, \ell]^{\mathbf{J}}  \\
						(l_{j})_{1 \leq j \leq \kappa}	&\mapsto (\ell\mathbf {1}_{\{j \in\mathbf{J}'\}} + l_{j}\mathbf {1}_{\{j \not\in \mathbf{J}'\}})_{1 \leq j \leq k},
	\end{align}
where $[0, \ell]^{\mathbf{J}} := \left\{l \in ([0, \ell] \cup \partial)^{\kappa} : l_{j} \not= \partial \Leftrightarrow  j \in \mathbf{J}\right\}$.
\begin{remark}
	Note that the above corresponds to the definition of a system of PDEs linked by the common boundary conditions.
\end{remark}

\noindent We now define what is a viscosity solution of (\ref{interieur})-(\ref{bordT})-(\ref{bordl}). For $\mathbf{J} \in \mathcal{P}(\mathbf{K})$, we define:
		\[
			 \mathscr{C}^{1}_{\mathbf{J}} := \left\{\varphi : \mathbf{D}_{\mathbf{J}} \mapsto \mathbb{R}, \ \ \varphi \in C^{1,1,(0,1),1, 0}(\mathbf{D}_{\mathbf{J}})\right\}.
		\]

\begin{definition}
	We say that a upper-semicontinuous function $u$ on $\mathbf{D}$ is a viscosity sub-solution of (\ref{interieur})-(\ref{bordT})-(\ref{bordl}) if, for any $\mathbf{J} \in \mathcal{P}(\mathbf{K})$, $(z_{\circ}, p_{\circ}, m_{\circ}) \in \mathbf{D}_{\mathbf{J}}$, and $\varphi \in \mathscr{C}^{1}_{\mathbf{J}}$ such that $\max_{\mathbf{D}_{\mathbf{J}}}(u-\varphi) = (u-\varphi)(z_{\circ}, p_{\circ}, m_{\circ}) = 0$ we have, if $t_{\circ} < T$,
	\[
		\mathbf{1}_{\{\mathbf{J} = \mathbf{K}\}}\left[-\mathcal{L}_{\star}^{\mathbf{K}}\varphi\right] + \mathbf{1}_{\{\mathbf{J} \not= \mathbf{K}\}}\min\{-\mathcal{L}_{\star}^{\mathbf{J}}\varphi, \varphi - \mathcal{K}u)\}(z_{\circ}, p_{\circ}, m_{\circ}) \leq 0,
	\]
if $\mathbf{J} \not= \emptyset$, for any non-empty $\mathbf{J}' \in \mathcal{P}(\mathbf{J})$, with $d_{\circ} = (t_{\circ}, x_{\circ}, c_{\circ}, \mathfrak{L}^{\mathbf{J}'}_{\mathbf{J}}(l_{\circ}), p_{\circ}, m_{\circ})$ and $d_{\circ}' = (t_{\circ}, x_{\circ}, \mathfrak{C}_{-}^{\ell}[c_{\circ}, \mathfrak{L}^{\mathbf{J}'}_{\mathbf{J}}(l_{\circ})], p_{\circ}, m_{\circ})$,
	\[
		\limsup_{(z, p, m) \rightarrow d_{\circ} } u(z, p, m) \leq  \max\left\{u(d_{\circ}'), \mathcal{K}u(d_{\circ}')\right\},
	\]
and, if $t_{\circ} = T$,
	\[
		u(z_{\circ}, p_{\circ}, m_{\circ}) \leq \left\{\mathbf{1}_{\{\mathbf{J} = \mathbf{K}\}}g + \mathbf{1}_{\{\mathbf{J} \not= \mathbf{K}\}}\max(\mathcal{K} g, g)\right\}(x_{\circ}, c_{\circ}, l_{\circ}, p_{\circ}, m_{\circ}).
	\]

	We say that a lower-semicontinuous function $v$ on $\mathbf{D}$ is a viscosity super-solution of (\ref{interieur})-(\ref{bordT})-(\ref{bordl}) if, for any $\mathbf{J} \in \mathcal{P}(\mathbf{K})$, $(z_{\circ}, p_{\circ}, m_{\circ}) \in \mathbf{D}_{\mathbf{J}}$, and $\varphi \in \mathscr{C}^{1}_{\mathbf{J}}$ such that $\min_{\mathbf{D}_{\mathbf{J}}}(v-\varphi) = (v-\varphi)(z_{\circ}, p_{\circ}, m_{\circ}) = 0$ we have, if $t_{\circ} < T$,
	\[
		\mathbf{1}_{\{\mathbf{J} = \mathbf{K}\}}\left[-\mathcal{L}_{\star}^{\mathbf{K}}\varphi\right] + \mathbf{1}_{\{\mathbf{J} \not= \mathbf{K}\}}\min\{-\mathcal{L}_{\star}^{\mathbf{J}}\varphi, \varphi - \mathcal{K}v)\}(z_{\circ}, p_{\circ}, m_{\circ}) \geq 0,
	\]
if $\mathbf{J} \not= \emptyset$, for any non-empty $\mathbf{J}' \in \mathcal{P}(\mathbf{J})$, with $d_{\circ} = (t_{\circ}, x_{\circ}, c_{\circ}, \mathfrak{L}^{\mathbf{J}'}_{\mathbf{J}}(l_{\circ}), p_{\circ}, m_{\circ})$ and $d_{\circ}' = (t_{\circ}, x_{\circ}, \mathfrak{C}_{-}^{\ell}[c_{\circ}, \mathfrak{L}^{\mathbf{J}'}_{\mathbf{J}}(l_{\circ})], p_{\circ}, m_{\circ})$,
	\[
		\liminf_{(z, p, m) \rightarrow d_{\circ}} v(z, p, m) \geq  \max\left\{v(d_{\circ}'), \mathcal{K}v(d_{\circ}')\right\}
	\]
and, if $t_{\circ} = T$,
	\[
		v(z_{\circ}, p_{\circ}, m_{\circ}) \geq \left\{\mathbf{1}_{\{\mathbf{J} = \mathbf{K}\}}g + \mathbf{1}_{\{\mathbf{J} \not= \mathbf{K}\}}\max(\mathcal{K} g, g)\right\}(x_{\circ}, c_{\circ}, l_{\circ}, p_{\circ}, m_{\circ}).
	\]

	We say that a function $u$ is a viscosity solution of (\ref{interieur})-(\ref{bordT})-(\ref{bordl}) if its upper-semicontinuous envelope $u^{*}$ is a viscosity sub-solution and its lower-semicontinuous envelope $u_{*}$ is a viscosity super-solution of (\ref{interieur})-(\ref{bordT})-(\ref{bordl}).

\end{definition}

To ensure that the above operator is continuous, we first assume that:
\begin{Assumption}\label{K_continuity}
		 $\mathcal{K}\varphi$ is upper- (resp. lower-) semicontinuous, for all upper- (resp. lower-) semicontinuous bounded function $\varphi$.
\end{Assumption}

A sufficient condition for Assumption \ref{K_continuity} to hold is provided in \cite{baradel2016optimal}, see the discussion after equation (3.6).

In order to ensure that $\mathcal{L}_{*}^{\mathbf{J}}$ is continuous for all $\mathbf{J} \in \mathcal{P}(\mathbf{K})$, we make the following assumption.

\begin{Assumption}\label{AL_continuity}
	We assume that
    \begin{itemize}
		\item The functions $F$ and $\mathfrak{M}_{\gamma}$ are continuous ;
        \item The stochastic kernel $\gamma \mapsto \Upsilon(\gamma, du)$ is continuous ;
        \item The map $(t, \lambda) \mapsto \Lambda(t, \lambda)$ is continuous.
    \end{itemize}
\end{Assumption}

\begin{lemma}\label{L_continuity}
	Assume that Assumption \ref{AL_continuity} holds. Then, for all $(c, m) \in \mathbf{C} \times \mathbf{M}$, with $\mathbf{J} := \{j \in \mathbf{K} : c_{j} \not= \partial\}$, and for all bounded function $\varphi \in \mathscr{C}^{1}_{\mathbf{J}}$, the operator $\mathcal{L}_{\star}^{\mathbf{J}}\varphi$ is continuous.
\end{lemma}

\begin{proof} Let $(c, m) \in \mathbf{C} \times \mathbf{M}$ and $\mathbf{J}$ defined as above. Recall that
    	\[
		\begin{aligned}
		 \mathcal{L}_{\star}^{\mathbf{J}}\varphi = &\; \partial_{t}\varphi +  \langle \mu + \overline{C}, D\varphi\rangle + \langle  \mathbf{1}_{\mathbf{J}}, D_{l}\varphi \rangle + \langle  h_{1}, D_{p}\varphi \rangle  \\
		&+ \mathbb{E}_{\overline{m}}\left[\int_{\mathbb{R}^{d}}\left[\mathcal{I}[\varphi, u] - \varphi\right]\Lambda(t, \lambda_{0})\Upsilon(\gamma_{0}, du)\right].
		\end{aligned}
        \]
        
        For the first line above, since all involved functions are continuous, the operator is continuous. For the second line, since $\varphi$ is bounded, one easily checks that the expected value with respect to $(\lambda, \gamma)$ is well defined and one can apply Fubini's theorem. This is rewritten:
        
        \[
        		\overline{\Lambda}(t, p)\int_{U^{\gamma}}\left[\int_{\mathbb{R}^{d}}\left[\mathcal{I}[\varphi, u] - \varphi\right]\Upsilon(\gamma, du)\right]dm^{\gamma}(\gamma)
        \]
        with $\overline{\Lambda}(t, p) := \int_{U^{\lambda}}\Lambda(t,  \lambda)dm^{\lambda}(\lambda)$ which is continuous, see \cite[Proposition 7.30 p145]{dimitri1996stochastic}.
        
        Now, remark that the function integrated through $\Upsilon(\gamma, du)$ with $\gamma \in U^{\gamma}$ fixed is continuous by definition. Since the stochastic kernel $\gamma \mapsto \Upsilon(\gamma, du)$ is assumed to be continuous,  we get again from \cite[Proposition 7.30 p145]{dimitri1996stochastic} that the function integrated through $m^{\gamma}$ is continuous and bounded. And then, the operator is continuous.
\end{proof}

\noindent We now assume that we have a comparison principle. A sufficient condition is provided in Proposition \ref{cat_comparaison} below.

\begin{Assumption} \label{hypcomparaison} Let $U$ (resp. $V$) be a upper- (resp. lower-) semicontinuous bounded viscosity sub- (resp. super-) solution of (\ref{interieur})-(\ref{bordT})-(\ref{bordl}). Assume further that  $U \leq V$ on $\mathbf{D}_{T}$. Then, $U \leq V$ on $\mathbf{D}$.

\begin{theorem}\label{viscosity}
	The function $\vr$ is the unique viscosity solution of (\ref{interieur})-(\ref{bordT})-(\ref{bordl}).
\end{theorem}

\end{Assumption}

\section{Viscosity solution properties}

This part is dedicated to the proof of the viscosity solution characterization of Theorem \ref{viscosity}. We start with the sub-solution property and continue with the super-solution property. The main difficulty relies on the fact that the filtration depends on the initial data. The results can be obtained along the lines of \cite{baradel2016optimal}. 

\subsection{Sub-solution property}

\begin{proposition}\label{subviscosity}
	The function $\vr$ is a viscosity sub-solution of (\ref{interieur})-(\ref{bordT})-(\ref{bordl}).
\end{proposition}

The proof of this proposition, as usual, relies on a dynamic programming principle. For this part, the dependency of the filtration on the initial data in not problematic as it only requires a conditioning argument. We have the following result:
\begin{proposition}\label{PPDsous}
	Fix $\mathbf{J} \in \mathcal{P}(\mathbf{K})$ and $(z, p, m) \in \mathbf{D}_{\mathbf{J}}$, and let $\theta$ be the first exit time of $(Z^{z, \phi^{0}}, P^{t, p})$ from a Borel set $B \subset \mathbf{D}_{\mathbf{J}}$ containing $(z, p, m)$ where $\phi^{0} \in \Phi^{z, m}$ is a control such that $\tau_{1}^{\phi^{0}} > t$. Then,
    \begin{equation}
    	\vr(z, p, m) \leq \sup_{\phi \in \Phi^{z, m}_{\geq t}} \mathbb{E}_{m}\left[\vr^{*}(Z_{\theta}^{z, \phi}, P_{\theta}^{t, p}, m)\mathbf{1}_{\{\theta < \tau_{1}^{\phi}\}} + \mathcal{K}^{\alpha_{1}^{\phi}}\vr^{*}(Z_{\tau_{1}^{\phi}-}^{z, \phi}, P_{\tau_{1}^{\phi}-}^{t, p}, m)\mathbf{1}_{\{\theta \geq \tau_{1}^{\phi}\}}	\right]
    \end{equation}
    in which $z := (t,x, c, l)$, $\Phi^{z, \overline{m}}_{\geq t} := \{\phi \in \Phi^{z, \overline{m}}_{\kappa} : \tau_{1}^{\phi} \geq t\}$.
\end{proposition}
\begin{proof}
	It suffices to follow the arguments of Proposition 4.2 in \cite{baradel2016optimal}.
\end{proof}

\noindent We now prove Proposition \ref{subviscosity}.

\begin{proof} 

 Since, for each $\mathbf{J} \in \mathcal{P}(\mathbf{K})$, the operator $\mathcal{L}_{\star}^{\mathbf{J}}$ is continuous, the proof of $(\ref{interieur})$ and $(\ref{bordT})$ can be obtained by using the same arguments as in Proposition 4.1 in \cite{baradel2016optimal}.
 
 To prove (\ref{bordl}), one can use the same arguments used in order to prove (\ref{bordT}).
\end{proof}

\subsection{Super-solution property}

Because of the non-trivial dependence of the filtration $\mathbb{F}^{z, m, \phi}$ with respect to the initial data, in order to prove the super-solution property associated with Theorem \ref{viscosity}, we shall use a discrete version of our impulse control problem, as in \cite{baradel2016optimal}. We shall show that the limit problem is a super-solution of (\ref{interieur})-(\ref{bordT})-(\ref{bordl}). Proposition \ref{subviscosity} and the comparison assumption will show that the limit problem is $\vr$.

We shall use a dynamic programing principle in some discrete form defined below.

\begin{proposition}\label{PPDsur}
	Fix $\mathbf{J} \in \mathcal{P}(\mathbf{K})$ and $(z, p, m) \in \mathbf{D}_{\mathbf{J}}$. Let $\Phi_{n}^{z, \overline{m}}$ be the subset of elements of $\Phi^{z, \overline{m}}_{\kappa}$ such that the stopping times $\tau_{i}^{\phi}, i \geq 1$ are valued in $\{t\} \cup \pi_{n} \cap [t, T]$ with $\pi_{n} := \{kT/2^{n} ; 0 \leq k \leq 2^n\}$. The corresponding value function is:
    
	\[
		\vr_{n}(z, p, m) := \sup_{\phi \in \Phi^{z,\overline{m}}_{n}} J(z, p, m ; \phi), \ \ (z, p, m) \in \mathbf{D}.
	\]
    
    Let $(\theta^{\phi}, \phi \in \Phi^{z, m}_{n})$ be such that $\theta^{\phi}$ is a $\mathbb{F}^{z, \overline{m}, \phi}$-stopping time valued in $\{t\} \cup \pi_{n} \cap [t, T]$. Then,
    
    \[
    		\vr_{n}(z, p, m) = \sup_{\phi \in \Phi^{z, \overline{m}}_{n}}\mathbb{E}_{\overline{m}}\left[\vr_{n}(Z_{\theta^{\phi}}^{z, \phi}, P_{\theta^{\phi}}^{t, p}, M_{\theta^{\phi}}^{z, m, \phi})\right].
    \]
    
    \end{proposition}
    
    \begin{proof}
    	It suffices to follow the arguments of Lemma 4.1, Proposition 4.3 and Corollary 4.1 in \cite{baradel2016optimal}.
    \end{proof}

\noindent We now consider the limit $n \rightarrow +\infty$. Let us set, for $(z, p, m) \in \mathbf{D}$,

	\[
    	\vr_{\circ}(z, p, m) := \liminf_{(z', p', m', n) \rightarrow (z, p, m, +\infty)}\vr_{n}(z', p', m').
    \]

\begin{proposition}
	The function $\vr_{\circ}$ is a viscosity super-solution of (\ref{interieur})-(\ref{bordT})-(\ref{bordl}).
\end{proposition}

\begin{proof} 

	The equations (\ref{interieur}) and (\ref{bordT}) can be obtained by using Proposition \ref{PPDsur} and following the arguments in the proof of Proposition 4.4 in \cite{baradel2016optimal}. We now prove the boundary condition (\ref{bordl}).

    Step 1. Fix $\mathbf{J} \subset \mathcal{P}(\mathbf{K})$ and $(z, p, m) \in \mathbf{D}_{\mathbf{J}}$. 
    
    Let $n_{k} \rightarrow +\infty$ and $(z_{k}, p_{k}, m_{k}) \rightarrow (z, p, m)$ such that $\vr_{n_{k}}(z_{k}, p_{k}, m_{k}) \rightarrow \vr_{\circ}(z, p, m)$. Let $k_{\circ} \geq 1$ and define the lower semi-continuous function $\varphi_{k_{\circ}}$ as in the proof of Proposition 4.4 in \cite{baradel2016optimal}. Then, from Proposition \ref{PPDsur}, with $\phi^{0} \in \Phi^{t, x, m}$  a control such that $\tau_{1}^{\phi^{0}} > T$, we get for $k \geq k_{\circ}$
    
    \[
    	\vr_{n_{k}}(z_{k}, p_{k}, m_{k})  \geq \mathbb{E}_{\overline{m}}\left[\varphi_{k_{\circ}}(Z_{\theta^{\phi^{0}}}^{z_{k}, \phi^{0}}, P_{\theta^{\phi^{0}}}^{t_{k}, p_{k}}, M_{\theta^{\phi^{0}}}^{z_{k}, m_{k}, \phi^{0}})\right].
    \]
    
    Then, $k \rightarrow +\infty$ leads to
    
    \[
    	\vr_{\circ}(z, p, m)  \geq \mathbb{E}_{\overline{m}}\left[\varphi_{k_{\circ}}(Z_{\theta^{\phi^{0}}}^{z, \phi_{0}}, P_{\theta^{\phi^{0}}}^{t, p}, M_{\theta^{\phi^{0}}}^{z, m, \phi^{0}})\right]
    \]
    
    and, again from the proof of Proposition 4.4 in \cite{baradel2016optimal}, we get that $\lim_{k_{\circ} \rightarrow +\infty}\varphi_{k_{\circ}} \geq \vr_{\circ}$. By Fatou's lemma we have 

    \[
    	\vr_{\circ}(z, p, m)  \geq \mathbb{E}_{\overline{m}}\left[\vr_{\circ}(Z_{\theta^{\phi^{0}}}^{z, \phi_{0}}, P_{\theta^{\phi^{0}}}^{t, p}, M_{\theta^{\phi^{0}}}^{z, m, \phi^{0}})\right].
    \]

    Step 2. Now fix $\mathbf{J}' \subset \mathbf{J}$ and $(z_{\circ}, p_{\circ}, m_{\circ}) \in \mathbf{D}_{\mathbf{J}}$. Let $k \rightarrow +\infty$ and $(z_{k}, p_{k}, m_{k}) \rightarrow (t_{\circ}, x_{\circ}, c_{\circ}, \mathfrak{L}^{\mathbf{J}'}_{\mathbf{J}}(l_{\circ}), p_{\circ}, m_{\circ})$ such that 
    \[
    \vr_{\circ}(z_{k}, p_{k}, m_{k}) \rightarrow \liminf_{(z, p, m) \rightarrow (t_{\circ}, x_{\circ}, c_{\circ}, \mathfrak{L}^{\mathbf{J}'}_{\mathbf{J}}(l_{\circ}), p_{\circ}, m_{\circ})}\vr_{\circ}(z, p, m). 
    \]
    
    We introduce $h_{k} := \mathfrak{L}^{\mathbf{J}'}_{\mathbf{J}}(l_{\circ}) - l_{k}$. Then, for $k_{\circ}$ large enough, we can find $\varepsilon > 0$ such that $\sup_{k \geq k_{\circ}}\max_{j \in \mathbf{J}'} h_{k}^{j} < \varepsilon < \inf_{k \geq k_{\circ}}\max_{j \in \mathbf{J} \backslash \mathbf{J}'}(\ell - l_{k}^{j})$. Then, for $k \geq k_{\circ}$,
   		 \[
    	\vr_{\circ}(t_{k}, x_{k}, c_{k}, \mathfrak{L}^{\mathbf{J}'}_{\mathbf{J}}(l_{\circ}) - h_{k}, p_{k}, m_{k}) \geq \mathbb{E}\left[\vr_{\circ}(Z^{z_{k}, \phi_{0}}_{t+\varepsilon}, P^{t_{k}, p_{k}}_{t+\varepsilon}, M^{z_{k}, m_{k}, \phi_{0}}_{t+\varepsilon})\right].
		\]
        
        Now, we send $k \rightarrow +\infty$, since the functions in the diffusion are Lipschitz, using Fatou's lemma leads to

   		 \[
    	\lim_{k \rightarrow +\infty}\vr_{\circ}(t_{k}, x_{k}, c_{k}, \mathfrak{L}^{\mathbf{J}'}_{\mathbf{J}}(l_{\circ}) - h_{k}, p_{k}, m_{k})  \geq \mathbb{E}\left[\vr_{\circ}(Z^{z, \phi_{0}}_{t+\varepsilon}, P^{t, p}_{t+\varepsilon}, M^{z, m, \phi_{0}}_{t+\varepsilon})\right].
		\]

Since, under the control $\phi^{0}$, the processes $X$, $P$ and $M$ are driven here by the random Poisson measure with finite activity, they satisfy the stochastic continuity property. Moreover, since the probability of observing a jump decreases to 0 when $\varepsilon \rightarrow 0$, one easily shows that, 

   		 \[
    	\lim_{k \rightarrow +\infty}\vr_{\circ}(t_{k}, x_{k}, c_{k}, \mathfrak{L}^{\mathbf{J}'}_{\mathbf{J}}(l_{\circ}) - h_{k}, p_{k}, m_{k})  \geq \vr_{\circ}(t_{\circ}, x_{\circ}, \mathfrak{C}_{-}^{\ell}[c_{\circ},\mathfrak{L}^{\mathbf{J}'}_{\mathbf{J}}(l_{\circ})], p_{\circ}, m_{\circ}),
		\]
        by using the fact that $\vr_{\circ}$ is bounded and the definition of the process $C$ and $L$ after the end of one or several CAT bonds.

        Step 3. In order to show the second inequality, repeat Step 1. and Step 2. using, instead of $\phi^{0}$, a control $\phi^{a} \in \Phi^{z,\overline{m}}_{\kappa}$ such that $\{\tau^{\phi^{a}}_{1} = t, \alpha_{1}^{\phi^{a}} = a, \tau_{2}^{\phi^{a}} > T\}$ holds with probability one.

\end{proof}

\noindent We now prove   Theorem \ref{viscosity}.

\begin{proof}[Proof of Theorem \ref{viscosity}.]
	We already know that $\vr^{*}$ and $\vr_{\circ}$ are respectively a bounded sub- and super-solution of (\ref{interieur})-(\ref{bordT})-(\ref{bordl}). Then, under Assumption \ref{hypcomparaison}, $\vr^{*} \leq \vr_{\circ}$. Moreover, by construction, $\vr_{\circ} \leq \vr \leq \vr^{*}$. Then, $\vr$ is continuous and the unique solution of (\ref{interieur})-(\ref{bordT})-(\ref{bordl}).
\end{proof}

        \begin{remark}
        	If we denote by $\mathcal{S}_{\mathbf{K}}$ the set of permutation of $\{1 \leq k \leq \kappa\}$, then, by symmetry, 
            \[
            	\vr(z, p, m) = \vr(t, x, (c, l) \circ \Sigma, p, m) 
            \]
            for each $\Sigma \in \mathcal{S}_{\mathbf{K}}$, $(z, p, m) \in \mathbf{D}$.
            From a numerical point of view, this make it easier to compute the value function.
        \end{remark}

\section{A sufficient condition for the comparison}

In this section, we provide a sufficient condition for Assumption \ref{hypcomparaison} to hold.

\begin{proposition}\label{cat_comparaison}
	Assumption \ref{hypcomparaison} holds whenever there exists a function $\Psi$ on $[0, T) \times \mathbb{R}^{d} \times \mathbf{CL} \times \mathbf{P} \times \mathbf{M}$ such that, for each $\mathbf{J} \in \mathcal{P}(\mathbf{K})$,
    \begin{itemize}
    	\item[(i)] $(t, x, l, p) \mapsto \Psi(t, x, c, l, p, m) \in C^{1,1,1,1}([0, T) \times \mathbb{R}^{d} \times [0, \ell) \times \mathbf{P})$ for all $(c, m) \in \mathbf{C} \times \mathbf{M}$,
        \item[(ii)] $\varrho\Psi \geq \mathcal{L}_{*}^{\mathbf{J}}\Psi$ on $\mathbf{D}_{\mathbf{J}}$ for some $\varrho > 0$,
        \item[(iii)] $\Psi - \mathcal{K}\Psi \geq \delta$ on $\mathbf{D}_{\mathbf{J}}$ for some $\delta > 0$,
        \item[(iv)] $\Psi \geq max(\mathcal{K}\tilde{g}, \tilde{g})$ on $\mathbb{R}^{d} \times \mathbf{CL}_{\mathbf{J}} \times \mathbf{P} \times \mathbf{M}$ with $\tilde{g}(t, \cdot) := e^{\varrho t}g(t, \cdot)$ and $\varrho$ is defined in \textit{(ii)},
        \item[(v)] $\liminf_{l' \rightarrow \mathfrak{L}^{\mathbf{J}'}_{\mathbf{J}}(l)}\Psi(\cdot, c, l', \cdot) - \Psi(\cdot, \mathfrak{C}^{\ell}_{-}(c, \mathfrak{L}^{\mathbf{J}'}_{\mathbf{J}}(l)), .)\geq  0$ for all $\mathbf{J}' \subset \mathbf{J}$,
        \item[(vi)] $\Psi^{-} \leq \overline{\Psi}(x) = o(\|x\|^{2})$ as $\|x\|^{2} \rightarrow +\infty$ for some $\overline{\Psi} : \mathbb{R}^{d} \rightarrow \mathbb{R}$. 
        
    \end{itemize}
    
 \end{proposition}
 
 \begin{proof}
 
 Step 1. As usual, we shall argue by contradiction. We assume that there exists some $\mathbf{J}_{0} \in \mathcal{P}(\mathbf{K})$ and some $(z_{0}, p_{0}, m_{0}) \in \mathbf{D}_{\mathbf{J}}$ such that $(U - V)(z_{0}, p_{0}, m_{0}) > 0$, in which $U$ is a sub-solution of (\ref{interieur})-(\ref{bordT})-(\ref{bordl}) and $V$ is a super-solution of (\ref{interieur})-(\ref{bordT})-(\ref{bordl}). Recall the definition of $\Psi, \varrho$ and $\tilde{g}$ in Proposition \ref{cat_comparaison}. We set $\tilde{u}(t, .) = e^{\varrho t}U(t, .)$ and $\tilde{v}(t, .) = e^{\varrho t}V(t, .)$ for all $(t, .) \in \mathbf{D}_{\mathbf{J}}$ for all $\mathbf{J} \in \mathcal{P}(\mathbf{K})$. Then, there exists $\lambda \in (0, 1)$ such that
 	\begin{equation}\label{comparison_absurde}
    	(\tilde{u} - \tilde{v}^{\lambda})(z_{0}, p_{0}, m_{0}) > 0,
    \end{equation}
    in which $\tilde{v}^{\lambda} := (1-\lambda)\tilde{v} + \lambda\Psi$. Note that $\tilde{u}$ and $\tilde{v}$ are sub and super-solution on $\mathbf{D}_{\mathbf{J}}$ of
    \[
    	\min\left\{\varrho \varphi - \mathcal{L}_{*}^{\mathbf{J}}\varphi, \varphi - \mathcal{K}\varphi\right\} = 0
    \]
   for each $\mathbf{J} \in \mathcal{P}(\mathbf{K})$, 
 with the boundary conditions 
    \begin{equation}\label{comparaison_boundaryT}
    	\mathbf{1}_{\{\mathbf{J} = \mathbf{K}\}}(\varphi(T,\cdot) - \tilde{g}) + \mathbf{1}_{\{\mathbf{J} \not= \mathbf{K}\}}\min\left\{\varphi(T,\cdot) - \tilde{g}, \varphi(T,\cdot) - \mathcal{K}\tilde{g}\right\} = 0,
    \end{equation}
    and  
    \begin{equation}\label{comparaison_boundaryL}
    \ \lim_{l' \rightarrow \mathfrak{L}^{\mathbf{J}'}_{\mathbf{J}}(l)} \varphi (., c, l', .) = \varphi (., \mathfrak{C}_{-}^{\ell}[c, \mathfrak{L}^{\mathbf{J}'}_{\mathbf{J}}(l)], .) \; \ \	\forall \mathbf{J}' \subset \mathbf{J}, \;(c,l)\in \mathbf{CL}_{\mathbf{J}}
    \end{equation}
 Step 2. Let $d_{\mathbf{M}}$ be a metric on $\mathbf{M}$ compatible with the weak topology. For $(t, x, y, c, l, p, q, m) \in \mathbf{D}' := [0, T] \times \mathbb{R}^{d} \times \mathbb{R}^{d} \times \mathbf{C}\mathbf{L} \times \mathbf{P}^{2} \times \mathbf{M}$, we set :
 	\begin{equation}\label{comparison_defgamma}
    	\begin{aligned}
    	\Gamma_{\varepsilon}(t, x, y, c, l, p, q, m) :=  &\tilde{u}(t, x, c, l, p, m) - \tilde{v}^{\lambda}(t, y, c, l, q, m) \\ 
        &- \varepsilon\left(\|x\|^2 + \|y\|^2 + \mathfrak{d}(c, l) + \|p\|^2 + \|q\|^2+ d_{\mathbf{M}}(m) \right)
        \end{aligned}
    \end{equation}
 with $\varepsilon > 0$ small enough such that $\Gamma_{\varepsilon}(t_{0}, x_{0}, x_{0}, c_{0}, l_{0}, p_{0}, p_{0}, m_{0}) > 0$. Although $[0, \ell)$ is not closed, note that the supremum is achieved for some $\mathbf{J}_{\varepsilon} \in \mathcal{P}(\mathbf{K})$ by some $(t_{\varepsilon}, x_{\varepsilon}, y_{\varepsilon},$ $ c_{\varepsilon}, l_{\varepsilon}, p_{\varepsilon}, q_{\varepsilon}, m_{\varepsilon}) \in \mathbf{D}_{\mathbf{J}_{\varepsilon}}$. This follows from the upper-semicontinuity of $\Gamma_{\varepsilon}$, the fact that $\tilde{u}, -\tilde{v}$ and $-\Psi$ are bounded from above, and by the fact that
 	\[
    	\limsup_{l' \rightarrow \mathfrak{L}^{k}_{\mathbf{J}}(l)}(\tilde{u} - \tilde{v}^{\lambda})(., c, l', .) \leq (\tilde{u} - \tilde{v}^{\lambda})(., \mathfrak{C}^{\ell}_{-}(c, \mathfrak{L}^{k}_{\mathbf{J}}(l)), .).
    \]
            For $(t, x, y, c, l, p, q, m) \in \mathbf{D}'$, we set
	\[
    	\Theta_{\varepsilon}^{n}(t, x, y, c, l, p, q, m) = \Gamma_{\varepsilon}(t, x, y, c, l, p, q, m) - n\left(\|x-y\|^{2} + \|p-q\|^{2}\right).
    \]
    
    Again, there is $(t_{n}^{\varepsilon}, x_{n}^{\varepsilon}, y_{n}^{\varepsilon}, c_{n}^{\varepsilon}, l_{n}^{\varepsilon}, p_{n}^{\varepsilon}, q_{n}^{\varepsilon}, m_{n}^{\varepsilon}) \in \mathbf{D}'$ such that
    \[
    	\sup_{\mathbf{D}'} \Theta_{\varepsilon}^{n} = \Theta_{\varepsilon}^{n}(t_{n}^{\varepsilon}, x_{n}^{\varepsilon}, y_{n}^{\varepsilon}, c_{n}^{\varepsilon}, l_{n}^{\varepsilon}, p_{n}^{\varepsilon}, q_{n}^{\varepsilon}, m_{n}^{\varepsilon}).
    \]
    It is standard to show that, after possibly considering a subsequence,
    \begin{equation}\label{comparison_convergence}
    	\begin{aligned}
        	&(t_{n}^{\varepsilon}, x_{n}^{\varepsilon}, y_{n}^{\varepsilon}, c_{n}^{\varepsilon}, l_{n}^{\varepsilon}, p_{n}^{\varepsilon}, q_{n}^{\varepsilon}, m_{n}^{\varepsilon}) \rightarrow (\hat{t}_{\varepsilon}, \hat{x}_{\varepsilon}, \hat{y}_{\varepsilon}, \hat{c}_{\varepsilon}, \hat{l}_{\varepsilon}, \hat{p}_{\varepsilon}, \hat{q}_{\varepsilon}, \hat{m}_{\varepsilon}) \in \mathbf{D}', \\
            & n\left(\|x_{n}^{\varepsilon}-y_{n}^{\varepsilon}\|^{2} + \|p_{n}^{\varepsilon}-q_{n}^{\varepsilon}\|^{2}\right) \rightarrow 0, \text{ and}\\
            & \Theta_{\varepsilon}^{n}(t_{n}^{\varepsilon}, x_{n}^{\varepsilon}, y_{n}^{\varepsilon}, c_{n}^{\varepsilon}, l_{n}^{\varepsilon}, p_{n}^{\varepsilon}, q_{n}^{\varepsilon}, m_{n}^{\varepsilon}) \rightarrow \Gamma_{\varepsilon}(\hat{t}_{\varepsilon}, \hat{x}_{\varepsilon}, \hat{y}_{\varepsilon}, \hat{c}_{\varepsilon}, \hat{l}_{\varepsilon}, \hat{p}_{\varepsilon}, \hat{q}_{\varepsilon}, \hat{m}_{\varepsilon}) = \Gamma_{\varepsilon}(t_{\varepsilon}, x_{\varepsilon}, y_{\varepsilon}, c_{\varepsilon}, l_{\varepsilon}, p_{\varepsilon}, q_{\varepsilon}, m_{\varepsilon}),
        \end{aligned}
    \end{equation}
see e.g. \cite[Lemma 3.1]{crandall1992user}. Moreover, up to a subsequence, there exists $n_{0} \in \mathbb{N}$, such that, for all $n \geq n_{0}$, $(t_{n}^{\varepsilon}, x_{n}^{\varepsilon}, c_{n}^{\varepsilon}, l_{n}^{\varepsilon}, p_{n}^{\varepsilon}, m_{n}^{\varepsilon}) \in D_{\mathbf{J}_{\varepsilon}}$ and $(t_{n}^{\varepsilon}, y_{n}^{\varepsilon}, c_{n}^{\varepsilon}, l_{n}^{\varepsilon}, q_{n}^{\varepsilon}, m_{n}^{\varepsilon}) \in D_{\mathbf{J}_{\varepsilon}}$.

Step 3. We first assume that, up to a subsequence, $(\tilde{u} - \mathcal{K}\tilde{u})(t_{n}^{\varepsilon}, x_{n}^{\varepsilon}, c_{n}^{\varepsilon}, l_{n}^{\varepsilon}, p_{n}^{\varepsilon}, m_{n}^{\varepsilon}) \leq 0$, for $n \geq 1$. Then, it follows from the supersolution property of $\tilde{v}$ and Condition (iii)  that
	\[
    	\begin{aligned}
    	&\tilde{u}(t_{n}^{\varepsilon}, x_{n}^{\varepsilon}, c_{n}^{\varepsilon}, l_{n}^{\varepsilon}, p_{n}^{\varepsilon}, m_{n}^{\varepsilon}) - \tilde{v}^{\lambda}(t_{n}^{\varepsilon}, y_{n}^{\varepsilon}, c_{n}^{\varepsilon}, l_{n}^{\varepsilon}, q_{n}^{\varepsilon}, m_{n}^{\varepsilon}) \leq \\
         & \mathcal{K}\tilde{u}(t_{n}^{\varepsilon}, x_{n}^{\varepsilon}, c_{n}^{\varepsilon}, l_{n}^{\varepsilon}, p_{n}^{\varepsilon}, m_{n}^{\varepsilon}) - \mathcal{K}\tilde{v}^{\lambda}(t_{n}^{\varepsilon}, y_{n}^{\varepsilon}, c_{n}^{\varepsilon}, l_{n}^{\varepsilon}, q_{n}^{\varepsilon}, m_{n}^{\varepsilon}) - \lambda\delta.
        \end{aligned}
    \]
    Passing to the $\limsup$ and using (\ref{comparison_convergence}) and (\ref{K_continuity}), we obtain 
    \[
    	(\tilde{u} - \tilde{v}^{\lambda})(\hat{t}_{\varepsilon}, \hat{x}_{\varepsilon}, \hat{c}_{\varepsilon}, \hat{l}_{\varepsilon}, \hat{p}_{\varepsilon}, \hat{m}_{\varepsilon}) + \lambda\delta \leq \mathcal{K}(\tilde{u} - \tilde{v}^{\lambda})(\hat{t}_{\varepsilon}, \hat{x}_{\varepsilon}, \hat{c}_{\varepsilon}, \hat{l}_{\varepsilon}, \hat{p}_{\varepsilon}, \hat{m}_{\varepsilon})
        \]
        Now let us observe that
        \begin{equation}\label{comparison_gamma}
        \begin{aligned}
        	\sup_{\mathbf{D}}(\tilde{u} - \tilde{v}^{\lambda}) &= \lim_{\varepsilon \rightarrow 0} \ \sup_{(t, x, c, l, p, m) \in \mathbf{D}} \Gamma_{\varepsilon}(t, x, x, c, l, p, p, m) \\
            &= \lim_{\varepsilon \rightarrow 0} \Gamma_{\varepsilon}(t_{\varepsilon}, x_{\varepsilon}, x_{\varepsilon}, c_{\varepsilon}, l_{\varepsilon}, p_{\varepsilon}, p_{\varepsilon}, m_{\varepsilon}) \\ 
            &= \lim_{\varepsilon \rightarrow 0} \Gamma_{\varepsilon}(\hat{t}_{\varepsilon}, \hat{x}_{\varepsilon}, \hat{x}_{\varepsilon}, \hat{c}_{\varepsilon}, \hat{l}_{\varepsilon}, \hat{p}_{\varepsilon}, \hat{p}_{\varepsilon}, \hat{m}_{\varepsilon}),
            \end{aligned}
        \end{equation}
  in which the last identity follows from (\ref{comparison_convergence}). Combined with the above inequality, this shows that $\sup_{\mathbf{D}}(\tilde{u} - \tilde{v}^{\lambda}) + \lambda\delta \leq \limsup_{\varepsilon \rightarrow 0}\mathcal{K}(\tilde{u} - \tilde{v}^{\lambda})(\hat{t}_{\varepsilon}, \hat{x}_{\varepsilon}, \hat{c}_{\varepsilon}, \hat{l}_{\varepsilon}, \hat{p}_{\varepsilon}, \hat{m}_{\varepsilon})$, which leads to a contradiction for $\varepsilon$ small enough.

Step 4. We now show that there is a subsequence such that $t_{n}^{\varepsilon} < T$ for all $n \geq 1$. If not, one can assume that $t_{n}^{\varepsilon} = T$. If, up to a subsequence, one can have $\tilde{u}(T, x_{n}^{\varepsilon}, c_{n}^{\varepsilon}, l_{n}^{\varepsilon}, p_{n}^{\varepsilon}, m_{n}^{\varepsilon})) \leq \tilde{g}(T, x_{n}^{\varepsilon}, c_{n}^{\varepsilon}, l_{n}^{\varepsilon}, p_{n}^{\varepsilon}, m_{n}^{\varepsilon})$, then it follows from (\ref{comparaison_boundaryT}) and Condition (iv) that,
			\[
            	\tilde{u}(T, x_{n}^{\varepsilon}, c_{n}^{\varepsilon}, l_{n}^{\varepsilon}, p_{n}^{\varepsilon}, m_{n}^{\varepsilon})) - \tilde{v}^{\lambda}(T, y_{n}^{\varepsilon}, c_{n}^{\varepsilon}, l_{n}^{\varepsilon}, q_{n}^{\varepsilon}, m_{n}^{\varepsilon}) \leq \tilde{g}(T, x_{n}^{\varepsilon}, c_{n}^{\varepsilon}, l_{n}^{\varepsilon}, p_{n}^{\varepsilon}, m_{n}^{\varepsilon}) - \tilde{g}(T, y_{n}^{\varepsilon}, c_{n}^{\varepsilon}, l_{n}^{\varepsilon}, q_{n}^{\varepsilon}, m_{n}^{\varepsilon}).
               \]
               Hence,
               	\[
                	\Gamma_{\varepsilon}(t_{\varepsilon}, x_{\varepsilon}, x_{\varepsilon}, c_{\varepsilon}, l_{\varepsilon}, p_{\varepsilon}, p_{\varepsilon}, m_{\varepsilon}) \leq \tilde{g}(T, x_{n}^{\varepsilon}, c_{n}^{\varepsilon}, l_{n}^{\varepsilon}, p_{n}^{\varepsilon}, m_{n}^{\varepsilon}) - \tilde{g}(T, y_{n}^{\varepsilon}, c_{n}^{\varepsilon}, l_{n}^{\varepsilon}, q_{n}^{\varepsilon}, m_{n}^{\varepsilon}),
                    \]
                    and (\ref{comparison_convergence}) with (\ref{comparison_gamma}) leads to $\sup_{\mathbf{D}}(\tilde{u} - \tilde{v}^{\lambda}) \leq 0$, a contradiction. If, up to a subsequence, $\tilde{u}(T, x_{n}^{\varepsilon}, c_{n}^{\varepsilon}, l_{n}^{\varepsilon}, p_{n}^{\varepsilon}, m_{n}^{\varepsilon}) \leq \mathcal{K}\tilde{g}(T,  x_{n}^{\varepsilon}, c_{n}^{\varepsilon}, l_{n}^{\varepsilon}, p_{n}^{\varepsilon}, m_{n}^{\varepsilon})$, by (\ref{comparaison_boundaryT}) and Condition (iv),
                    \[
                    \tilde{u}(T, x_{n}^{\varepsilon}, c_{n}^{\varepsilon}, l_{n}^{\varepsilon}, p_{n}^{\varepsilon}, m_{n}^{\varepsilon}) - \tilde{v}^{\lambda}(T, y_{n}^{\varepsilon}, c_{n}^{\varepsilon}, l_{n}^{\varepsilon}, q_{n}^{\varepsilon}, m_{n}^{\varepsilon}) \leq \mathcal{K}\tilde{g}(T, x_{n}^{\varepsilon}, c_{n}^{\varepsilon}, l_{n}^{\varepsilon}, p_{n}^{\varepsilon}, m_{n}^{\varepsilon}) - \mathcal{K}\tilde{g}(T, y_{n}^{\varepsilon}, c_{n}^{\varepsilon}, l_{n}^{\varepsilon}, q_{n}^{\varepsilon}, m_{n}^{\varepsilon}).
                    \]
                    Hence,
                    \[
                    	\Gamma_{\varepsilon}(t_{\varepsilon}, x_{\varepsilon}, x_{\varepsilon}, c_{\varepsilon}, l_{\varepsilon}, p_{\varepsilon}, p_{\varepsilon}, m_{\varepsilon}) \leq \mathcal{K}\tilde{g}(T, x_{n}^{\varepsilon}, c_{n}^{\varepsilon}, l_{n}^{\varepsilon}, p_{n}^{\varepsilon}, m_{n}^{\varepsilon}) - \mathcal{K}\tilde{g}(T, y_{n}^{\varepsilon}, c_{n}^{\varepsilon}, l_{n}^{\varepsilon}, q_{n}^{\varepsilon}, m_{n}^{\varepsilon}),
					\]
                        and combining Assumption \ref{K_continuity} with (\ref{comparison_convergence}) and (\ref{comparison_gamma}) leads to $\sup_{\mathbf{D}}(\tilde{u} - \tilde{v}^{\lambda}) \leq 0$, the same contradiction.

Step 5. In view of step 2, 3, 4, one can assume that $t_{n}^{\varepsilon} < T$\bru{,} $(\tilde{u} - \mathcal{K}\tilde{u})(t_{n}^{\varepsilon}, x_{n}^{\varepsilon}, c_{n}^{\varepsilon}, l_{n}^{\varepsilon}, p_{n}^{\varepsilon}, m_{n}^{\varepsilon}) > 0$ and $(c_{n}^{\varepsilon}, l_{n}^{\varepsilon}) \in \mathbf{C}\mathbf{L}_{\mathbf{J}^{\varepsilon}}$ for all $n \geq 1$. Using Ishii's Lemma and following standard arguments, see Theorem 8.3 and the discussion after Theorem 3.2 in \cite{crandall1992user}, we deduce from the sub- and supersolution viscosity solutions property of $\tilde{u}$ and $\tilde{v}^{\lambda}$, and the Lipschitz continuity assumptions on $\mu, \sigma$ and $\beta$, that
	\[
    	\begin{aligned}
    	&\varrho\left(\tilde{u}(t_{n}^{\varepsilon}, x_{n}^{\varepsilon}, c_{n}^{\varepsilon}, l_{n}^{\varepsilon}, p_{n}^{\varepsilon}, m_{n}^{\varepsilon}) - \tilde{v}^{\lambda}(t_{n}^{\varepsilon}, y_{n}^{\varepsilon}, c_{n}^{\varepsilon}, l_{n}^{\varepsilon}, q_{n}^{\varepsilon}, m_{n}^{\varepsilon})\right) \leq \\ &C\left(n(\|x_{n}^{\varepsilon} - y_{n}^{\varepsilon}\|^{2} + \|p_{n}^{\varepsilon} - q_{n}^{\varepsilon}\|^{2}) + \varepsilon(1+\|x_{n}^{\varepsilon}\|^{2}+\|y_{n}^{\varepsilon}\|^{2}+\|p_{n}^{\varepsilon}\|^{2}+\|q_{n}^{\varepsilon}\|^{2})\right),
        \end{aligned}
    \]
for some $C > 0$, independent of $n$ and $\varepsilon$. In view of (\ref{comparison_defgamma}) and (\ref{comparison_convergence}), we get
		\begin{equation}\label{comparison_ishii}
        	\varrho\Gamma_{\varepsilon}(\hat{t}_{\varepsilon}, \hat{x}_{\varepsilon}, \hat{x}_{\varepsilon}, \hat{c}_{\varepsilon}, \hat{l}_{\varepsilon}, \hat{p}_{\varepsilon}, \hat{p}_{\varepsilon}, \hat{m}_{\varepsilon}) \leq 2 C \varepsilon\left(1+\|\hat{x}_{\varepsilon}\|^{2}+\|\hat{p}_{\varepsilon}\|^{2}\right).
        \end{equation}
We shall prove in next step that the right-hand side of (\ref{comparison_ishii}) goes to 0 as $\varepsilon \rightarrow 0$, up to a subsequence. Combined with (\ref{comparison_gamma}), this leads to a contradiction of (\ref{comparison_absurde}).

Step 6. We conclude the proof by proving the claim used above. First note that we can always construct a sequence $(\tilde{t}_{\varepsilon}, \tilde{x}_{\varepsilon}, \tilde{c}_{\varepsilon}, \tilde{l}_{\varepsilon}, \tilde{p}_{\varepsilon}, \tilde{m}_{\varepsilon})_{\varepsilon})_{\varepsilon > 0}$ such that
	\[
    \begin{aligned}
    	\Gamma_{\varepsilon}(\tilde{t}_{\varepsilon}, \tilde{x}_{\varepsilon}, \tilde{x}_{\varepsilon}, \tilde{c}_{\varepsilon}, \tilde{l}_{\varepsilon}, \tilde{p}_{\varepsilon}, \tilde{p}_{\varepsilon}, \tilde{m}_{\varepsilon}) \rightarrow \sup_{\mathbf{D}}(\tilde{u} - \tilde{v}^{\lambda}) \ \text{ and } \\ 
        \varepsilon\left(\|\tilde{x}_{\varepsilon}\|^{2} + \mathfrak{d}(\tilde{c}_{\varepsilon}, \tilde{l}_{\varepsilon}) + \|\tilde{p}_{\varepsilon}\|^{2} + d_{\mathbf{M}}(\tilde{m}_{\varepsilon})\right) \rightarrow 0 \text{ \ \ as } \varepsilon \rightarrow 0.
        \end{aligned}
        \]
By (\ref{comparison_convergence}), $\Gamma_{\varepsilon}(\tilde{t}_{\varepsilon}, \tilde{x}_{\varepsilon}, \tilde{x}_{\varepsilon}, \tilde{c}_{\varepsilon}, \tilde{l}_{\varepsilon}, \tilde{p}_{\varepsilon}, \tilde{p}_{\varepsilon}, \tilde{m}_{\varepsilon}) \leq \Gamma_{\varepsilon}(\hat{t}_{\varepsilon}, \hat{x}_{\varepsilon}, \hat{x}_{\varepsilon}, \hat{c}_{\varepsilon}, \hat{l}_{\varepsilon}, \hat{p}_{\varepsilon}, \hat{p}_{\varepsilon}, \hat{m}_{\varepsilon})$. Hence, $\sup_{\mathbf{D}}(\tilde{u} - \tilde{v}^{\lambda}) \leq \sup_{\mathbf{D}}(\tilde{u} - \tilde{v}^{\lambda}) - 2\liminf_{\varepsilon \rightarrow 0}\varepsilon\left(\|\hat{x}_{\varepsilon}\|^{2}+\|\hat{p}_{\varepsilon}\|^{2}\right)$.

\end{proof}

\section{Numerical Scheme}\label{cat_numerique}

We let $h_{\circ}$ be a time-discretization step such that both $T/h_{\circ}$ and $\ell/h_{\circ}$ are an integer. In order to ensure the existence of such a $h_{\circ}$, we shall assume that $(T/\ell) \in \mathbb{Q}_{+}^{*}$ which does not appear as a restriction from a practical point of view. We set $\mathbf{T}^{h_{\circ}} := \{t_{i}^{h_{\circ}} := ih_{\circ}, i \leq T/h_{\circ}\}$ and, for $\mathbf{J} \in \mathcal{P}(\mathbf{K})$, we set $\mathbf{L}_{\mathbf{J}}^{h_{\circ}} = \prod_{j=1}^{\kappa}(\partial \mathbf{1}_{\mathbf{J}^{c}}(j) + \mathbf{L}^{h_{\circ}}\mathbf{1}_{\mathbf{J}}(j))$ in which $\mathbf{L}^{h_{\circ}} := \{l^{h_{\circ}}_{i} := ih_{\circ}, i < \ell/h_{\circ}\}$.

The space $\mathbb{R}^{d}$ is discretized with a space step $h_{\star}$ on a rectangle $[-c, c]^{d}$ containing $N_{h_{\star}}^{x}$ points on each direction. The corresponding set is denoted by $\mathbf{X}_{c}^{h_{\star}}$. Recall that $\mathbf{P}$ is a subset of $\mathbb{R}^{d}$. We again discretise $\mathbb{R}^{d}$ with the same step space $h_{\star}$ on a rectangle $[-c, c]^{d}$ containing $N_{h_{\star}}^{p}$ points. The corresponding set is denoted by $\mathbf{P}_{c}^{\circ, h_{\star}}$, thus, the discretization of $\mathbf{P}$ is $\mathbf{P}_{c}^{h_{\star}} := \mathbf{P}_{c}^{\circ, h_{\star}} \cap \mathbf{P}$.

We set $h = (h_{\circ}, h_{\star})$. The first order derivatives $(\partial_{t}\varphi)$, $(\partial_{x_{i}}\varphi)_{i \leq d}$, $(\partial_{l_{i}}\varphi)_{i \leq \kappa}$ and $(\partial_{p_{i}}\varphi)_{i \leq d}$ are approximated by using the standard up-wind approximations:

\[
\begin{aligned}
    \Delta_{i}^{h_{\circ}, t} \varphi(z, p, m) &:= h_{\circ}^{-1}\left(\varphi(t + h_{\circ}, \cdot) - \varphi\right) \\
    \Delta_{i}^{h_{\star}, x} \varphi(z, p, m) &:= \left\{
    \begin{array}{ll}
        h_{\star}^{-1}\left(\varphi(\cdot, x + e_{i}h_{\star}, \cdot) - \varphi\right) & \mbox{if } \mu_{i} + \overline{C} \geq 0 \\
        h_{\star}^{-1}\left(\varphi - \varphi(\cdot, x - e_{i}h_{\star}, \cdot)\right) & \mbox{else}
    \end{array}
\right. \\
	\Delta_{i}^{h_{\star}, \ell} \varphi(z, p, m) &:= \left\{
    \begin{array}{ll}
        h_{\star}^{-1}\left(\varphi(\cdot, l + e_{i}h_{\star}, \cdot) - \varphi\right) & \mbox{if } i \in \mathbf{J} \\
        0 & \mbox{else}
    \end{array}
\right. \\
	\Delta_{i}^{h_{\star}, p} \varphi(z, p, m) &:= \left\{
    \begin{array}{ll}
        h_{\star}^{-1}\left(\varphi(\cdot, p + e_{i}h_{\star}, \cdot) - \varphi\right) & \mbox{if } h_{1} \geq 0 \\
        h_{\star}^{-1}\left(\varphi - \varphi(\cdot, p - e_{i}h_{\star}, \cdot)\right) & \mbox{else}
    \end{array}
\right. \\
\end{aligned}
\]

in which $e_{i}$ is $i-th$ unit vector of $\mathbb{R}^{d}$. 

We shall assume that $\mathbf{A}$ is finite. We introduce:
\[
\mathbf{C}^{h_{\star}}_{\mathbf{J}} := \prod_{j=1}^{\kappa}(\partial \mathbf{1}_{\mathbf{J}^{c}}(j) +
(\mathbf{X}_{c}^{h_{\star}} \times \mathbf{R}_{c}^{h_{\star}} \times \mathbf{A})\mathbf{1}_{\mathbf{J}}(j)),
\]
in which $\mathbf{R}_{c}^{h_{\star}} := \{ih_{\star}  :  -c/h_{\star} \leq i \leq c/h_{\star}\} $. 

Then, the discrete counter-part of the set of policies running in indexes $\mathbf{J}$ is defined by
\[
	\mathbf{CL}_{\mathbf{J}}^{h} := \mathbf{C}^{h_{\star}}_{\mathbf{J}} \times \mathbf{L}_{\mathbf{J}}^{h_{\circ}}. 
\]

We introduce:
\[
\overline{\Lambda}[h_{\circ}](t, p) = h_{\circ}^{-1}\int_{t}^{t+h_{\circ}}\int_{U^{\lambda}}\Lambda(s, \lambda)dm^{\lambda}(\lambda)ds,
\]

in which $m^{\lambda}$ is completely determined by $p$, recall Assumption \ref{m_lambda}.

Note that, for $u \in U^{\gamma}$, we may have $x + \beta(\cdot, u) + \mathfrak{F}(\cdot ; u) \not\in \mathbf{X}_{c_{x}}^{h_{\star}}$. One needs to approximate $\varphi$ with the \textit{closest} points in $\mathbf{X}_{c_{x}}^{h_{\star}}$. We have the same issue with $\mathbf{P}_{c_{p}}^{h_{\star}}$. We define $[\varphi]_{h_{\star}}$ as an approximation of $\varphi$ by
	\[
    	[\varphi]_{h_{\star}} = \sum_{(x', p') \in C_{h_{\star}}(x) \times C_{h_{\star}}(p)} \omega(x', p' \mid x, p)\varphi(\cdot, x', \cdot,  p', \cdot).
    \]
in which $C_{h_{\star}}(x)$ (resp. $C_{h_{\star}}(p)$) denotes the corners of the cube of $\mathbb{R}^{d}$ (resp. $\mathbb{R}^{d}$) in which $x$ (resp. $p$) belongs too and $\omega(\cdot \mid x, p)$ is a weight function.

Moreover, in order to integrate the boundary condition when $l_{j} \rightarrow \ell$ for some $j \in \mathbf{J}$, we define $\overline{\mathbf{L}}^{h_{\circ}} = \mathbf{L}^{h_{\circ}} \cup \ell$ and $\mathbf{L}_{\mathbf{J}}^{h_{\circ}} = \prod_{j=1}^{\kappa}(\partial \mathbf{1}_{\mathbf{J}^{c}}(j) + \mathbf{L}^{h_{\circ}}\mathbf{1}_{\mathbf{J}}(j))$. We introduce

	\[
    	[\varphi]^{\ell}(\cdot, c, l, \cdot) = \varphi(\cdot, \mathfrak{C}_{-}^{\ell}(c, l), \cdot), \ \ (c, l) \in \mathbf{C}^{h_{\star}}_{\mathbf{J}} \times \overline{\mathbf{L}}_{\mathbf{J}}^{h_{\circ}}.
    \]
    
    And finally,
    
    \[
    	[\varphi]_{h_{\star}}^{\ell} = [[\varphi]^{\ell}]_{h_{\star}}.
    \]

The discrete counterpart of $\mathcal{L}_{*}^{\mathbf{J}}$ for all $\mathbf{J} \in \mathcal{P}(\mathbf{K})$ is

\begin{equation}\label{Lh}
		\begin{aligned}
    		\mathcal{L}_{h}^{\mathbf{J}}\varphi := \ &  \Delta_{i}^{h_{\circ}, t} [\varphi]^{\ell} + \sum_{1 \leq i \leq d} \mu^{i}\Delta_{i}^{h_{\star}, x}[\varphi]^{\ell}  + \sum_{i \in \mathbf{J}} \Delta_{i}^{h_{\star}, \ell}[\varphi]^{\ell}  + \sum_{1 \leq i \leq d} h_{1} \Delta_{i}^{h_{\star}, p}[\varphi]^{\ell} \\
            &+ \overline{\Lambda}[h_{\circ}]\int_{U^{\gamma}}\int_{\mathbb{R}^{d}}\left[\mathcal{I}\left[[\varphi]_{h_{\star}}^{\ell}, u\right](t+h_{\circ}, \cdot) - \varphi\right]\Upsilon(\gamma, du)dm^{\gamma}(\gamma).
        \end{aligned}
\end{equation}
If $\Upsilon(\gamma, du)$ corresponds to a discret distribution, the corresponding integral is explicit. In the examples in next section, we shall use a discrete approximation.

For the sequel, we set $\phi^{\circ} \in \Phi^{z, \overline{m}}_{\kappa}$ a control such that $\tau_{1}^{\phi^{\circ}} > T$ a.s. and $\phi^{a} \in \Phi^{z, \overline{m}}_{\kappa}$ a control such that $\tau_{1}^{\phi^{a}} = t$ a.s. and $\tau_{2}^{\phi^{a}} > T$ a.s. for $a \in \mathbf{A}$. Thus, the discrete counterpart of $\mathcal{K}$ is

\begin{equation}\label{cat Kh}
	\mathcal{K}^{h}\varphi := \sup_{a \in \mathbf{A}}\mathbb{E}_{\overline{m}}\left[[\varphi]_{h_{\star}}^{\ell}(Z_{t+h_{\circ}}^{z, \phi^{a}}, P_{t+h_{\circ}}^{t, p}, M^{z, m, \phi^{a}}_{t+h_{\circ}})\right].	
\end{equation}

  We set $\mathring{\mathbf{X}}_{c_{x}}^{h_{\star}} := (\mathbf{X}_{c_{x}}^{h_{\star}} \backslash \partial\mathbf{X}_{c_{x}}^{h_{\star}})$,  and $\mathring{\mathbf{P}}_{c_{p}}^{h_{\star}} := (\mathbf{P}_{c_{p}}^{h_{\star}} \backslash \partial\mathbf{P}_{c_{p}}^{\circ, h_{\star}})$. Our numerical scheme consists in solving, for all $\mathbf{J} \in \mathcal{P}(\mathbf{K})$:

	\begin{align}
        0 = & \mathbf{1}_{\{\mathbf{J} = \mathbf{K}\}}\left[-\mathcal{L}_{h}^{\mathbf{J}}\varphi\right] +  \mathbf{1}_{\{\mathbf{J} \not= \mathbf{K}\}}\min\left\{-\mathcal{L}_{h}^{\mathbf{J}}\varphi \, , \, \varphi - \mathcal{K}^{h}\varphi\right\} & \text{on } (\mathbf{T}^{h_{\circ}} \backslash T) \times  \mathring{\mathbf{X}}_{c_{x}}^{h_{\star}} \times \mathbf{CL}_{\mathbf{J}}^{h} \times \mathring{\mathbf{P}}_{c_{p}}^{h_{\star}} \times \mathbf{M} \label{num_interieur} \\
         \varphi = & g\mathbf{1}_{\{\mathbf{J} = \mathbf{K}\}} + \left(g \vee \mathcal{K}[g]_{h_{\star}}\right)\mathbf{1}_{\{\mathbf{J} \not= \mathbf{K}\}} & \text{on }  \{ T \} \times  \mathring{\mathbf{X}}_{c_{x}}^{h_{\star}} \times \mathbf{CL}_{\mathbf{J}}^{h} \times \mathring{\mathbf{P}}_{c_{p}}^{h_{\star}} \times \mathbf{M} \label{num_T} \\
       \varphi = & g & \text{on } \mathbf{T}^{h_{\circ}} \times \partial\mathbf{X}_{c}^{h_{\star}} \times \mathbf{CL}_{\mathbf{J}}^{h} \times \mathring{\mathbf{P}}_{c}^{h_{\star}} \times \mathbf{M} \label{num_bord}
	\end{align}
    
   \begin{proposition}
   	Let $\vr^{c}_{h}$ denote the solution of (\ref{num_interieur})-(\ref{num_T})-(\ref{num_bord}). Then $\vr_{h}^{c} \rightarrow \vr$ when $(h_{\star}, h_{\circ}/h_{\star}) \rightarrow 0$ and $c \rightarrow +\infty$.
   \end{proposition}

   \begin{proof}
   	We check that the conditions of \cite[Theorem 2.1.]{barles1991convergence} are satisfied as in \cite{baradel2016optimalMML}.
   \end{proof}

It remains to explain how to deduce the $\varepsilon$-optimal policy. At each point $(z, p, m)$ of the grid, if the number of running CAT bonds is $\kappa-1$ or less (recall that $\kappa$ is the maximum number of running CAT bonds), one computes
\[
\widehat{a}(z, p, m) \in \argmax_{a \in \mathbf{A}}\mathbb{E}_{\overline{m}}\left[[\vr^{c}]_{h}^{\ell}(Z_{t+h_{\circ}}^{z, \phi^{a}}, P_{t+h_{\circ}}^{t, p}, M^{z, m, \phi^{a}}_{t+h_{\circ}})\right]
\]
If $\vr^{c}_{h}(z, p, m)$ is equal to the above maximum, then we play the control $\widehat{a}(z, p, m)$ otherwise we wait for the next time step. This is the usual philosophy: we act on the system only if this increases the expected value.

	\section{Example: CAT bonds in a \textit{per event} framework for Hurricanes in Florida}

We focus on a simple example where the controller is an insurance or a reinsurance company which can issue CAT bonds in order to cover its risk in natural disasters. 

We consider CAT bonds of \textit{per event} type. The framework is the following:

\smallskip

\begin{itemize}
    \item The studied risk is the hurricanes ;
    \item We only consider one region : Florida ;
    \item The time-unit will be the year and we fix $\ell = 3$ which corresponds to the average maturity of CAT bonds in years ;
    \item The insurer can issue CAT bonds on different layers.
\end{itemize}

The motivation of considering hurricanes in Florida comes from the fact that this region is well exposed, about one hurricane every two years in average, see \cite{malmstadt2009florida} ; and has an important and increasing insured value about 4000 billion in 2015, see \cite{air2016}. Therefore, it has been well studied and we will take the parameters of our toy model from different papers.

\smallskip

We consider a 1-dimension random Poisson measure $N$, which represents the intensity of arrival and the severity of hurricanes. Only the intensity of arrival of the hurricanes is unknown. We use two different priors. The first one is the case with a Gamma distribution on the unknown parameter. The second one is the case with a Bernoulli distribution.

\smallskip

This leads to two different definitions of the intensity that we first explain in subsection 7.1 and 7.2. We describe in subsection 7.3 the severity of the hurricanes, which is assumed to be known. In subsection 7.4, we give the set of controls (which kind of CAT bonds the insurer can issue) and the output process (the process defined in equation (\ref{X})). In subsection 7.5 we describe the gain function and a specific dimension reduction that can be used for the numerical implementation. In subsection 7.6, we fix and explain the numerical values chosen for the parameters of the control problem. The numerical results are presented in subsection 7.7. Finally, we discuss in subsection 7.8 about the benefits and the limits in practice of this approach for a decision making process.

\subsection{Intensity of Hurricanes: the Gamma case}

We define the intensity $\Lambda$ as the function:
	\[
    	\Lambda(t, \lambda) = \lambda h(t), \ \ \ (t, \lambda) \in [0, T] \times \mathbb{R}_{+}^{*},
    \]
    
in which $h : t \mapsto h(t)$ is a positive continuous function which represents the seasonality of the arrival of hurricanes and some growth according to the global warming. The parameter $\lambda \in U^{\lambda} := \mathbb{R}_{+}^{*}$, which is unknown, represents a level of intensity.

\smallskip

We set $m_{0}^{\lambda} = \mathcal{G}(\alpha_{0}, \beta_{0})$ with $(\alpha_{0}, \beta_{0}) \in (\mathbb{R}_{+}^{*})^{2}$ as an initial prior on $\lambda$. Thus, by Example \ref{gamma}, we deduce that the process $M^{t, m^{\lambda}}$, starting from $m^{\lambda} := \Gamma(\alpha_{t}, \beta_{t})$ at $t \in [0, T]$, remains in the family of Gamma distributions and, for all $s \geq t$,
	\[
    	M_s = \mathcal{G}\left(\alpha_{t} + N_{s} - N_{t}, \beta_{t} + \int_{t}^{s}h(u)du\right).
    \]
Moreover, we can define two processes $P^{\alpha}$ and $P^{\beta}$ :
	\[
    	\begin{aligned}
        	P^{\alpha} &= P_{t}^{\alpha} + \int_{t}^{\cdot}dN_s, \\
            P^{\beta} &= P_{t}^{\beta} + \int_{t}^{\cdot}h(s)ds.
        \end{aligned}
    \]
and, by construction, $M = \mathcal{G}(P^{\alpha}, P^{\beta})$.

It remains to define the function $h$. The seasonality of hurricanes has been studied, especially on big Hurricanes, in \cite{parisi2000seasonality} in which the authors give a curve based on a kernel density estimation. One close parametric density function over one year can be found in the form:

	\begin{align}
    	h_{0} : [0, 1] &\rightarrow \mathbb{R}_{+} \\
        	t &\mapsto \left\{
    \begin{array}{ll}
        f_{\hat{\alpha}, \hat{\beta}}\left(\frac{t-d_{0}}{d_{1} - d_{0}}\right) & \mbox{if } t \in (d_{0}, d_{1}) \\
        0 & \mbox{else}
    \end{array}
\right.
    \end{align}

in which $f_{\hat{\alpha}, \hat{\beta}}$ is the density function of the Beta distribution of parameters $(\hat{\alpha}, \hat{\beta}) \in (\mathbb{R}_{+}^{*})^{2}$.

The Figure $\ref{seasonality}$ shows a representation of $h_{0}$ close to the one obtained in \cite{parisi2000seasonality}.

\begin{figure}[!h]
\centering
\includegraphics[scale=0.6]{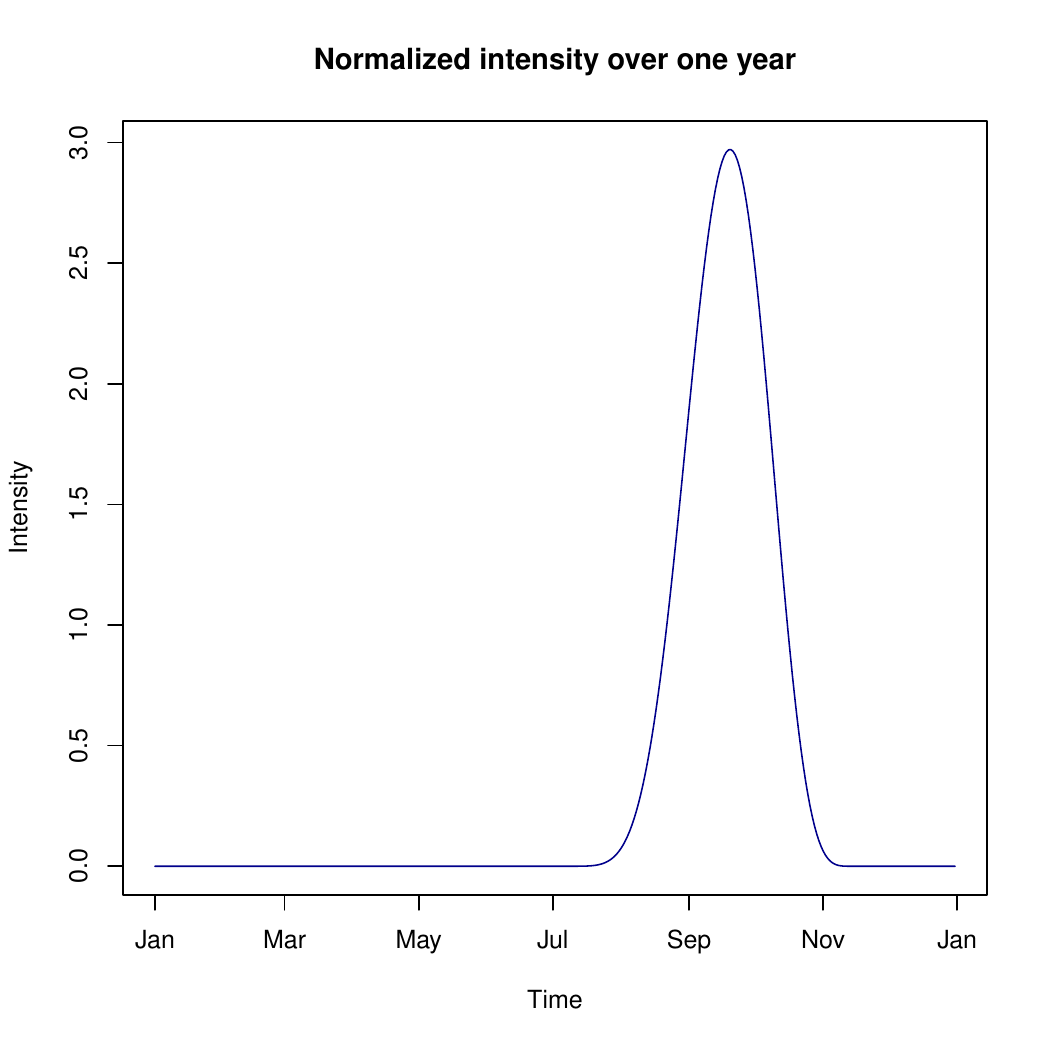}
\caption{Representation of $h_{0}$ over one year with $d_{0} = 1^{\text{st}}$ July, $d_{1} = 15^{\text{th}}$ November, $\hat{\alpha} = 8$ and $\hat{\beta} = 6$.}\label{seasonality}
\end{figure}

The function $h$ is simply defined by $h(t) := h_{0}(t - E(t))$ for $t \geq 0$, in which $E$ is the integer part function. $h$ is 1-periodic. Here we do not consider a global warming effect, which would have been deterministic through the function $h$.

\subsection{Intensity of Hurricanes: the Bernoulli case}

Although the Gamma prior gives parameters that belongs in $\mathbb{R}_{+}$, in order to remains in the Gamma distribution over time, it requires the form $(t, \lambda) \mapsto \lambda h(t)$ and then the intensity of the whole period is proportional to $\lambda$. We introduce a Bernoulli case with three alternatives in which we can choose any function depending on time with each alternative.

With $E : \mathbb{R}_{+} \mapsto \mathbb{N}$ the integer part function, we define the intensity as:
	\begin{equation}\label{dirac_intensity}
    	\Lambda(t, \lambda) = \frac{1}{2}h(t)\left(1 + \frac{E(t)}{T}\lambda\right), \ \ \ (t, \lambda) \in [0, T] \times \{\lambda_{1}, \lambda_{2}, \lambda_{3}\},
    \end{equation}
in which the parameter $\lambda \in U^{\lambda} := \{\lambda_{1}, \lambda_{2}, \lambda_{3}\} \subset \mathbb{R}_{+}$ represents 3 scenarios of the evolution of the intensity, as a consequence of the global warming.

Following Example \ref{bernoulli}, we can define 3 processes, starting from $p := (p^{1}, p^{2}, p^{3}) \in \mathbb{R}_{+}^{3}$ at time $t\in[0, T]$:
	\begin{equation}\label{dirac_prior}
    	P^{i} := p^i - \int_{t}^{\cdot} P_{s}^{i} \Lambda(s, \lambda_{i})ds + \int_{t}^{\cdot} P_{s-}^{i}\left[\Lambda(s, \lambda_{i}) - 1\right]dN_{s}, \ \ \ 1 \leq i \leq 3.
    \end{equation}

\subsection{Severity of the Hurricanes}

As in \cite{malmstadt2009florida}, we use a Generalized Pareto Distribution for the simulation of the severity of the claim, over the exposure of 4000 billion. 
Their threshold (minimum claim size) is $\mu = 0.25$ billion for an exposure of 2000 billion. Here, we shall use: $\mu = 0.5$, $\sigma = 5$ and $\xi = 0.5$. To fix ideas, the median is 4.5 billion, the quantile at 90\% is 22 billion and the quantile at 99.5\% is 132 billion. We also bound the distribution by the total exposure of 4000 billion.

\bigskip

We also introduce the so-called Occurrence Exceedance Probability (OEP) curve. To this aim, we introduce the random variable:
		\[
        	\iota_{t} := \underset{t \leq s \leq t+1}{\max}\int_{\mathbb{R}^{*}}uN(du, \{s\}),
        \]
which is the greatest Hurricane in $[t, t+1]$ for $t \in [0, T-1]$. The OEP curve is simply:
		\[
        	OEP_{\mathfrak{t}}^{t} := \inf\left\{x \in \mathbb{R} : \mathbb{P}\left(\iota_{t} \leq x\right) \geq 1-\frac{1}{\mathfrak{t}}\right\}, \ \ \mathfrak{t} \geq 0,
        \]
in which $\mathfrak{t}$ is called the \textit{Return period}. By construction,  $OEP^{t}_{\mathfrak{t}}$ is the quantile of order $1-1/\mathfrak{t}$ of $\iota_{t}$.

The Figure \ref{OEP} shows the corresponding OEP curve with the Gamma prior $(p^{\alpha}, p^{\beta}) := (25, 50)$, for any year (in this case, it does not depend on $t$, for a fixed prior).

\begin{figure}[!h]
\centering
\includegraphics[scale=0.6]{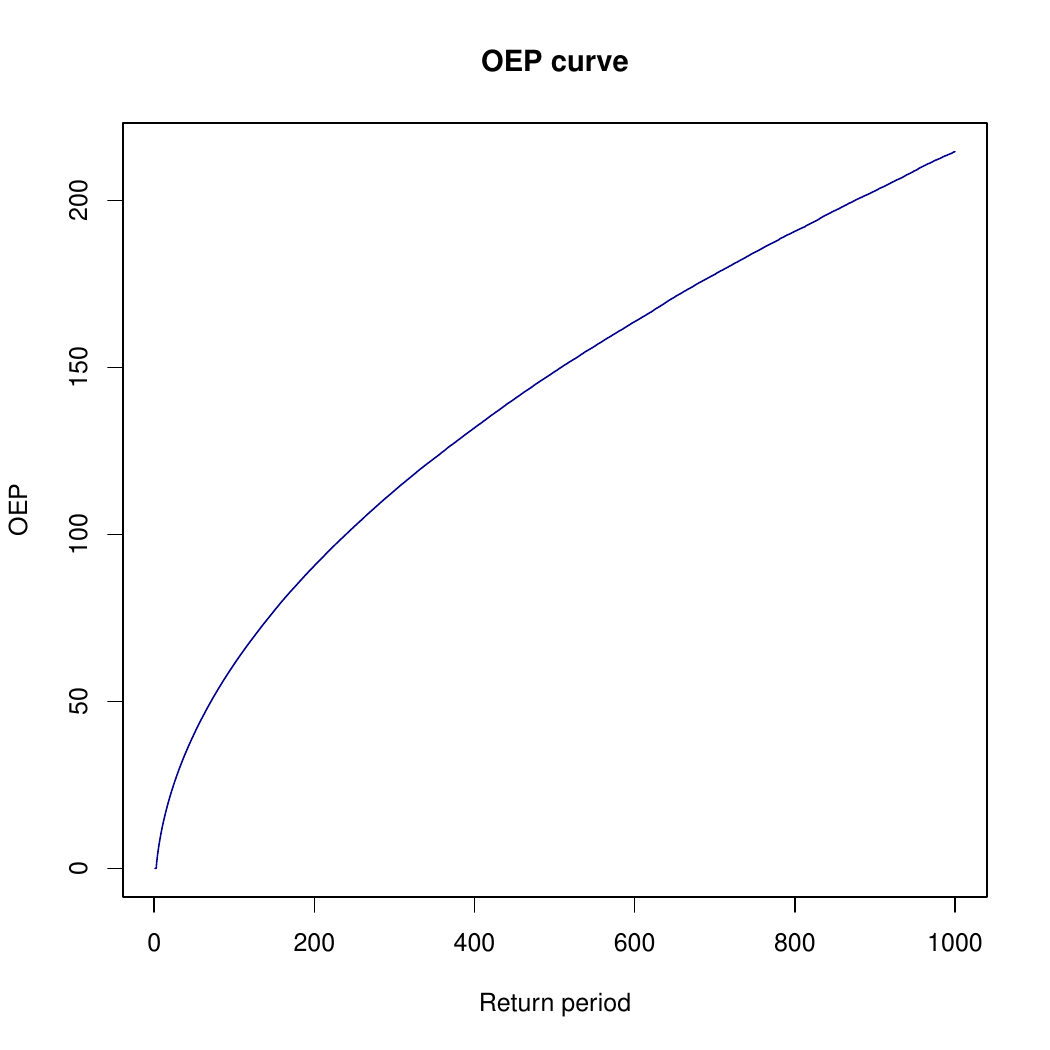}
\caption{Representation of an OEP curve, with the parameter $(\mu, \sigma, \xi)$ defined in the text and with the prior $(p^{\alpha}, p^{\beta}) := (25, 50)$.}\label{OEP}
\end{figure}

\smallskip

We now define the set of controls and the output process.

\subsection{The set of controls and the output process}

Recall that a control $\phi$ has the form $(\tau_{i}^{\phi}, k_{i}^{\phi}, n_{i}^{\phi})_{i \geq 0}$. Here, $n_{i}$ is the linked to the notional of the CAT bonds. A CAT bond covers a layer (defined hereafter) of the portfolio of the insurer. We fix $n_{i} := 1$ for simplicity, so if the insurer issue a CAT bond, the whole corresponding layer will be covered.

\smallskip

We introduce $\{K_{1}, K_{2}, K_{3}, K_{4}\} := \{10, 50, 200, 1000\}$. We define what will be the capacity of the CAT bonds: $\mathfrak{l}_{K_{j}}^{t} = OEP_{K_{j+1}}^{t} - OEP_{K_{j}}^{t}$ for $1 \leq j \leq 3$ and $t \in [0, T-1]$.

The value $k_{i}$ can be chosen in $\{K_{1},K_{2}, K_{3}\}$ and the associated sets $A_{k_{i}}$ are defined by:

	\[
    	A_{k_{i}}^{t} = [OEP_{k_{i}}^{t}, +\infty[, \ \ i \geq 1.
    \]
If a Hurricane leads to a cost in $A_{k_{i}}^{t}$, then the default of the CAT bond is activated. It remains to define the payout for the insurer in the default case. It corresponds to cover the layer $[OEP_{k_{i}}^{t}, OEP_{k_{i}}^{t} + \mathfrak{l}_{k_{i}}^{t}]$ at a ratio of $n_{i}$. We define the payout of the $j-th$ CAT bonds as:
	\[
    	F_1(t, x, n_{j}, k_{j}, l_{j}, u) := n_{j}\left[\left(u - OEP_{k_{j}}^{t-l_{j}}\right)^{+} \wedge \mathfrak{l}_{k_{j}}\right], \ \ j \in \{1, \ldots, \kappa\}.
    \]
Note that, in our example, the risk cannot be covered above the return period of 1000.

We consider the process $X := (X^{1}, X^{2})$ valued in $\mathbb{R}^{2}$. The first component represents the cash of the Insurer/Reinsurer and the second component represents the risk premium, in term of percentage of the pure premium, of the market about the CAT bonds, defined later.

\smallskip

We set, with $x:=(x^{1}, x^{2})$ :
	\[
    	\begin{aligned}
        	\mu(t, x) &= \begin{pmatrix}
        \mu + r_0 x^{1} \\ 
        -\rho x^{2} 
    \end{pmatrix}
, \\
         \beta(t, x, u) &=
    \begin{pmatrix}
        u \\ 
        \rho^{\star}(u) 
    \end{pmatrix},\\
        	H(t, x, a) &=
    \begin{pmatrix}
        -H_{0} \\ 
        0
    \end{pmatrix}, \\
         \overline{C}(t, r) &=
    \begin{pmatrix}
        \sum_{i=1}^{\kappa} r_i \\ 
        0
    \end{pmatrix}, \\
         F &=
    \begin{pmatrix}
        F_1 \\ 
        0
    \end{pmatrix},
        \end{aligned}
    \]
    
in which:

\begin{itemize}
    \item $\mu$ represents the premium rate,  the insurer is profitable if $\mu > \mathbb{E}_{\overline{m}}\left[\Lambda(t, \lambda)\right]\int_{\mathbb{R}^{*}}u\Upsilon(du)$ ;
    \item $r_0 > 0$ is the constant interest rate ;
    \item $\rho > 0$ is the speed return to 0 of the risk premium of the CAT bonds ;
    \item $\rho^{\star} : \mathbb{R} \mapsto \mathbb{R}$ is an increasing function which represents the immediate increase of the premium rate after a claim ;
    \item $H_{0} > 0$ is the initial cost of issuing a CAT bond.
\end{itemize}

How the coupon $r$ is fixed when issuing a CAT bond and is defined in subsection \ref{param_choice}.

\subsection{Gain function and dimension reduction}

The controller wants to maximize, with $c := (n, k, r)$, for some $\gamma > 0$, the criteria
	\[
    	g(x, c, l, p) := -\exp\left[-\gamma\left(x^{1} + \frac{H_{0}}{\ell}\sum_{k=1}^{\kappa}\mathbf{1}_{\{l_{k} \not= \partial\}}(\ell-l_{k})\right)\right] \vee \hat{C}.
    \]
    
    The right part inside the exponential function compensates the initial cost for remaining CAT bonds, in order to avoid particular behaviors of issuing nothing close to the end. We take $\hat{C} := -10^{300}$ which ensures that $g$ is bounded and big enough such that it will not play an essential role.
    
    \bigskip
    
	Note that in the Gamma prior case, we have $P^{\beta} = P_{t}^{\beta} + \int_{t}^{\cdot}h(s)ds$ which is a function of time with no randomness. Then, we can avoid it in the numerical scheme since it is a function of time fully characterized by the initial prior.
	
	\smallskip
    
    In the Bernoulli case, one can see that, if we set for the prior
    
    	\[
        	p' := \delta p,
        \]
        
        for some $\delta > 0$, then, for all $s \geq t$, we have $P'_{s} = \delta P_{s}$ and then $\mathcal{D}(P'_{s}) = \mathcal{D}(P_{s})$. One can normalize $P$ such that the sum is 1 and avoid the last component.
        
	\smallskip

Moreover, the associated value function satisfies:
    
    \[
    	\vr(t, x^{1}, x^{2}, c, l, p) = e^{x^{1}e^{r_0(T-t)}}\vr(t, 0, x^{2}, c, l, p),
    \]
    which avoids in computation the dimension of $x^{1}$.

\subsection{The choice of the parameters}\label{param_choice}

We choose here the form, the functions and the parameters for our toy example. We first describe the Gamma case (for the prior) and then describe the Bernoulli case.

Just after the occurrence of Katrina, the price of the reinsurance was about two or three times greater with a persistence of about two years and can be also seen on the CAT bond market, see Figure 9 in \cite{cummins2012cat}. Thus, we set
	\[
        	\rho := 2.
    \]
    Moreover, the estimated return-period of such event is about $20$-year return period, see \cite{kates2006reconstruction}. Since the increase was about two of three times greater, we set
    
    \[
                \rho^{\star}(u) := \frac{0.05}{1-F_{\mu, \sigma, \xi}(u)},
    \]
    
    in which $F_{\mu, \sigma, \xi}$ denotes the cumulative distribution function of the Pareto distribution of parameters $(\mu, \sigma, \xi)$. Then, here, for a return period of 40 years (recall that we have in average one claim each 2-year period), it gives an increase of 100\% of the price.
    
The insurer has a market share of $e_{0} \in ]0, 1]$ that we fix at 10\%. We shall assume that, the insurer is profitable until $\lambda = 0.65$. Then, the premium rate is

	\[
    	\mu := 0.65\,e_{0} \int_{\mathbb{R}^{*}}u\Upsilon(du) = 0.65\times e_{0} \times \left(\mu_{0} + \frac{\sigma_{0}}{1-\xi}\right) = 0.6825.
    \]
    
    We now define the coupon fixing. If $k_{i} = K_{j}$ with $j \in \{1, 2, 3\}$ (recall  that $k_{i}$ is the choice of the layer for the CAT bond), the coupon is:
    \begin{equation}\label{numerique_coupon}
    	r_{i} = \mathfrak{C}_{0}(\tau_{i}, X_{\tau_{i}}, \alpha_{i}, \varepsilon_{i}) := n_{i}\left[e_{0}\left(\frac{1}{K_{j+1}} + \frac12\left(\frac{1}{K_{j}} - \frac{1}{K_{j+1}}\right)\right)\right]\mathfrak{l}_{K_{j}}\left(1 + x^{2} + \varepsilon_{i}\right).
    \end{equation}

Thus, the CAT bond price is decomposed by:

\begin{itemize}
    \item The part $\frac{1}{K_{j+1}}$ which is the probability that a claim is above the layer within one year and then the payout is the layer
    \item The part $\frac{1}{K_{j}} - \frac{1}{K_{j+1}}$ which is the probability that the greatest claim is in the layer, and we multiply it by one half like if it was uniformly distributed in the layer, which is greater than the true value.
    \item The factor $x^{2}$ is the risk aversion of the market, and $\varepsilon_{i}$ is some random value about the price the coupon.
\end{itemize}

Finally, the cost of issuing a CAT bond is fixed at: $H_{0} := 0.0025$, the interest rate is fixed at $r_0 := 1\%$.

\begin{remark}
In these examples, we deal with \textit{per event} CAT bonds. One also can deal with aggregated losses within the period. In this case, it requires to record the current accumulation of claims and to introduce another dimension in the output process $X$.
\end{remark}

\begin{remark}
	In practice, in general, a partial default below 70\%-80\% of the capacity does not end the CAT bonds: the coupon is reduced by the proportional loss and another loss may lead to the complete default, using the same limits. Here, for simplification, the CAT bond ends whenever the layer is attained.
\end{remark}

\begin{remark}
Note that the function $\Psi(x, c, l, p) := \frac{\mu}{r_0} + x^{1} + \delta$ satisfies the conditions of Proposition \ref{cat_comparaison}, for $\delta > 0$ great enough.
\end{remark}

\subsubsection{With the Bernoulli prior}

In this case, the intensity grows over time, recall (\ref{dirac_intensity}). We fix $\lambda_{1} = 0.2, \lambda_{2} = 0.3, \lambda_{3} = 0.4$ and $P_{0}^{1} = P_{0}^{2} = P_{0}^{3} = \frac{1}{3}$, recall (\ref{dirac_prior}).

To be consistent, we say that the premium rate also rises over time following the rise of intensity, but by 35\%, and then is:
\[
    \mu(t, x) = \begin{pmatrix}\mu\left(1 + 0.35\frac{t}{T}\right) + r x^1 \\ -\rho x^2\end{pmatrix}, \ \ (t, x) \in [0, T] \times \mathbb{R}^{2}.
\]
We assume that the market is updating the OEP with the same rate:

    \[
        OEP^{t} := OEP^{0}\left(1 + 0.35\frac{t}{T}\right), \ \ t \in [0, T].
    \]

\subsection{Results}

Recall that, for each CAT bond that the insurer can issue, we need to add its characteristics and the complexity increases hugely in $\kappa$, depending on possible policies. Thus, in our simulation, we use $\kappa = 2$. The controller can choose at most 2 layers among the three available (recall them in term of return periods: [10, 50], [50, 200] and [200, 1000] which corresponds to [1.23, 4], [4, 9], and [9, 21.5] in billion dollars).

\subsubsection{With the Gamma prior}

In Figure \ref{hugeclaim}, we provide a simulated path of the optimal strategy in which the Pareto distribution is discretized in 2500 points (the highest possible value is 49 billion dollars). The top left graphic describes the control played by the insurer. The top part represents the issue of CAT bonds, the level is the lower bound of the layer. The bottom part represents the running CAT bonds with respect to the layer. The double dash says that two CAT bonds at the same layer are running. The top right graphic describes the arrival of natural disasters. The bottom part gives the size of the claim of the insurer while the top part gives the payoff of the CAT bond(s). The middle left graphic describes the evolution of the cash of the insurer. The middle right graphic gives the evolution of $X^{2}$, the price penalty of the CAT bonds which appears in (\ref{numerique_coupon}). The bottom left graphic gives the evolution of the mean of the estimated distribution of $\lambda_{0}$ (true value is 0.6), defined by $\frac{P^{\alpha}}{P^{\beta}}$, and the bottom right graphic gives the evolution of the standard deviation, defined by $\frac{\sqrt{P^{\alpha}}}{P^{\beta}}$.

\begin{figure}
\centering\includegraphics[scale = 0.7]{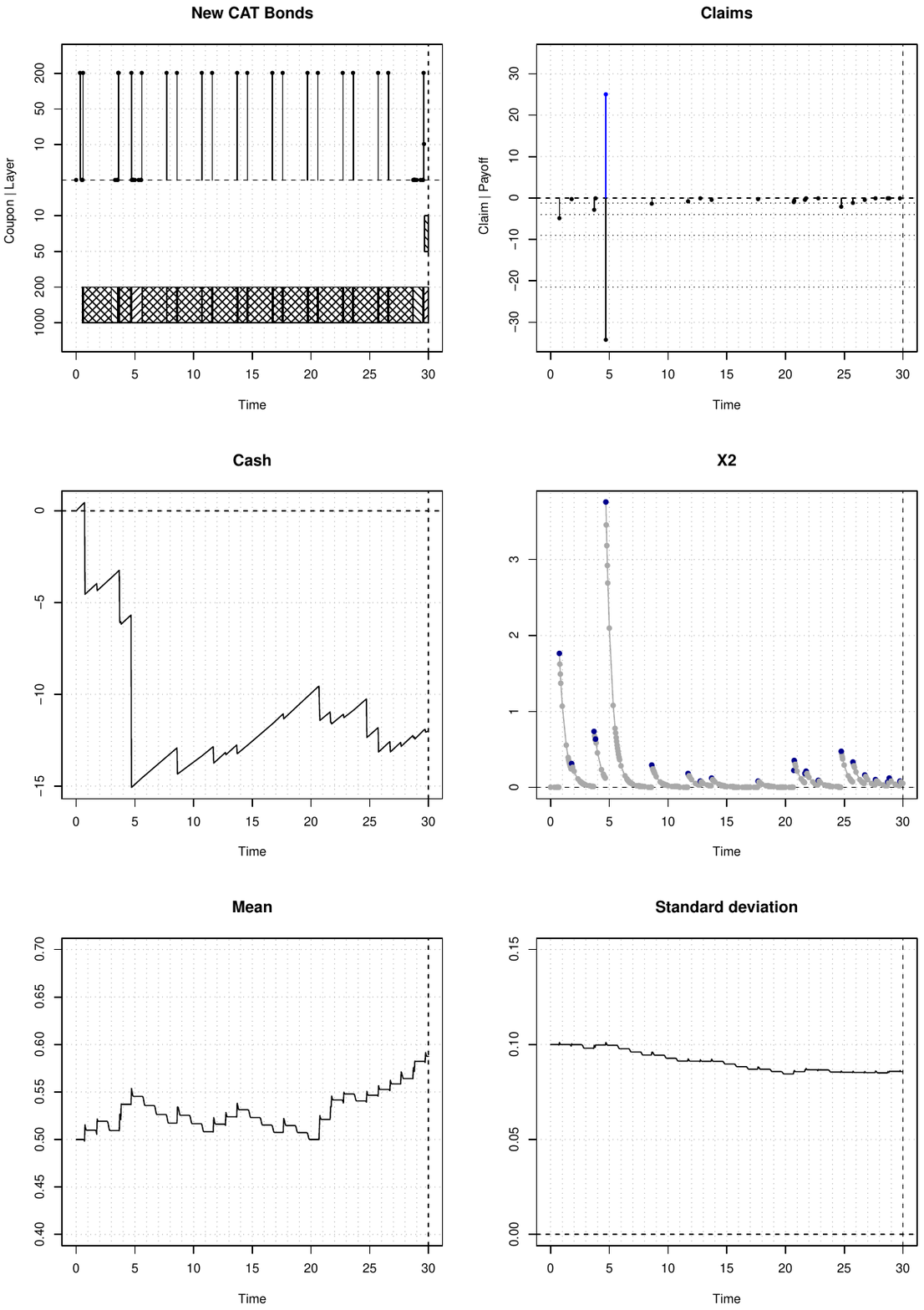}
\caption{Simulated path of the optimal strategy of the insurer (true value is $\lambda_0 = 0.6$).}\label{hugeclaim}
\end{figure}

\smallskip

At the beginning, the insurer does not issue any CAT bond. Since we start in January, there is no risk to experiment a claim and thus the insurer delays the issue. Just when the season starts, he first chooses to issue two CAT bonds on the layer $[200, 1000]$. Recall that it is the highest layer which corresponds to $[9, 21.5]$ in billion dollars. It is possible to have a claim highly above the layer and having a double cover on this big layer gives, indirectly, a cover against huge claims above the layer (recall that the maximum claim size is 49 billion dollars). He renews each CAT bond at the maturity until he meets a claim with a return period above 1000 during the $5^{\text{th}}$ year. He gets the associated payoff. Despite the huge increase of the price of CAT bonds, by almost $400\%$, he immediately issues a new one on the layer [200, 1000], but only one. He waits the next season, with a better expected price, to issue the other one. After, he follows this strategy to the end.

In Figure \ref{finalcash}, we represent the approximated density (by kernel estimation) of the total cash of the insurer at the end of the 30 years. On the left, it is the case with $\lambda_{0} = 0.6$ (the value used in the simulated path of Figure \ref{hugeclaim}) and on the right with $\lambda_{0} = 0.5$, i.e. what believes the insurer at the beginning. The solid curve is the case when the insurer plays the optimal control and the dashed curve is when he never issues any CAT bond. We also add the quantiles at 99.5\% in term of losses, see the legend. In the case with $\lambda_{0} = 0.6$ (left), from which the path in Figure \ref{hugeclaim} comes from, we can see that the standard deviation is reduced. And the quantile at 99.5\% is strongly reduced (in absolute value). One can observe that the case $\lambda_{0} = 0.6$ strongly reduces the expected net return in average. Without CAT bonds, the mean of the cash distribution is higher, mainly due to the costs associated with issuance and the premium related to the risk transfer to the market. On the other hand, the 99.5\% quantile is lower (indicating higher losses) when CAT bonds are not present, as they partially cover extreme risks.

\begin{figure}
\centering\includegraphics[scale = 0.48]{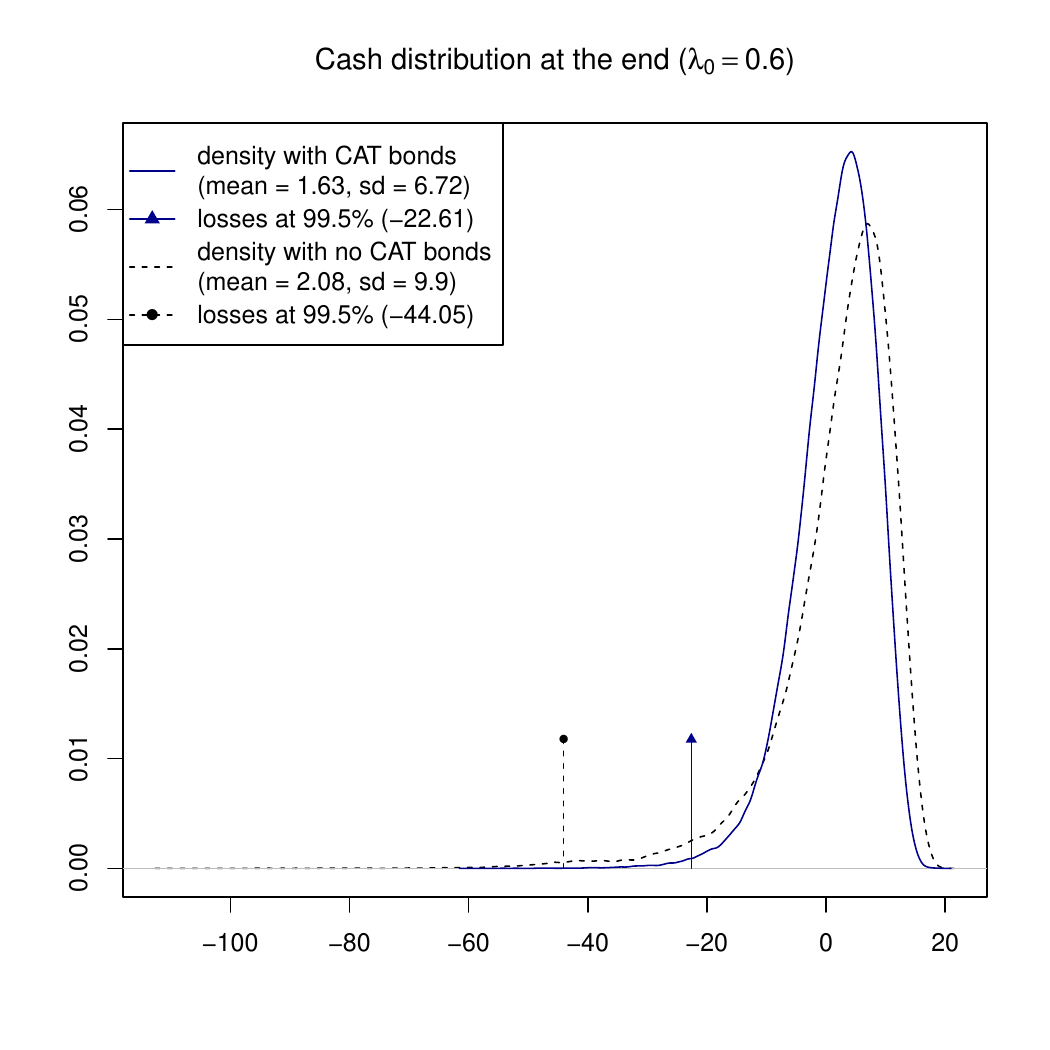}\includegraphics[scale = 0.48]{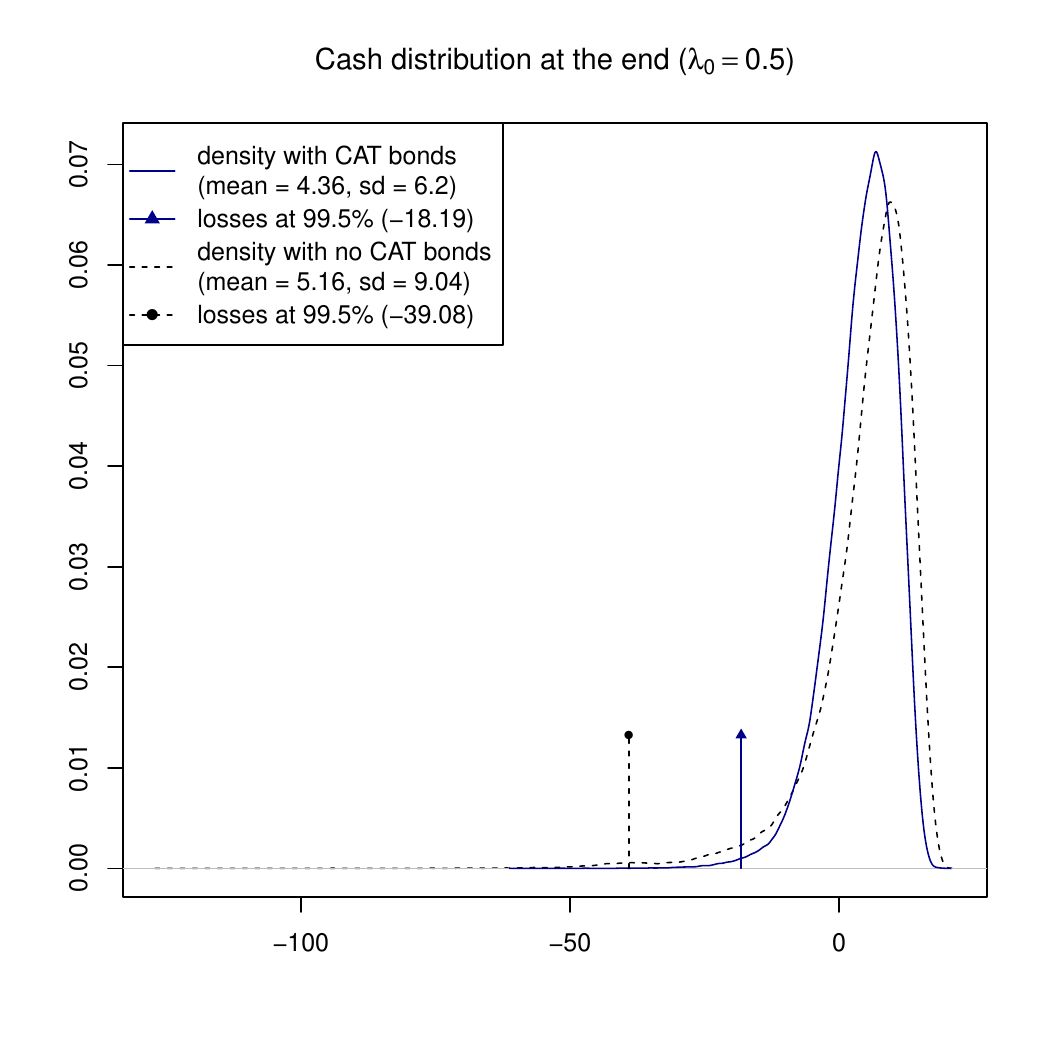}
\caption{Cash distribution (with 200 000 simulations) for $\lambda_{0} = 0.6$ (left) and $\lambda_{0} = 0.5$ (right) with the optimal control (solid dark blue) and without any CAT bond (dashed black).}\label{finalcash}
\end{figure}

\smallskip

We now look at the case with a discretization of 500 of the Pareto distribution. In particular, the maximum claim size is 21.4 billion which does not exceed the maximum layer $[9.0, 21.5]$. In Figure \ref{500normalclaim}, we show a simulated path. This time, the insurer chooses to get two CAT bonds at the layer $[50, 200]$. Actually, with this discretization, the layer $[200, 1000]$ appears to be less competitive since the discretization of 500 leads to a lower expected payoff. In the first years, the expected intensity is revised higher and the relative price of the layer $[10, 50]$ decreases (this layer requires the highest coupon since it is frequently hit). At the $4^{\text{th}}$ year, he changes his strategy and gets one CAT bond on the layer $[10, 50]$ and the other one on the layer $[50, 200]$. A catastrophe above the return period of 200 occurs at the $20^{\text{th}}$ year and both CAT bonds end. He prefers to wait the next season because of the consecutive price increase. Note that, in the previous cases (with Pareto distribution discretized in 2500 points), he was never without any CAT bond, even after an increase of $400\%$. Then, he continues his strategy to get a CAT bond on the layer $[10, 50]$ and the other one on the layer $[50, 200]$, until the end.

\begin{figure}
\centering\includegraphics[scale = 0.7]{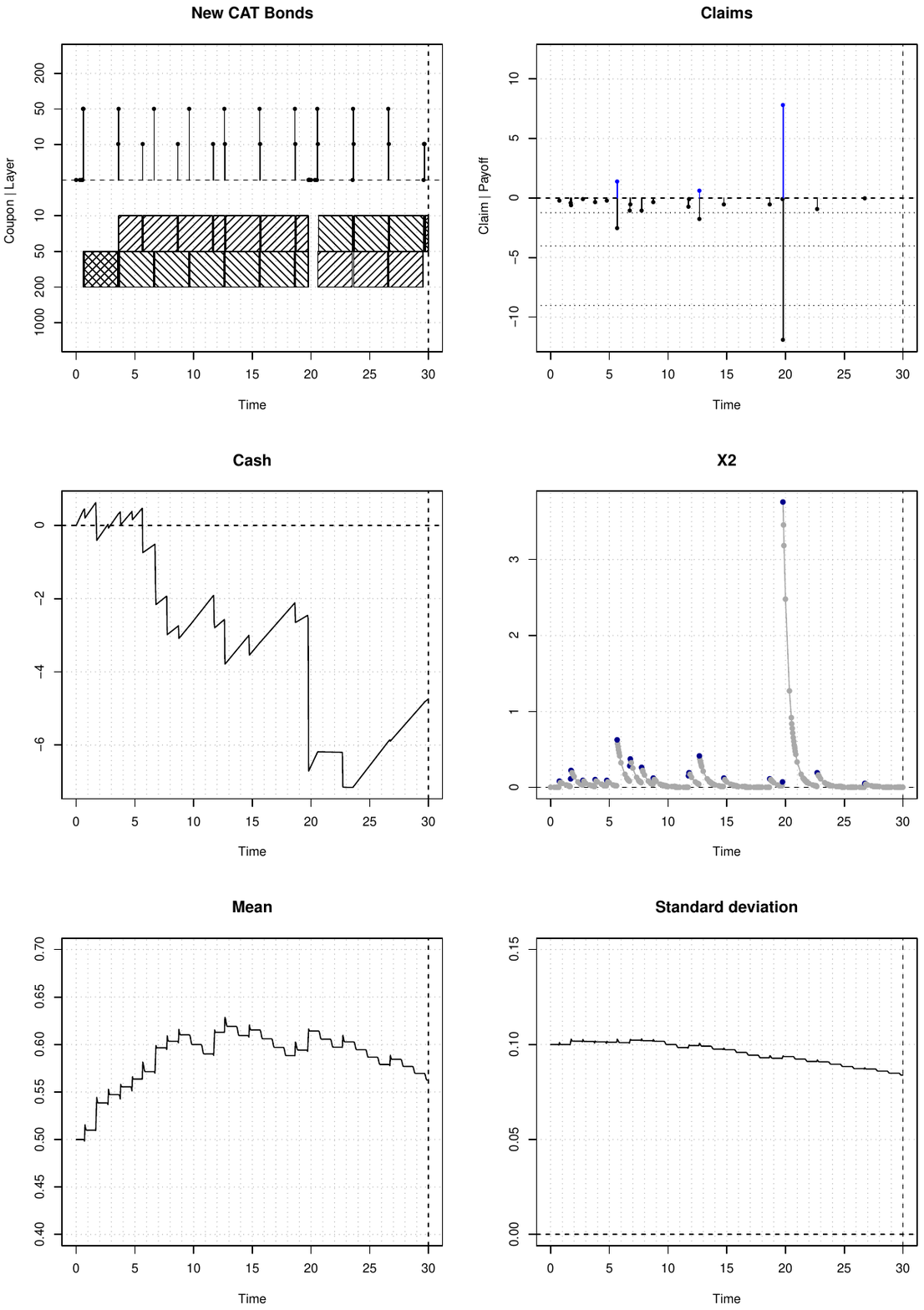}
\caption{Simulated path of the optimal strategy of the insurer.}\label{500normalclaim}
\end{figure}

\clearpage

\subsubsection{With the Bernoulli prior}

In Figure \ref{bernoulli_2500normalclaim}, we provide a simulated path of the optimal strategy in which the Pareto distribution is discretized in 2500 points (recall that the highest possible value is 49 billion dollars). As in the Gamma prior case, the insurer chooses to get two CAT bonds at the highest layer. When he experiences a huge claim during the second year, he still gets twice the layer but prefers to wait before to take a new CAT bond, according to the huge rise of the price. He waits the next year and restarts the same strategy until the $12^{th}$ year. Then, he issues CAT bonds on the layer $[50, 200]$ and $[200, 1000]$ until close to the end. 

The estimated probabilities on $\lambda_{0}$ evolve slowly at the beginning since $\lambda_{0}$ has an impact which rises over time (true value is $\lambda_0 = 0.4$).

\smallskip

\begin{figure}[h]
\centering\includegraphics[scale = 0.7]{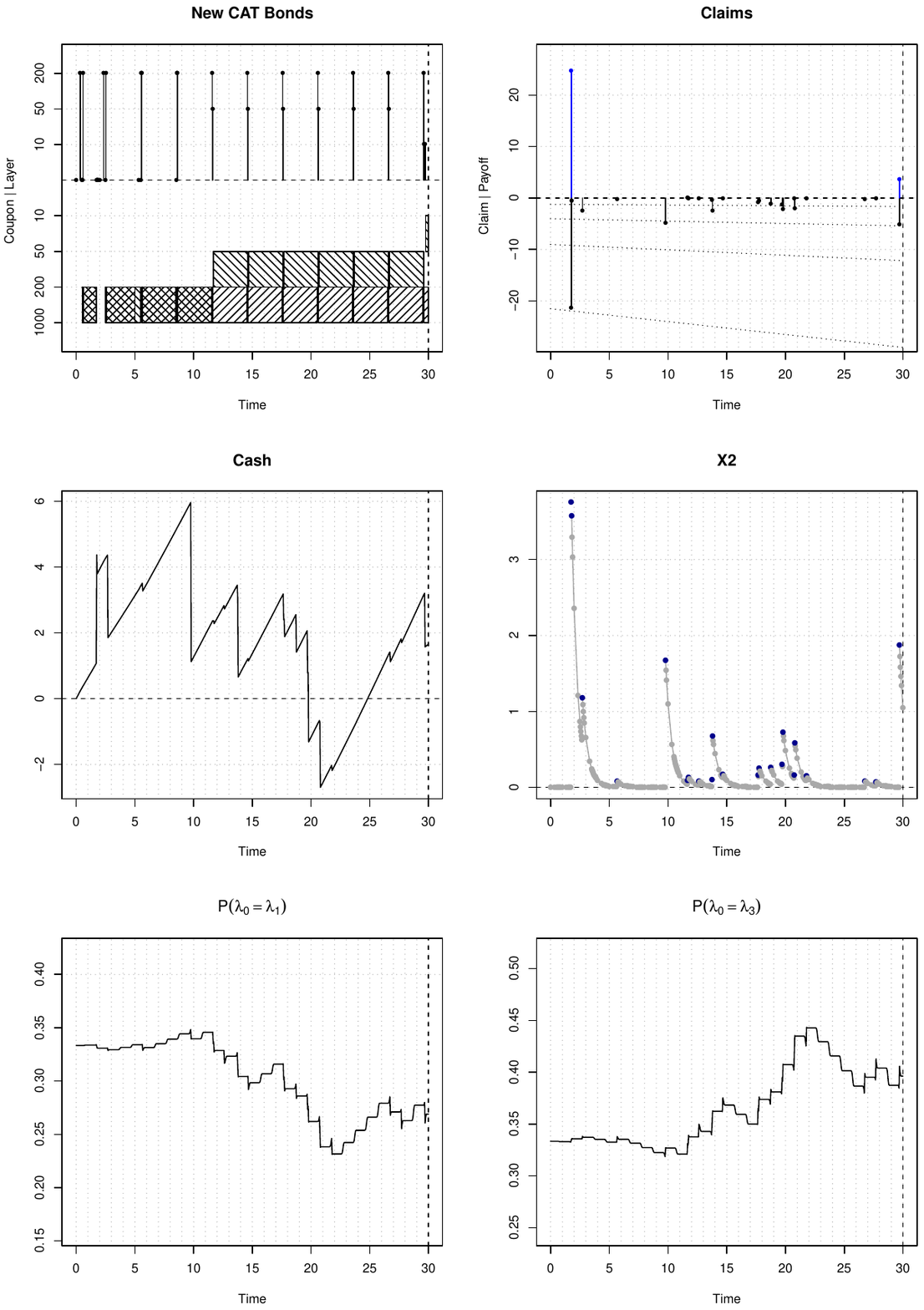}
\caption{Simulated path of the optimal strategy of the insurer (true value is $\lambda_0 = 0.4$, recall that $\lambda_1 = 0.2$, $\lambda_2 = 0.3$, and $\lambda_3 = 0.4$).}\label{bernoulli_2500normalclaim}
\end{figure}

In Figure \ref{bernoulli_finalcash}, we represent the approximated density (by kernel estimation) of the total cash of the insurer at the end of the 30 years. On the left, it is the case with $\lambda_{0} = 0.4$ (as it is also the case in Figure \ref{hugeclaim}) and on the right with $\lambda_{0} = 0.3$. The legend is the same as in Figure \ref{finalcash} and we get close distributions.

\begin{figure}[h]
\centering\includegraphics[scale = 0.48]{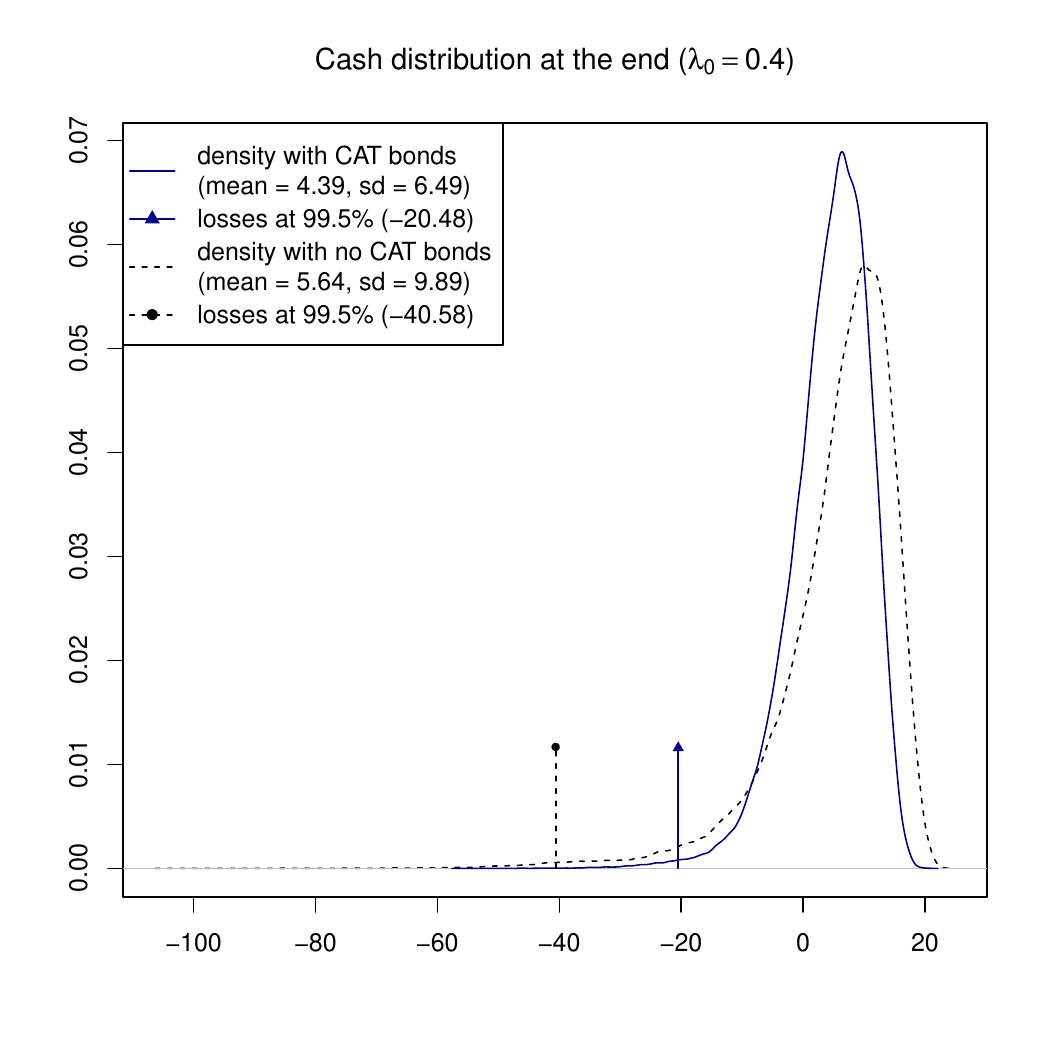}\includegraphics[scale = 0.48]{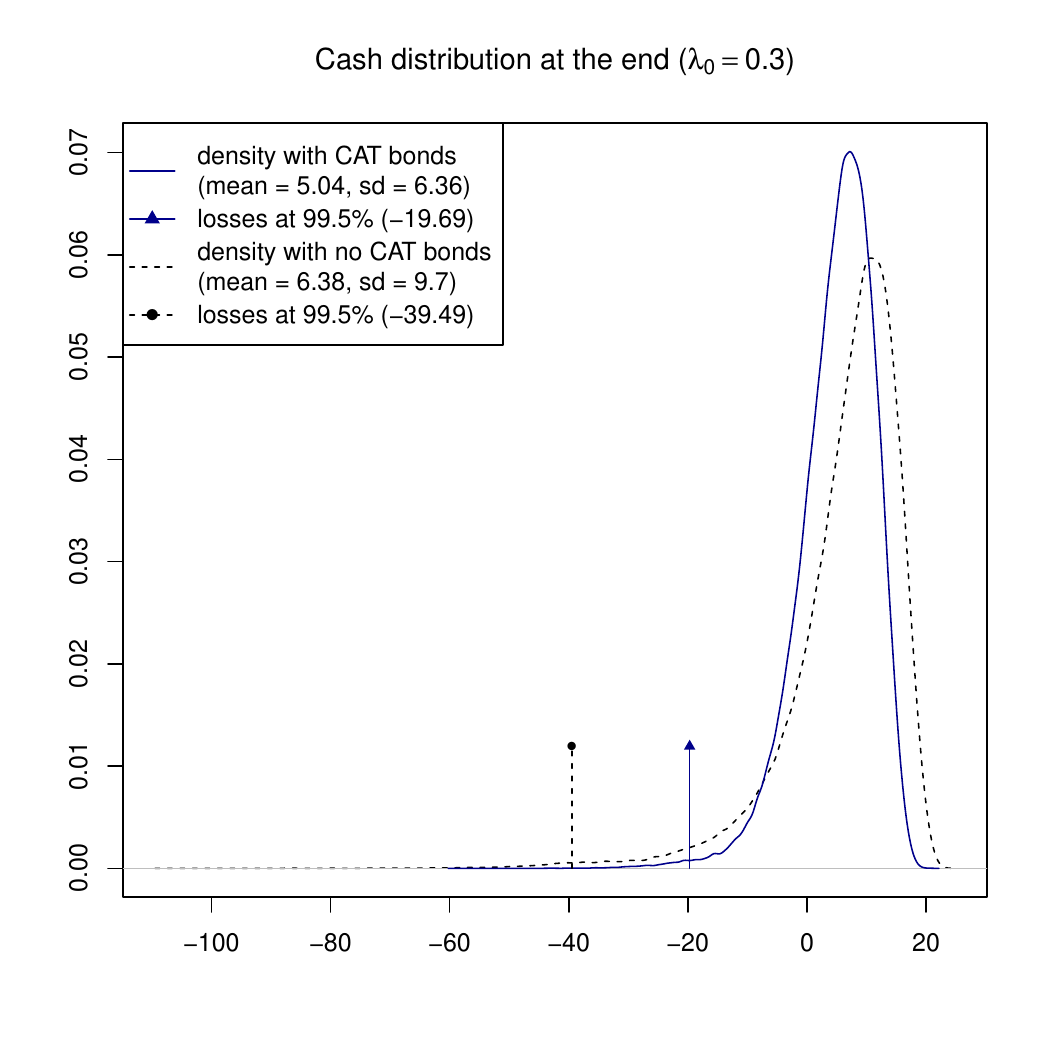}
\caption{Cash distribution (with 200 000 simulations) for the increase parameter $\lambda_{0} = 0.4$ (left) and $\lambda_{0} = 0.3$ (right) with the optimal control (solid dark blue) and without any CAT bonds (dashed black).}\label{bernoulli_finalcash}
\end{figure}

\smallskip

We now look at the case with a discretization of 500 of the Pareto distribution and show a simulated path in Figure \ref{bernoulli_500normalclaim}. As in the Gamma prior case, at the beginning, the insurer chooses to get two CAT bonds at the layer $[50, 200]$. He follows this strategy until he meets a huge claim in the $16^{th}$ year. He waits the next season and restarts the same strategy. At the $24^{th}$ year, he chooses to issue CAT bonds on two different layers, at $[50, 200]$ and $[10, 50]$. As in Figure \ref{500normalclaim}, this results in a change on the belief on the intensity.

\begin{figure}[h]
\centering\includegraphics[scale = 0.7]{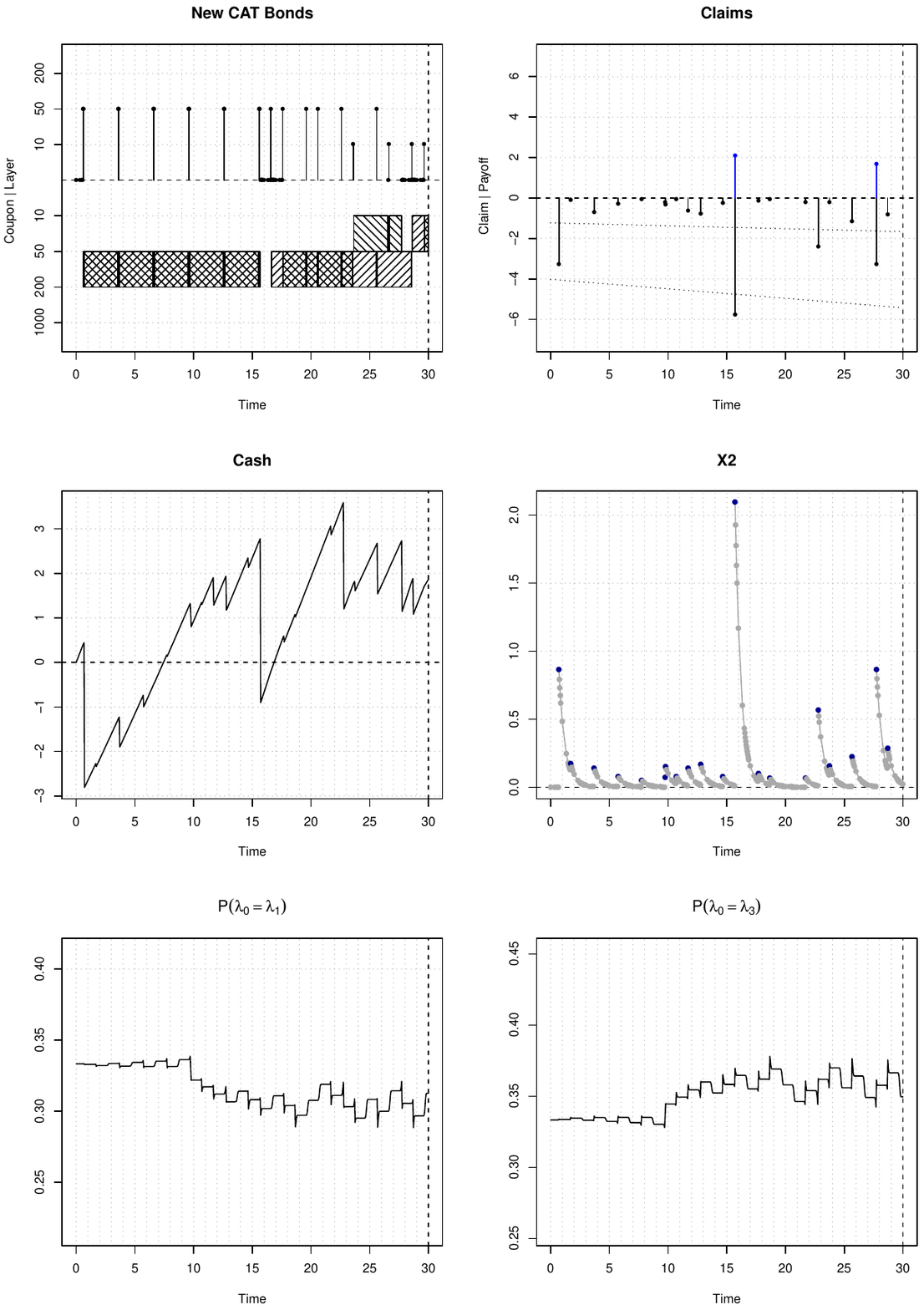}
\caption{Simulated path of the optimal strategy of the insurer.}\label{bernoulli_500normalclaim}
\end{figure}

\smallskip

Finally, in Figure \ref{finalprior}, we display the distribution of the probabilities on $\lambda_{0}$. This highlights the fact that it is very difficult to estimate it with observations through time.

\begin{figure}[h]
\centering\includegraphics[scale = 0.7]{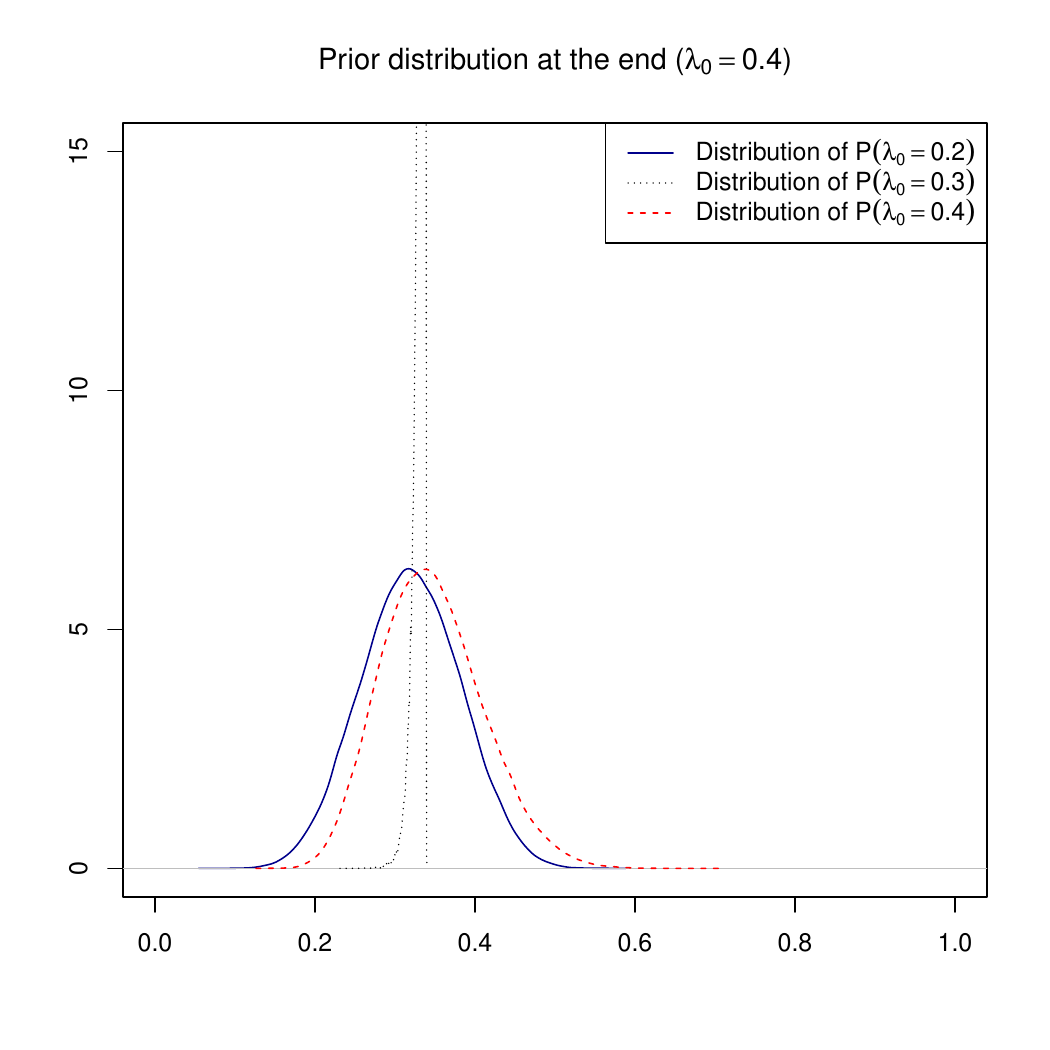}
\caption{Distribution of the probabilities on $\lambda_{0}$ at the end (with 200 000 simulations).}\label{finalprior}
\end{figure}

\clearpage

\subsection{Comparison between the Gamma case and the Bernoulli case}

Recall that in the Gamma case, the intensity is $\Lambda(\lambda, t) = \lambda h(t)$ where $h$ is 1-periodic, whereas the intensity of the Bernoulli case has the form $\Lambda(\lambda, t) = \frac{1}{2}h(t)\left[1 + \frac{E(t)}{T}\lambda\right]$ and takes into account an unknown factor due to the global warming.

Between the Figures \ref{hugeclaim} and \ref{bernoulli_2500normalclaim}, we can notice that the issues of both CAT bonds (top left figure) are delayed during one period, which is not the case in the Bernoulli case. This is due to the specific trajectories, \ref{hugeclaim} a claim occurred at $t = 5$ and hits the last layer. The two CAT bonds are resumed, but the price is high, the second one is issued later. During the whole period after, there are always two running CAT bonds, with separated renewal, while in the Bernoulli case, both are renewed later, at the beginning of the new hurricane season. The difference comes to the fact that in \ref{hugeclaim} the hurricane occurred before the end of the hurricane season, and despite the high price, the reinsurer issue a CAT bond immediately, whereas in \ref{bernoulli_2500normalclaim}, it is at the end of the season and he prefers to wait the next one, with a better price.

\medskip

The difference in the middle left figure (the Cash process) is simply due to the randomness of the trajectories, as in the middle right figure (the risk premium process, $X^2$).

\medskip

Finally, in \ref{hugeclaim}, the mean of the prior on $m$ is represented in the bottom left figure, since the true value is $\lambda_{0} = 0.6$ (which is also the true mean), it rises around it whereas the standard deviation decreases with the new information. In \ref{bernoulli_2500normalclaim}, the unknown factor is the growth of the claim rate. The bottom left and right graphics are now the probabilities of occurrence of $\{\lambda_0 = \lambda_1\}$ and $\{\lambda_0 = \lambda_3\}$ (the last one can be easily deduced). Since the true value is $\lambda_{0} = \lambda_{3}$, we see that it evolves in the correct direction. However, at the beginning it evolves very slowly, since the effect of the global warming (the growth rate) is difficult to estimate at the beginning, and easier after when it occurred. However, as we can see in \ref{finalprior}, in practice, it is very difficult to estimate to growth factor.

\section{Benefits and limits}

The framework given is quite general. In addition to CAT bonds, it could handle reinsurance treaties and to choice between them. It also tackles some lack of information about some parameters or their evolution, and the framework is adaptive. Corresponding to the expected prices on the CAT bond market for each layer, it could help in a decision process by telling which CAT bonds to issue and at which notional.

\medskip

Nonetheless, the curse of dimensionality is important, especially in the number of CAT bonds and their admissible characteristics. Adding the ability to hold one more CAT bond ask to up the dimension by its characteristics and the time-length elapsed. The uncertainty on one parameter also requires at least one dimension for the prior evolution. The most heavy example provided, from a computing time point of view, is the Bernoulli case with the Pareto Distribution discretized by 2500. On a Ryzen 7 1800X (8 cores, 16 threads, at 3.6 Ghz), with a program written in C++ completely multi-threaded, it takes around 120 hours (5 days) in order to compute the optimal control of this example. In practice, it appears complicated to use $\kappa \geq 3$ without restriction, and one should consider each peril / region separately.

\smallskip

 \bibliographystyle{plain}


\begin{thebibliography}{10}

\bibitem{artemis2017}
ARTEMIS.
\newblock Q4 2017 catastrophe bond \& ils market report.
\newblock Technical report, 2018.

\bibitem{baradel2016optimalMML}
Nicolas Baradel, Bruno Bouchard, and Ngoc~Minh Dang.
\newblock Optimal trading with online parameter revisions.
\newblock {\em Market microstructure and liquidity}, 2(03n04):1750003, 2016.

\bibitem{baradel2016optimal}
Nicolas Baradel, Bruno Bouchard, and Ngoc~Minh Dang.
\newblock Optimal control under uncertainty and bayesian parameters
  adjustments.
\newblock {\em SIAM Journal on Control and Optimization}, 56(2):1038--1057,
  2018.

\bibitem{barles1991convergence}
Guy Barles and Panagiotis~E Souganidis.
\newblock Convergence of approximation schemes for fully nonlinear second order
  equations.
\newblock {\em Asymptotic analysis}, 4(3):271--283, 1991.

\bibitem{billingsley2013convergence}
Patrick Billingsley.
\newblock Convergence of probability measures
\newblock John Wiley \& Sons, 2013.


\bibitem{guy2015}
Guy Carpenter.
\newblock Catastrophe bond update: Fourth quarter and full year 2015.
\newblock Technical report, 2016.

\bibitem{crandall1992user}
Michael~G Crandall, Hitoshi Ishii, and Pierre-Louis Lions.
\newblock User's guide to viscosity solutions of second order partial
  differential equations.
\newblock {\em Bulletin of the American Mathematical Society}, 27(1):1--67,
  1992.

\bibitem{cummins2008cat}
J~David Cummins.
\newblock Cat bonds and other risk-linked securities: state of the market and
  recent developments.
\newblock {\em Risk Management and Insurance Review}, 11(1):23--47, 2008.

\bibitem{cummins2012cat}
J~David Cummins.
\newblock Cat bonds and other risk-linked securities: product design and
  evolution of the market.
\newblock {\em The Geneva Reports}, 39, 2012.


\bibitem{easley1988controlling}
David Easley and Nicholas M Kiefer
\newblock Controlling a stochastic process with unknown parameters.
\newblock {\em Econometrica: Journal of the Econometric Society}, 1045--1064, 1988.

\bibitem{hernandez2012adaptive}
On{\'e}simo Hern{\'a}ndez-Lerma
\newblock Adaptive Markov control processes.
\newblock {\em Springer Science \& Business Media}, 79, 2012.

\bibitem{kates2006reconstruction}
Robert~William Kates, Craig~E Colten, Shirley Laska, and Stephen~P Leatherman.
\newblock Reconstruction of new orleans after hurricane katrina: a research
  perspective.
\newblock {\em Proceedings of the national Academy of Sciences},
  103(40):14653--14660, 2006.

\bibitem{malmstadt2009florida}
Jill Malmstadt, Kelsey Scheitlin, and James Elsner.
\newblock Florida hurricanes and damage costs.
\newblock {\em southeastern geographer}, 49(2):108--131, 2009.

\bibitem{oouchi2006tropical}
Kazuyoshi Oouchi, Jun Yoshimura, Hiromasa Yoshimura, Ryo Mizuta, Shoji
  Kusunoki, and Akira Noda.
\newblock Tropical cyclone climatology in a global-warming climate as simulated
  in a 20 km-mesh global atmospheric model: Frequency and wind intensity
  analyses.
\newblock {\em Journal of the Meteorological Society of Japan. Ser. II},
  84(2):259--276, 2006.

\bibitem{parisi2000seasonality}
Francis Parisi and Robert Lund.
\newblock Seasonality and return periods of landfalling atlantic basin
  hurricanes.
\newblock {\em Australian \& New Zealand Journal of Statistics},
  42(3):271--282, 2000.

\bibitem{dimitri1996stochastic}
S.E. Shreve and D.P. Bertsekas.
\newblock {\em Stochastic optimal control: The discrete time case}.
\newblock Athena Scientific, 1996.

\bibitem{air2016}
AIR Worldwide.
\newblock The coastline at risk: 2016 update to the estimated insured value of
  u.s. coastal properties.
\newblock Technical report, 2016.

\end{thebibliography}

\end{document}